\documentclass[11pt,twoside]{article}
\usepackage[textwidth=6in,textheight=8.2in,centering]{geometry}
\usepackage{fancyhdr}
\usepackage{tikz}
\usepackage{pgfplots}
\usepackage{amsmath}
\usepackage{amssymb}
\usepackage{amsthm}
\usepackage{mathtools}
\usepackage{setspace}
\usepackage{mathrsfs}
\usepackage{graphicx}
\usepackage{subcaption}
\usepackage{accents}
\usepackage{hyperref}
\usepackage[english]{babel}
\usepackage{enumitem}
\usepackage{float}
\usepackage{comment}
\usepackage{hyperref}
\mathtoolsset{showonlyrefs}

\hypersetup{
    colorlinks=true,
    linkcolor=blue}
\usetikzlibrary{arrows,automata}

\DeclareMathOperator*{\esssup}{ess\,sup}

\DeclareMathOperator*{\essinf}{ess\,inf}
\DeclareMathOperator{\leb}{\mathrm{Leb}}

\newcommand{\dleb}{d\mathrm{Leb}}

\newtheorem{theorem}{Theorem}[section] 

\theoremstyle{definition}
\newtheorem{remark}[theorem]{Remark}
\newtheorem{definition}[theorem]{Definition}
\newtheorem{notation}[theorem]{Notation}
\newtheorem{example}[theorem]{Example}

\theoremstyle{plain}
\newtheorem{lemma}[theorem]{Lemma} 

\newtheorem{corollary}[theorem]{Corollary}

\newtheorem{step}{Step}
\newtheorem{stp}{Step}

\newcommand{\tnorm}{|\mkern-1.5mu|\mkern-1.5mu|}
\usepackage{amsmath, amssymb}

\newcommand {\Sec}[1] {Section~\ref{#1}}

\newcommand {\Step}[1] {Step~\ref{#1}}

\newcommand {\fig}[1] {Figure~\ref{#1}}

\newcommand {\thrm}[1] {Theorem~\ref{#1}} 
\newcommand {\cor}[1] {Corollary~\ref{#1}} 
 
\newcommand {\dfn}[1] {Definition~\ref{#1}} 
\newcommand {\rem}[1] {Remark~\ref{#1}} 
\newcommand {\lem}[1] {Lemma~\ref{#1}}

\newcommand{\Change}[1]{\textcolor{black}{#1}}
\newcommand{\jp}[1]{\textcolor{black}{#1}}
\newcommand{\jpt}[1]{\textcolor{black}{#1}}

\usepackage{dsfont}
\DeclareMathOperator*{\BV}{BV}
\newcommand{\norm}[1]{\left\lVert #1 \right\rVert}

\DeclareMathOperator*{\var}{var}
\newcommand\restr[2]{{
  \left.\kern-\nulldelimiterspace 
  #1 
  \vphantom{\big|} 
  \right|_{#2} 
  }}
\pgfplotsset{compat=1.18}
\usepackage[multiple]{footmisc}

\begin{document}

\title{\textbf{JUMPING FOR DIFFUSION IN RANDOM METASTABLE SYSTEMS}}

\author{Cecilia Gonz\'alez-Tokman\thanks{School of Mathematics and Physics, University of Queensland, St Lucia QLD 4072, Australia. \\
\texttt{cecilia.gt@uq.edu.au}}, Joshua Peters\thanks{School of Mathematics and Physics, University of Queensland, St Lucia QLD 4072, Australia. \\
\texttt{joshua.peters@uq.edu.au}}}

\maketitle

\begin{abstract}\noindent
Random metastability occurs when an externally forced or noisy system possesses more than one state of apparent equilibrium. This work investigates fluctuations in a class of random dynamical systems, arising from randomly perturbing a piecewise smooth expanding interval map with more than one invariant subinterval. Upon perturbation, this invariance is destroyed, allowing trajectories to switch between subintervals, giving rise to metastable behaviour. We show that the distributions of jumps of a time-homogeneous Markov chain approximate the distributions of jumps for random metastable systems. Additionally, we demonstrate that this approximation extends to the diffusion coefficient for (random) observables of such systems. As an example, our results are applied to Horan's random paired tent maps.    

\end{abstract}

\newpage
\begingroup
\hypersetup{linkcolor=[rgb]{0.0,0.0,0.0}}
\tableofcontents
\endgroup
\newpage
\noindent

\section{Introduction}	
Metastability describes systems that exhibit multiple states of apparent equilibrium. Such systems arise in numerous examples of natural phenomena. In molecular dynamics, transitions between different conformations are rare events, enabling the macroscopic dynamical behaviour to be modelled as a flipping process between metastable states \cite{conformations2,conformations}. In the context of oceanic flows, metastable systems have been used to analyse slow mixing regions of the ocean, known as gyres, which have contributed to phenomena such as the \textit{Great Pacific Garbage Patch} \cite{DFH,FPE,FSM,FSN,FSS}. Random metastable systems emerge when these transition patterns or ocean currents are influenced by external forces, such as variations in chemical potentials or wind patterns, respectively.  
\\
\\
Through a dynamical systems approach, Keller and Liverani pioneered in \cite{KL_escape} the study of metastability. Such systems may emerge by adding a small perturbation to a map $T:I\to I$ on a state space $I$, possessing more than two
ergodic invariant measures supported on their respective invariant subintervals. The perturbations are made in such a way that a \textit{hole} emerges, allowing the initially invariant sets to communicate, making the system ergodic on the entire space. For piecewise expanding maps of the interval with metastable states, \cite{GTHW_metastable} provides an approximation of the invariant density in terms of the system's induced finite state Markov chain for small hole sizes. In \cite{SD}, Dolgopyat and Wright show that this Markov chain approximation extends to the diffusion coefficient (or variance) for the Central Limit Theorem (CLT), admitted by \cite{Liv_CLTDet}. We refer the reader to \cite{wael_intermittent,gibbs,open,hitting} where the authors investigate various other questions on metastability in related settings.   
\\
\\
In random dynamical systems we consider a probability space $(\Omega,\mathcal{F},\mathbb{P})$ and a family of maps $(T_\omega)_{\omega\in\Omega}$ acting on a state space $I$. The dynamics is driven by $\sigma:\Omega \to \Omega$ to form a cocycle where for $n\in\mathbb{N}$, iterates are given by $T_\omega^{(n)}:=T_{\sigma^{n-1}\omega}\circ T_{\sigma^{n-2}\omega}\circ \cdots \circ T_{\omega}$. Typically, minimal assumptions are made on the driving $\sigma:\Omega\to \Omega$, namely, invertible, ergodic, and measure-preserving, to retain the most generality for applications. In the random setting, metastability is characterised by the second Lyapunov exponent in the Perron-Frobenius operator's so-called Oseledets decomposition, discovered by Oseledets in \cite{Oseledets} and later developed by Froyland, Lloyd, and Quas in \cite{FLQ_coherent,FLQ_semi} for Perron-Frobenius operator cocycles. Since the second Lyapunov exponent is an asymptotic quantity, constructing random metastable systems proves to be a difficult task. In \cite{horan_lin} and \cite{Horan}, Horan provides an upper estimate on the second Lyapunov exponent for so-called random paired tent maps. This estimate is refined in \cite{GTQ_quarantine} where it is shown that the top two Lyapunov exponents are both simple and the only exceptional exponents outside a readily computed range. Recently, in \cite{gtp_met}, we generalise the results of \cite{GTHW_metastable} and \cite{BS_rand}, providing a quenched approximation of the functions spanning the leading and second Oseledets spaces of Perron-Frobenius operator cocycles of random metastable systems subject to small perturbations. Further, we relate such results to the system's averaged finite state Markov chain. Results in related settings have been investigated in \cite{BS_rand,Crimmins,GS_entropy,GTQ_stab}, we refer the reader to \cite{gtp_met} for further details.       
\\
\\
Limit theorems in the setting of random dynamical systems provide fundamental insights into the statistical behaviour of compositions of random transformations. When the driving $\sigma:\Omega\to \Omega$ is a Bernoulli shift, the sequence of maps $(T_\omega)_{\omega\in\Omega}$ becomes an i.i.d. process, and in such cases, the focus is typically on annealed limit laws. In this setting, one studies iterates of the annealed Perron-Frobenius operator, which captures averaged statistical properties across realisations of $\omega\in\Omega$ \cite{RDS_A,Annealed_O}. We refer the reader to \cite{MR_Limthrm,LIMTHRM_ANV,quenchedLT_DS} and the references therein for results in this direction. On the other hand, quenched (or fibrewise) limit theorems describe the statistical behaviour for individual realisations. In this setting, the quenched limit theorems in \cite{DFGTV_limthrm} apply to a large class of observables that may depend on $\omega\in\Omega$. The family of observables must satisfy a fibrewise centering condition. A discussion regarding why this assumption is necessary in the quenched setting may be found in \cite{MR_Limthrm}. Unlike annealed results, the quenched CLT accounts for fluctuations about a time-dependent mean. For instance, if the observable represents the number of flowers on a university campus, the quenched CLT describes fluctuations about a mean that varies with the seasons. The diffusion coefficient (or variance) given by the CLT characterises the spread of data relative to this time-dependent mean. Further details and results on quenched limit theorems can be found in \cite{MR_Limthrm,quenchedLT_DS} and related references.  

\subsection{Statement of main results}
In this paper, we consider a Perron-Frobenius operator cocycle of so-called random metastable systems driven by an ergodic, invertible and measure-preserving transformation $\sigma$ of the probability space $(\Omega,\mathcal{F},\mathbb{P})$. The initial system $T^0$ is deterministic and admits $m\geq 2$ ergodic absolutely continuous invariant measures ($\mu_1,\dots,\mu_m$) with associated invariant densities ($\phi_1,\dots,\phi_m$) supported on $m$ disjoint invariant subintervals ($I_1,\dots,I_m$). We consider small random perturbations of $T^0$, denoted $T_\omega^\varepsilon$, which introduce random holes into our system (defined as $H_{i,j,\omega}^\varepsilon:= I_i \cap (T_\omega^\varepsilon)^{-1}(I_j)$ for $i,j\in\{1,\cdots,m\}$ and $\omega\in\Omega$), allowing trajectories to randomly switch between the initially invariant subintervals (see \fig{fig:ptm}), giving rise to a unique ergodic absolutely continuous random invariant measure on $I_1\cup\cdots\cup I_m$, denoted $(\mu_\omega^\varepsilon)_{\omega\in\Omega}$. The perturbations are made so that for $i,j\in\{1,\cdots,m\}$ there exists $\beta^*>0$ such that $\mu_i(H_{i,j,\omega}^\varepsilon)=\varepsilon \beta_{i,j,\omega}+o_{\varepsilon\to 0}(\varepsilon)$ where $\beta_{i,j} \in  L^\infty(\mathbb{P})$ satisfies $\beta_{i,j,\omega}\geq \beta^*$ for all $\omega\in\Omega$. \jp{For all $\varepsilon\geq 0$ we assume that the mapping $\omega\mapsto T_{\omega}^\varepsilon$ has finite range, with $\mathcal{L}_\omega^\varepsilon$ (the Perron-Frobenius operator associated with $T_\omega^\varepsilon$) acting on $(\BV(I),\|\cdot\|_{\BV(I)})$ enjoying good spectral properties. For further details on the class of maps for which our results apply, we refer the reader to \Sec{sec:map+pert} where assumptions \hyperref[list:I1]{\textbf{(I1)}}-\hyperref[list:I6]{\textbf{(I6)}} and \hyperref[list:P1]{\textbf{(P1)}}-\hyperref[list:P7]{\textbf{(P7)}} are defined, and to \Sec{sec:ptm} for a class of examples.} For presentation purposes, we assume that under one iteration of $T_\omega^\varepsilon$, for some $j\in\{1,\cdots,m\}$, points in $I_j$ can only map to neighbouring sets $I_{j-1},I_{j+1}$ or remain in $I_j$.\footnote{Here we take $I_0=I_{m+1}=\emptyset$.} We note that this assumption can be relaxed without difficulty.  
\\ \\ 
Our approach follows similar arguments to \cite{SD} and utilises the recent sequential perturbation results of \cite{thermoformalism} applied to an open system derived from the sequence of random metastable systems. For $j\in\{1,\cdots,m\}$ consider $H_{j,\omega}^\varepsilon:= H_{j,j-1,\omega}^\varepsilon \cup H_{j,j+1,\omega}^\varepsilon$ as a hole. For $f\in\BV(I_j)$ we let 
\begin{equation}
    \mathcal{L}_{j,\omega}^\varepsilon (f)(x):= \mathcal{L}_\omega^\varepsilon(\mathds{1}_{I_j\setminus H_{j,\omega}^\varepsilon}\cdot f)(x)\label{eqn:open-operator}
\end{equation}
be the Perron-Frobenius operator acting on $(\BV(I_j),\|\cdot\|_{\BV(I_j)})$. For $n\in\mathbb{N}$, the evolution of $f\in\BV(I_j)$ under $(\mathcal{L}_{j,\omega}^\varepsilon)_{\omega\in\Omega}$ is given by $\mathcal{L}_{j,\omega}^{\varepsilon\,(n)}:=\mathcal{L}_{j,\sigma^{n-1}\omega}^\varepsilon\circ\cdots\circ\mathcal{L}_{j,\sigma\omega}^\varepsilon\circ\mathcal{L}_{j,\omega}^\varepsilon$. \jp{We assume that the open system satisfies \hyperref[list:O1]{\textbf{(O1)}}, ensuring $\mathcal{L}_{j,\omega}^\varepsilon$ satisfies a uniform Lasota-Yorke inequality. We refer the reader to \Sec{sec:open} for further details.}\\
\\
Let {$\leb_j(f):= {\leb}( f\cdot\mathds{1}_{I_j})$} denote the Lebesgue measure on $I_j$. Our first result, \thrm{thrm:openopprop}, describes the spectral properties of $\mathcal{L}_{j,\omega}^\varepsilon$ for small $\varepsilon>0$.
\begin{theorem}
Let $\{(\Omega,\mathcal{F},\mathbb{P},\sigma,\BV(I),\mathcal{L}^\varepsilon)\}_{\varepsilon\geq 0}$ be a sequence of random dynamical systems of metastable maps $T_\omega^\varepsilon:I\to I$ satisfying \hyperref[list:I1]{\textbf{(I1)}}-\hyperref[list:I6]{\textbf{(I6)}} and \hyperref[list:P1]{\textbf{(P1)}}-\hyperref[list:P7]{\textbf{(P7)}}. For each $j\in \{1,\cdots,m\}$, if $\mathcal{L}_{j,\omega}^\varepsilon:\BV(I_j)\to \BV(I_j)$ is given by \eqref{eqn:open-operator} and satisfies \hyperref[list:O1]{\textbf{(O1)}}, then for $\varepsilon\geq 0$ sufficiently small and $\mathbb{P}$-a.e. $\omega\in\Omega$ there is a functional $\nu_{j,\omega}^\varepsilon\in \BV^*(I_j)$ that can be identified with a real finite Borel measure on $I_j$, $\lambda_{j,\omega}^\varepsilon\in [0,1]$, and $\phi_{j,\omega}^\varepsilon\in \BV(I_j)$ such that for all $f\in\BV(I_j)$ 
        $$\mathcal{L}_{j,\omega}^\varepsilon (\phi_{j,\omega}^\varepsilon)=\lambda_{j,\omega}^\varepsilon\phi_{j,\sigma\omega}^\varepsilon \quad \mathrm{and} \quad \nu_{j,\sigma\omega}^\varepsilon(\mathcal{L}_{j,\omega}^\varepsilon (f))=\lambda_{j,\omega}^\varepsilon\nu_{j,\omega}^\varepsilon(f).$$
Further, 
\begin{itemize}
    \item[(a)] For $\mathbb{P}$-a.e. $\omega\in\Omega$, $$\lambda_{j,\omega}^\varepsilon=1-\varepsilon(\beta_{j,j-1,\omega}+\beta_{j,j+1,\omega}) +o_{\varepsilon\to 0}(\varepsilon);$$ 
    \item[(b)] $$\lim_{\varepsilon\to 0} \esssup_{\omega\in\Omega} \|\nu_{j,\omega}^\varepsilon - \leb_j \|_{\BV^*(I_j)}=0;$$
    \item[(c)]  $$\lim_{\varepsilon\to 0} \esssup_{\omega\in\Omega} \|\phi_{j,\omega}^\varepsilon - \phi_j\|_{L^1(\leb_j)} = 0;$$
    \item[(d)] For any $k\in\{1,\cdots,m\}$, $$\lim_{\varepsilon\to 0}\esssup_{\omega\in\Omega}\sup_{x\in H_{j,k,\omega}^\varepsilon}|\phi_{j,\omega}^\varepsilon(x) - \phi_j(x)|=0;$$
\end{itemize}
and,
\begin{itemize}
    \item[(e)] There exist constants $C>0$ and $\theta\in(0,1)$ such that for all $\varepsilon>0$ sufficiently small, $f\in\BV(I_j)$, $n\in\mathbb{N}$, and $\mathbb{P}$-a.e. $\omega\in\Omega$
    $$\|(\lambda_{j,\omega}^{\varepsilon\, (n)})^{-1}\mathcal{L}_{j,\omega}^{\varepsilon\, (n)}(f)-\nu_{j,\omega}^\varepsilon(f)\phi_{j,\sigma^n\omega}^\varepsilon\|_{\BV(I_j)}\leq C\theta^n\|f\|_{\BV(I_j)},$$
    where $\lambda_{j,\omega}^{\varepsilon\,(n)}:=\lambda_{j,\sigma^{n-1}\omega}^\varepsilon\cdots\lambda_{j,\sigma\omega}^\varepsilon\lambda_{j,\omega}^\varepsilon$.
\end{itemize}
\label{thrm:openopprop}
\end{theorem}

Our second contribution, \thrm{thrm:conv_jump}, extends the results of \cite{SD} and \cite{gtp_met}, making the connection between random metastable systems and their associated averaged finite state Markov chain more precise. In particular, it shows that the distribution of jumps of an averaged Markov jump process approximates the distribution of jumps for random metastable systems.
\\
\\
{Consider a continuous time stochastic process $(X_t)_{t \geq 0}\subset \{1,\cdots, m\}^{[0,\infty)}$, whose evolution is governed by $(P(t))_{t\geq 0}:=(e^{t \bar{G}})_{t\geq 0}$, where $\bar{G}\in M_{m\times m}(\mathbb{R})$ has entries $(\bar{G})_{ij}=\int_\Omega \beta_{i,j,\omega}\, d\mathbb{P}(\omega)\jp{=:\bar{\beta}_{i,j}}$ for $i\neq j$, and $(\bar{G})_{ii}=-\sum_{j\neq i}(\bar{G})_{ij}$. Set $t_0^M:=0$, and for $i>0$, let $t_i^M = \inf\{t> t_{i-1}^M\ | \ X_t\neq X_{t_{i-1}^M}\}$. For $i\geq 1$, set $\mathcal{T}_i^M = t_i^M - t_{i-1}^M$. Let $z_i^M$ denote the state of the process following the $i^{\mathrm{th}}$ transition, that is, $z_i^M:=X_{t_i^M}$. For $j\in\{1,\cdots, m\}$, let $\mathbb{P}^j$ denote the probability measure constructed on $\{1,\cdots, m\}^{[0,\infty)}$ with the initial condition $z_0^M=j$ that is {evolved by $P(t)$}}. \\
\\
For the collection of random maps $(T_\omega^\varepsilon)_{\omega\in\Omega}$, let $t_{0,\omega}^\varepsilon(x):=0$ for all $\omega\in\Omega, x\in I$ and $\varepsilon>0$. Define $z:I\to \{1,\cdots,m\}$ such that $z(x)=j$ if $x\in I_j$. For $i> 0$ we let $t_{i,\omega}^\varepsilon(x):=\inf\{n>t_{i-1,\omega}^\varepsilon(x) \ | \ z(T_\omega^{\varepsilon\, (n)}(x))\neq z(T_\omega^{\varepsilon\, (t_{i-1,\omega}^\varepsilon(x))}(x))\}$. Finally, for $i\geq 1$ we define $\mathcal{T}_{i,\omega}^\varepsilon(x):=t_{i,\omega}^\varepsilon(x) - t_{i-1,\omega}^\varepsilon(x)$. We refer the reader to \Sec{sec:jump} where the above is described in further depth. 

\begin{theorem}
In the setting of \thrm{thrm:openopprop}, fix $j\in \{1,\cdots,m\}$ and $p\in \mathbb{N}$. For $k=1,\dots, p$, take a sequence of intervals $\Delta_k = [a_k,b_k]$ and numbers $r_k\in\{1,\cdots,m\}$, then for $\mathbb{P}$-a.e. $\omega\in\Omega$
\begin{align*}
&\lim_{\varepsilon\to 0}\mu_j\left(\left\{x\in I \ \big| \ \varepsilon\mathcal{T}_{k,\omega}^\varepsilon(x)\in \Delta_k \ \mathrm{and} \ z(T_{\omega}^{\varepsilon \, (t_{k,\omega}^\varepsilon(x))}(x))=r_k \ \mathrm{for} \ k=1,\dots, p\right\}  \right) \\
&\qquad  = \mathbb{P}^j(\mathcal{T}_k^M\in \Delta_k\ \mathrm{and} \ z_k^M=r_k \ \mathrm{for}\, k=1,\dots,p),
\end{align*}
where $\mu_j$ is as in \hyperref[list:I4]{\textbf{(I4)}}. 
\label{thrm:conv_jump}
\end{theorem}
In addition to \thrm{thrm:conv_jump}, we show that in our setting, random metastable systems satisfy the quenched CLT of \cite{DFGTV_limthrm} for a large class of fibrewise centered random observables. Further, we extend on the results of \cite{SD} by providing an approximation of the diffusion coefficient (or variance) in terms of an averaged Markov jump process when $\varepsilon>0$ is small. 
\\
\\
Suppose that $\phi_\omega^0:= \sum_{j=1}^m p_j\phi_j$ is the limiting invariant density, $\lim_{\varepsilon\to 0}d\mu_\omega^\varepsilon/\dleb$, admitted by \cite[Theorem 7.2]{gtp_met}. For a measurable observable $\psi:\Omega\times I \to \mathbb{R}$ and each $\varepsilon>0$ let
\begin{equation*}
\tilde{\psi}_\omega^\varepsilon:=\psi_\omega - \mu_{\omega}^\varepsilon(\psi_\omega).    
\end{equation*}
For $j\in\{1,\cdots, m\}$, set $\Psi_\omega(j):= \int_{I_j}\psi_\omega(x) \phi_j(x)\, \dleb(x)$. 
\begin{theorem}
    Fix $\varepsilon>0$. In the setting of \thrm{thrm:openopprop}, assume that the observable $\psi:\Omega\times I \to \mathbb{R}$ satisfies \jp{regularity and centering conditions} \eqref{eqn:obs-reg} and \eqref{eqn:obs-centering2}. Assume that the variance defined in \eqref{eqn:var} satisfies $(\Sigma^\varepsilon(\tilde{\psi}^\varepsilon))^2>0$. Then, for every bounded and continuous function $f:\mathbb{R}\to \mathbb{R}$ and $\mathbb{P}$-a.e. $\omega\in\Omega$, we have 
    \begin{equation}
        \lim_{n\to \infty} \int_{\mathbb{R}}f\left( \frac{1}{\sqrt{n}}\sum_{k=0}^{n-1}(\tilde{\psi}^\varepsilon_{\sigma^{k}\omega}\circ T_{\omega}^{\varepsilon\, (k)})\right)(x)\, d\mu_{\omega}^\varepsilon(x) = \int_\mathbb{R}f(x)\, d\mathcal{N}(0,(\Sigma^\varepsilon(\tilde{\psi}^\varepsilon))^2)(x),\label{eqn:clt-limit}
    \end{equation}
    where $\mathcal{N}(0,(\Sigma^\varepsilon(\tilde{\psi}^\varepsilon))^2)$ is the normal distribution with mean $0$ and variance $(\Sigma^\varepsilon(\tilde{\psi}^\varepsilon))^2$.
    Furthermore, if $p:=\begin{pmatrix} p_1 &\cdots & p_m \end{pmatrix}^T$ and $\Psi_\omega:=\begin{pmatrix} \Psi_\omega(1) &\cdots & \Psi_\omega(m) \end{pmatrix}^T$, then
    \begin{equation}
        \lim_{\varepsilon\to 0} \varepsilon(\Sigma^\varepsilon(\tilde{\psi}^\varepsilon))^2 = 2\left\langle p\odot\int_\Omega \Psi_\omega\, d\mathbb{P}(\omega),\int_0^\infty e^{t\bar{G}}\left(\int_\Omega \Psi_\omega\, d\mathbb{P}(\omega)\right) \, dt
 \right\rangle,\label{eqn:var-limit}
    \end{equation}
where $\odot$ denotes the Hadamard product, $\bar{G}$ is the generator for the averaged Markov jump process defined in \Sec{sec:avg_markov}, and $\langle\cdot,\cdot \rangle$ denotes the standard dot product.
 \label{thrm:variance-est}
\end{theorem}
\jpt{
\begin{remark}
Note that the integral over $t$ in \eqref{eqn:var-limit} converges due to the centering condition \eqref{eqn:obs-centering2} which ensures that $e^{t\bar{G}}\left(\int_\Omega \Psi_\omega\, d\mathbb{P}(\omega)\right)$ decays exponentially fast to zero as $t\to\infty$. \label{rmk:defined_int} \end{remark}
}
The paper is structured in the following manner. In \Sec{sec:prelim} we introduce relevant definitions and results related to random dynamical systems and Perron-Frobenius operators. The initial system and the perturbations made to it are introduced in \Sec{sec:map+pert}. Here we discuss how this gives rise to so-called random metastable systems. In \Sec{sec:open} we study the spectral properties of the Perron-Frobenius operators associated with a sequence of random open dynamical systems derived from the sequence of random metastable systems and prove \thrm{thrm:openopprop}. This allows us to prove \thrm{thrm:conv_jump} in \Sec{sec:jump}, providing an approximation for the distribution of jumps for random metastable systems through the distribution of jumps of an averaged Markov jump process when $\varepsilon>0$ is small. In \Sec{sec:var}, we utilise \thrm{thrm:conv_jump} to establish \thrm{thrm:variance-est}, connecting the diffusion coefficient for random metastable systems with an averaged Markov jump process. Finally, in \Sec{sec:ptm}, we apply our results to Horan's random paired tent maps.

\section{Preliminaries}
In this section, we collate definitions and results relevant to this paper. Primarily, we introduce transfer operator techniques for random dynamical systems.   
\begin{definition}
\sloppy A \textit{semi-invertible random dynamical system} is a tuple $(\Omega,\mathcal{F},\mathbb{P},\sigma,X,\mathcal{L})$, where the base $\sigma:\Omega\to \Omega$ is an \textit{invertible},\footnote{$\sigma^{-1}$ is measurable and exists for $\mathbb{P}$-a.e. $\omega\in \Omega$.} measure-preserving transformation of the probability space $(\Omega,\mathcal{F},\mathbb{P})$, $(X,\| \cdot \|)$ is a Banach space, and $\mathcal{L}:\Omega\to L(X)$ is a family of bounded linear operators of $X$, called the generator.\footnote{Here $L(X)$ denotes the space of bounded linear operators preserving the Banach space $X$.} 
\label{def:seminvertrds}
\end{definition}
\begin{remark}
    For convenience, whenever we refer to a random dynamical system, we assume it is semi-invertible. That is, the base map $\sigma:\Omega\to \Omega$ is invertible, but its generators need not be. 
\end{remark}
\noindent
In general, $X$ can be any given Banach space. In this paper, we will be interested in $X$ being the Banach space of functions with \textit{bounded variation}.
\begin{definition}
Let $(S,\mathcal{D},\mu)$ be a measure space with $S=[a,b]$. The space $\BV_\mu(S)$ is called the space of \textit{bounded variation} on $S$ where
$$\|f\|_{\BV_\mu(S)} = \inf_{\tilde{f} = f\, \mu-\mathrm{a.e.}}\var(\tilde{f})+\|\tilde{f}\|_{L^1(\mu)}$$
and $\var(f)$ is the total variation of $f$ over $S$;
    $$\var(f) = \sup\left\{ \sum_{i=1}^n |f(x_i)-f(x_{i-1})| \ | \ n\geq 1, \ a\leq x_0 <x_1<\cdots < x_n\leq b \right\}.$$
    \label{dfn:BV}
\end{definition}\noindent
Elements of $\BV_\mu(S)$, denoted $[f]_\mu$, are equivalence classes of functions with bounded variation on $S$. In this paper, we consider functions $f\in \BV(S)$ with norm $\|f\|_{\BV(S)}=\var(f)+\|f\|_{L^1(\mu)}$ and emphasise that $f$ is identified through a representative of minimal variation from the equivalence class $[f]_\mu$. 
\\
\\
\noindent
We associate with a \textit{non-singular transformation} $T$ a unique \textit{Perron-Frobenius operator}. Since it describes the evolution of ensembles of points, or densities, such operators serve as a powerful tool when studying the statistical behaviour of trajectories of $T$. 
\begin{definition}
    A measurable function $T:I\to I$ on a measure space $(I,\mathcal{D},\mu)$ is a \textit{non-singular transformation} if $\mu(T^{-1}(D))=0$ for all $D\in \mathcal{D}$ such that $\mu(D)=0$.
    \label{dfn:non-sing}
\end{definition}

\begin{definition}
\label{def:pfoperator}
     Let $(I,\mathcal{D},\mu)$ be a measure space and $T:I\to I$ be a non-singular transformation. \jp{The unique operator $\mathcal{L}_T:L^1(\mu)\to L^1(\mu)$ satisfying
    $$\int_D\mathcal{L}_T(f)(x)\, d\mu(x) = \int_{T^{-1}(D)}f(x)\, d\mu(x)$$
    for all $f\in L^1(\mu)$ and $D\in\mathcal{D}$ is called the \textit{Perron-Frobenius operator} (associated with $T$).}
\end{definition}\noindent
In certain cases, $\mathcal{L}_T$ may be restricted (or extended) to a bounded linear operator on another Banach space (e.g. $X=\BV(I)$), in which case the operators are still referred to as Perron-Frobenius operators. Combining \dfn{def:seminvertrds} and \dfn{def:pfoperator}, one can construct a random dynamical system from a family of Perron-Frobenius operators $(\mathcal{L}_{T_\omega})_{\omega\in\Omega}$ associated with the set of non-singular transformations $(T_\omega)_{\omega\in\Omega}.$ This forms a Perron-Frobenius operator cocycle.
\begin{remark}
\sloppy    For notational purposes, we denote by $(\mathcal{L}_\omega)_{\omega\in\Omega}$ the family of Perron-Frobenius operators associated with the family of non-singular transformations $(T_\omega)_{\omega\in\Omega}$.
\end{remark}
\begin{example}[Perron-Frobenius operator cocycle]
Consider a random dynamical system $(\Omega,\mathcal{F},\mathbb{P},\sigma,X,\mathcal{L})$, where $\sigma:\Omega\to \Omega$ is an invertible, ergodic and measure-preserving transformation, and its generators $\mathcal{L}:\Omega \to L(X)$ are Perron-Frobenius operators associated with the non-singular transformations $T_\omega$ of the measure space $(I,\mathcal{D},\mu)$, given by $\omega \mapsto \mathcal{L}_{\omega}$. This gives rise to a Perron-Frobenius operator cocycle
$$(n,\omega)\mapsto \mathcal{L}_{\omega}^{(n)}=\mathcal{L}_{\sigma^{n-1}\omega}\circ \cdots \circ \mathcal{L}_{\omega}.$$
Here, the evolution of a density $f$ is governed by a cocycle of Perron-Frobenius operators driven by the base dynamics $\sigma:\Omega\to \Omega$. That is, if $f$ represents the initial mass distribution in the system, then $\mathcal{L}_\omega^{(n)}$ describes the mass distribution after the application of $T_{\sigma^{n-1}\omega}\circ \cdots \circ T_\omega$. 
\end{example}
\noindent
The \textit{random fixed points} of the Perron-Frobenius operator are of interest in random dynamical systems. 
\begin{definition}
Let $(\Omega,\mathcal{F},\mathbb{P},\sigma,X,\mathcal{L})$ be a random dynamical system, with associated non-singular transformations $T_\omega:I\to I$. A family $(\mu_\omega)_{\omega\in\Omega}$ is called a \textit{random invariant measure} for $(T_\omega)_{\omega\in \Omega}$ if $\mu_\omega$ is a probability measure on $I$, for any Borel measurable subset $D$ of $I$ the map $\omega \mapsto \mu_\omega(D)$ is measurable, and 
$$\mu_\omega(T^{-1}_\omega(D)) = \mu_{\sigma\omega}(D)$$
for $\mathbb{P}$-a.e. $\omega\in \Omega$. A family $(h_\omega)_{\omega\in\Omega}$ is called a \textit{random invariant density} for $(T_\omega)_{\omega\in \Omega}$ if $h_\omega\geq 0$, $h_\omega\in L^1(\mu)$, $\|h_\omega\|_{L^1(\mu)}=1$, the map $\omega\mapsto h_\omega$ is measurable, and 
$$\mathcal{L}_\omega h_\omega = h_{\sigma\omega}$$
for $\mathbb{P}$-a.e. $\omega\in \Omega$. We also say $(h_\omega)_{\omega\in\Omega}$ is a \textit{random fixed point} of $(\mathcal{L}_\omega)_{\omega\in\Omega}$.
\label{def:RIM-RID}
\end{definition}
\begin{remark}
    If the random invariant measure $(\mu_\omega)_{\omega\in\Omega}$ is absolutely continuous with respect to Lebesgue, we refer to it as a random absolutely continuous invariant measure (RACIM). In this case, its density $(h_\omega)_{\omega\in\Omega} = (\frac{d\mu_\omega}{\dleb})_{\omega\in\Omega}$ is a random invariant density.
    \label{rem:RACIM}
\end{remark}
\jp{We focus on the setting where the mapping $\omega\mapsto \mathcal{L}_\omega$ is $\mathbb{P}$\textit{-continuous}, a concept introduced by Thieullen in \cite{IOC_TP}. This gives rise to a \textit{$\mathbb{P}$-continuous random dynamical system}.}
 \begin{definition}
      Let $(\Omega,\mathcal{F},\mathbb{P})$ be a Borel probability space and $(Y, \tau)$ a topological space. A mapping $\mathcal{L}:\Omega \to Y$ is said to be \textit{$\mathbb{P}$-continuous} if $\Omega$ can be expressed as a countable union of Borel sets such that the restriction of $\mathcal{L}$ to each of them is continuous. 
      \label{def:P-cont}
 \end{definition}
\begin{definition}  
    Let $(\Omega,\mathcal{F},\mathbb{P},\sigma,X,\mathcal{L})$ be a random dynamical system. If its generators $\mathcal{L}:\Omega \to L(X)$, given by $\omega\mapsto\mathcal{L}_\omega$, are $\mathbb{P}$-continuous with respect to the norm topology on $L(X)$, then the cocycle $(\mathcal{L}_\omega)_{\omega\in\Omega}$ is called $\mathbb{P}$-continuous, and the tuple $(\Omega, \mathcal{F}, \mathbb{P}, \sigma, X , \mathcal{L})$ is a \textit{$\mathbb{P}$-continuous random dynamical system}. 
\end{definition}
\noindent
The asymptotic behaviour of the spectral picture for the Perron-Frobenius operator cocycle is of great interest when studying the statistical properties of random systems. 
\begin{definition}
The \textit{index of compactness} of an operator $\mathcal{L}_\omega$ denoted $\alpha(\mathcal{L}_\omega)$, is the infimum of those real numbers $t$ such that the image of the unit ball in $X$ under $\mathcal{L}_\omega$ may be covered by finitely many balls of radius $t$.
\label{def:IndexOfCompactness}
\end{definition}\noindent
The index of compactness provides a notion of `how far' an operator is from being compact. This definition was extended by Thieullen to random compositions of operators in \cite{IOC_TP}.
\begin{definition}
The \textit{asymptotic index of compactness} for the cocycle $(\mathcal{L}_\omega)_{\omega\in\Omega}$ on $X$ is
$$\kappa(\omega):=\lim_{n\to \infty} \frac{1}{n}\log \alpha(\mathcal{L}_\omega^{(n)}).$$
\label{def:AssymptoticIndexOfCompactness}
\end{definition}\noindent
We call the cocycle \textit{quasi-compact} if $\kappa(\omega)<\lim_{n\to \infty} \frac{1}{n}\log\|\mathcal{L}_\omega^{(n)}\|=:l_1(\omega)$, whose limit exists for $\mathbb{P}$-a.e. $\omega\in\Omega$, {and is independent of $\omega\in\Omega$}, by the Kingman sub-additive ergodic theorem {\cite{Kingman}}, under the assumption that $\int \log\|\mathcal{L}_\omega\|\, d\mathbb{P}(\omega)<\infty$. The limit $l_1(\omega)$ is referred to as the top Lyapunov exponent of the cocycle, and under some assumptions on the random dynamical system, one can obtain a spectrum of these exponents through {multiplicative ergodic theorems}. One example is \textit{Oseledets decomposition}, which splits $X$ into $\omega$-dependent subspaces which decay/expand according to its associated Lyapunov exponent $l_i(\omega)$. These are constant $\mathbb{P}$-a.e. when $\sigma$ is ergodic. 
\begin{definition}
    Consider a random dynamical system $\mathcal{R}=(\Omega,\mathcal{F},\mathbb{P},\sigma,X,\mathcal{L})$. An \textit{Oseledets splitting} for $\mathcal{R}$ consists of isolated (exceptional) Lyapunov exponents 
    $$\infty>l_1>l_2>\cdots>l_c>\kappa\geq -\infty,$$
    where the index $c\geq 1$ is allowed to be finite or countably infinite, and Oseledets subspaces $V_1(\omega),\dots, V_c(\omega), W(\omega)$ such that for all $i=1,\dots, c$ and $\mathbb{P}$-a.e. $\omega\in\Omega$ we have
   \begin{itemize}
    \item[(a)] $\dim(V_i(\omega))=m_i<\infty$;
    \item[(b)] $\mathcal{L}_\omega V_i(\omega)=V_i(\sigma\omega)$ and $\mathcal{L}_\omega W(\omega)\subseteq W(\sigma\omega)$;
    \item[(c)] $V_1(\omega)\oplus\cdots \oplus V_c(\omega)\oplus W(\omega)=X$;
    \item[(d)] for $f\in V_i(\omega)\setminus\{0\}$, $\lim_{n\to\infty}\frac{1}{n}\log\|\mathcal{L}_\omega^{(n)}f\|\to l_i$;
    \item[(e)] for $f\in W(\omega)\setminus\{0\}$, $\lim_{n\to\infty}\frac{1}{n}\log\|\mathcal{L}_\omega^{(n)}f\|\leq \kappa$.
\end{itemize}
\label{def:Osc-splitting}
\end{definition}
The Lyapunov exponents of $\mathcal{R}$ counted with multiplicities is the sequence $$\overbrace{l_1,\dots, l_1}^{m_1},\overbrace{l_2,\cdots,l_2}^{m_2},l_3,\dots l_{c}.$$ 
\jpt{\begin{remark}
  For $i\in\{1,\cdots , \sum_{i=1}^c m_i\}$, we set $\gamma_i$ to be the $i^{\mathrm{th}}$ element of this sequence (from left to right). \label{rmk:gamma}  
\end{remark}}
\begin{theorem}[{\cite[Theorem 17]{FLQ_semi}}]
Let $\Omega$ be a Borel subset of a separable complete metric space, $\mathcal{F}$ the Borel $\sigma$-algebra and $\mathbb{P}$ a Borel probability. Let $X$ be a Banach space and consider a random dynamical system $\mathcal{R}=(\Omega,\mathcal{F},\mathbb{P},\sigma,X,\mathcal{L})$ with base transformation $\sigma:\Omega\to \Omega$ an ergodic homeomorphism, and suppose that the generator $\mathcal{L} : \Omega \to L(X)$ is $\mathbb{P}$-continuous and satisfies
$$\int_\Omega \log^+ \|\mathcal{L}_\omega\|\, d\mathbb{P}(\omega)<\infty.$$
If $\kappa<l_1$, $\mathcal{R}$ admits a unique $\mathbb{P}$-continuous Oseledets splitting.
\label{the:OsceledetsDecomp}
\end{theorem}
We refer to the set of all $l_i$ as the Lyapunov spectrum of the Perron-Frobenius operator cocycle $(\mathcal{L}_\omega)_{\omega\in\Omega}$. We recall that in \thrm{the:OsceledetsDecomp} the Lyapunov spectrum and asymptotic index of compactness are constant for $\mathbb{P}$-a.e. $\omega\in\Omega$ since $\sigma$ is ergodic. When the Oseledets splitting of $(\mathcal{L}_\omega)_{\omega\in\Omega}$ \jp{satisfies uniform hyperbolicity conditions}, by adopting the same terminology of \cite{Crimmins}, we call this a \textit{hyperbolic Oseledets splitting}. \jp{Let $\mathbb{X}=\Omega\times X$, and define the map $\pi:\mathbb{X}\to\Omega$ as the projection onto $\Omega$. Denote by $\mathrm{End}(\mathbb{X},\sigma)$ the set of all bounded linear endomorphisms of $\mathbb{X}$ covering $\sigma$, i.e.
$$\mathrm{End}(\mathbb{X},\sigma)=\{\mathcal{L}:\mathbb{X}\to \mathbb{X} \ | \ \pi\circ \mathcal{L} = \sigma \circ \pi \ \mathrm{and} \ f\mapsto \tau_{\sigma\omega}(\mathcal{L}(\omega,f))\in L(X{})\}$$
where $\tau_\omega:\pi^{-1}(\omega)\to X$ is given by $\tau_\omega(\omega,f)=f$. Let $\mathcal{G}(X)$ denote the Grassmanian of $X$ (the set of all closed, {complemented} subspaces of $X$). For $E,F\in\mathcal{G}(X)$ with $E\oplus F=X$, let $\Pi_{E||F}$ be the projection onto $E$ parallel to $F$ so that the image of $\Pi_{E||F}$ is $E$ and $\mathrm{ker}(\Pi_{E||F})=F$. For a fixed $d\in\mathbb{Z}^+$ we let $\mathcal{G}_d(X)$ and $\mathcal{G}^d(X)$ denote the set of closed $d$-dimensional and $d$-codimensional subspaces of $X$ respectively.
\begin{definition} 
    Suppose that $\mathcal{L}\in \mathrm{End}(\mathbb{X},\sigma)$, $d\in\mathbb{Z}^+$, $0\leq \mu<\lambda$, $(E_\omega)_{\omega\in\Omega}\in \mathcal{G}_d(X)^\Omega$ and $(F_\omega)_{\omega\in\Omega}\in \mathcal{G}^d(X)^\Omega$. We say that the family of subspaces $(E_\omega)_{\omega\in\Omega}$ and $(F_\omega)_{\omega\in\Omega}$ form a \textit{$(\lambda,\mu,d)$-hyperbolic splitting} for $\mathcal{L}$, and that $\mathcal{L}$ has a \textit{hyperbolic splitting} of index $d$, if there exists constants $C_{\lambda},C_\mu,\Theta>0$ such that:
    \begin{enumerate}
    \item[(a)] For $\mathbb{P}$-a.e. $\omega\in\Omega$ we have $E_\omega \oplus F_\omega = X$ and
    $$\max\{\|\Pi_{F_\omega||E_\omega} \|,\|\Pi_{E_\omega||F_\omega} \|\}\leq \Theta.$$
    \item[(b)] For $\mathbb{P}$-a.e. $\omega\in\Omega$ we have $\mathcal{L}_\omega E_\omega = E_{\sigma\omega}$. Moreover, for every $n\in\mathbb{N}$ and $f\in E_\omega$ we have 
    $$\|\mathcal{L}_\omega^{(n)}f\|\geq C_\lambda \lambda^n \|f\|.$$
    \item[(c)] For $\mathbb{P}$-a.e. $\omega\in\Omega$ we have $\mathcal{L}_\omega F_\omega \subseteq F_{\sigma\omega}$ and for every $n\in\mathbb{N}$ we have 
    $$\|\mathcal{L}_\omega^{(n)}|_{F_\omega}\|\leq C_\mu \mu^n.$$
    \end{enumerate}
    \label{def:hyperbolic-splitting}
\end{definition}}
\jp{\begin{remark}
    In the deterministic setting ($\mathcal L_\omega =\mathcal L_0$ for $\mathbb{P}$-a.e. $\omega \in \Omega$), if $\mathcal L_0$ is quasi-compact, the random dynamical system $(\Omega, \mathcal{F}, \mathbb{P}, \sigma, X, \mathcal{L})$ has a hyperbolic Oseledets splitting (see \cite[Example 5.6]{Crimmins}). Furthermore, if the leading eigenvalue of $\mathcal L_0$ is 1, with multiplicity $d\geq 1$, and $\mathcal L_0$ has no other eigenvalues of modulus one, then the random dynamical system has a hyperbolic Oseledets splitting of index $d$.  \label{rmk:free-splitting}
\end{remark}} 
\noindent As opposed to studying the asymptotic index of compactness \jp{directly}, one can often prove that the Perron-Frobenius operator cocycle is quasi-compact by showing the collection $(\mathcal{L}_\omega)_{\omega\in\Omega}$ satisfies a \textit{uniform Lasota-Yorke inequality}.
\begin{definition}
    We say that $\mathcal{L}_\omega:(X,\| \cdot\|)\to (X,\| \cdot\|)$ satisfies a \textit{uniform Lasota-Yorke inequality} with constants $C_1,C_2,r,R>0$ and $0<r<R\leq 1$, \jpt{if there exists a Banach space $(Y,|\cdot|)$ such that $X\subset Y$ and the embedding $X\hookrightarrow Y$ is compact,} and for $\mathbb{P}$-a.e. $\omega\in \Omega$, $f\in X$, and $n\in \mathbb{N}$ we have 
    $$\|\mathcal{L}_\omega^{(n)} f\|\leq C_1r^n\|f\|+C_2 R^n|f|$$
    where $|\cdot|$ is a weak norm on $(X,\| \cdot\|)$ \jpt{(i.e. there exists a constant $C>0$ such that $|f|\leq C \|f\|$ for all $f\in X$).}
    \label{def:ULY}
\end{definition}
\label{sec:prelim} \noindent
In the proceeding sections, we will consider perturbations of Perron-Frobenius operator cocycles. One way to quantify the size of perturbations is through the \textit{operator triple norm}.
\begin{definition}
    Let $\mathcal{L}_\omega :(X,\| \cdot\|_s)\to (X,\| \cdot\|_s)$ where $X$ is a Banach space equipped strong and weak norm $\|\cdot\|_s$ and $|\cdot|_w$, respectively. The \textit{operator triple norm} of $\mathcal{L}_\omega$ is
    $$\tnorm \mathcal{L}_\omega\tnorm_{s-w}:=\sup_{\|f\|_s=1}|\mathcal{L}_\omega f|_w.$$
    \label{def:triple-norm}
\end{definition} \noindent

\jp{Our results rely on the random perturbation theorem of Crimmins \cite{Crimmins}. For the reader's convenience we state the version of this result (developed in \cite{thermoformalism}) that may be applied in our setting. We refer the reader to \cite[Section 2]{Crimmins} for a primer on Saks spaces in the dynamical systems setting.
\\ \\
Let $\mathcal{LY}(C_1, C_2, r, R)$ be the collection of operators $\mathcal{L}\in\mathrm{End}(\mathbb{X},\sigma)$ that satisfy a uniform Lasota-Yorke inequality. We denote by $\mathrm{End}_S(\mathbb{X},\sigma)$ the set of all Saks space equicontinuous endomorphisms meaning that $\sup_{\omega\in\Omega}\|\mathcal{L}_\omega\|<\infty$ and for each $\eta>0$ there exists $C_\eta>0$ such that for $\mathbb{P}$-a.e. $\omega\in\Omega$ and $f\in X$
$$|\mathcal{L}_\omega f|\leq \eta \| f\| + C_\eta |f|.$$
For $\mathcal{L}\in \mathrm{End}_S(\mathbb{X},\sigma)$ and each $\varepsilon>0$ let
$$\mathcal{O}_\varepsilon(\mathcal{L})= \left\{ \mathcal{S}\in \mathrm{End}(\mathbb{X},\sigma) \ | \ \mathcal{S} \ \mathrm{is} \ \mathbb{P}\text{-}\mathrm{continuous \ with \ } \esssup_{\omega\in\Omega} \tnorm \mathcal{L}_\omega - \mathcal{S}_\omega \tnorm<\varepsilon \right\}.$$
In what follows, we write the Perron-Frobenius operator in subscripts to distinguish between projectors, Oseledets subspaces, and Lyapunov exponents.}   
\jp{\begin{theorem}[\cite{Crimmins,thermoformalism}]
    Suppose that $(X, \norm{\cdot}, |{\cdot}|)$ is a Saks space, with $(X, \norm{\cdot})$ being a Banach space, that $\mathcal{Q} = (\Omega, \mathcal{F}, \mathbb{P}, \sigma, X, Q)$ is a $\mathbb{P}$-continuous random dynamical system with ergodic invertible base and a hyperbolic Oseledets splitting of dimension $d \in \mathbb{Z}^+$, and that $Q \in \mathcal{LY}(C_1, C_2, r, R) \cap \mathrm{End}_S(\mathbb{X},\sigma)$ for some $C_1,C_2,R > 0$ and $r \in [0, e^{\kappa_Q})$.
  There exists $\varepsilon_0 >0$ such that if $\mathcal{P} = (\Omega, \mathcal{F}, \mathbb{P}, \sigma, X, P)$ is a $\mathbb{P}$-continuous random dynamical system with $P \in \mathcal{LY}(C_1, C_2, r, R) \cap \mathcal{O}_{\varepsilon_0}(Q)$ then $\mathcal{P}$ also has an Oseledets splitting of dimension $d$.
  In addition, there exists $c_0 < 2^{-1}\min_{1 \le i \le c_Q} \{l_{i,Q} - l_{i+1,Q}\}$ such that each $I_{i} = (l_{i,Q} - c_0, \max\{l_{i,Q}, \log(\delta_{1i} R)\} + c_0)$, $i \in \{1, \dots, c_Q\}$, separates the Lyapunov spectrum of $\mathcal{P}$, and the corresponding projections satisfy
  \begin{equation}\label{eq:stability_lyapunov_0000}
    \forall i \in \{1, \dots, c_Q\}, \text{ a.e. } \omega \in \Omega \quad \mathrm{rank}(\Pi_{I_i, P,\omega}) = m_{i,Q},
  \end{equation}
  and
  \begin{equation}\label{eq:stability_lyapunov_00000}
    \sup \left\{ \esssup_{\omega \in \Omega} \norm{\Pi_{I_i, P,\omega} } : P \in \mathcal{LY}(C_1, C_2, r, R) \cap \mathcal{O}_{\varepsilon_0}(Q), 1\le i \le c_Q \right\} < \infty.
  \end{equation}
  Moreover, for every $\nu > 0$ there exists $\varepsilon_\nu \in (0, \varepsilon_0)$ so that if $P \in \mathcal{LY}(C_1, C_2, r, R) \cap \mathcal{O}_{\varepsilon_\nu}(Q)$ then
  \begin{equation}\label{eq:stability_lyapunov_0}
      \sup_{1 \le i \le d} |{\gamma_{i, Q} - \gamma_{i, P}}| \le \nu,
  \end{equation}
  \begin{equation}\label{eq:stability_lyapunov_00}
      \sup_{1 \le i \le c_Q} \esssup_{\omega \in \Omega} \tnorm{\Pi_{I_{i},Q,\omega} - \Pi_{I_{i},P,\omega}}\tnorm \le \nu,
  \end{equation}
  and
  \begin{equation}\label{eq:stability_lyapunov_000}
      \esssup_{\omega \in \Omega} d_H(F_{c_Q,Q,\omega}, F_{c_P,P,\omega} ) \le \nu.
  \end{equation} \label{thrm:crimmins}
\end{theorem}
Note that for $i\in\{1,\cdots, c_Q\}$, $\Pi_{I_{i},Q,\omega}$ is the projection onto the fast Oseledets subspaces of $\mathcal{Q}$, denoted $E_{i,Q,\omega}$ with Lyapunov exponent $l_{i,Q}$ \jpt{(see \dfn{def:Osc-splitting} and \rem{rmk:gamma})}. Further, in \eqref{eq:stability_lyapunov_000}, $d_H$ is the metric given at the beginning of Section 2.1 in \cite{Crimmins}.} \\ \\
Finally, we make use of Landau notation throughout the paper. In the following definitions, we consider functions $f,g:\mathbb{R}\to \mathbb{R}$.
\begin{notation}
     We write $f(x)=O_{x\to a}(g(x))$ if there exists $M,\delta>0$ such that for all $x$ satisfying $|x-a|<\delta$,
     $$|f(x)|\leq M|g(x)|.$$
     \label{def:O}
\end{notation}
\begin{notation}
    We write $f(x)=o_{x\to a}(g(x))$ if for all $C>0$ there exists $\delta>0$ such that for all $x$ satisfying $|x-a|<\delta$,
    $$|f(x)|\leq C|g(x)|.$$
    \label{def:o}
\end{notation}
\begin{remark}
   In many situations, the constants $C,M$ involved in the above asymptotic approximations may depend on a second variable, say $\omega$. In this case we write $f(x)=O_{\omega,x\to a}(g(x))$ and $f(x)=o_{\omega,x\to a}(g(x))$, respectively.
\end{remark}

\section{Metastable systems and their perturbations}
\label{sec:map+pert}
In this section, we introduce a class of random dynamical systems with $m$ metastable states. We define these maps as perturbations of an autonomous system possessing $m\geq 2$ initially invariant intervals $I_1,\dots,I_m$, each supporting a unique ergodic absolutely continuous invariant measure (ACIM). Upon perturbation, so-called \textit{random holes} emerge, allowing trajectories to switch between $I_1,\dots,I_m$. This gives rise to systems with a unique ergodic random absolutely continuous invariant measure (RACIM) on $I_1\cup \cdots \cup I_m$, describing the long-term statistical behaviour of metastable systems.    

\subsection{The initial system}
Let $I=[-1,1]$ be equipped with the Borel $\sigma$-algebra $\mathcal{B}$, and Lebesgue measure $\leb$. Suppose that $T^0: I\to I$ is a piecewise $C^2$ uniformly expanding map with $m\geq 2$ invariant subintervals. This means $T^0$ satisfies the following conditions.
\begin{itemize}
    \item[\textbf{(I1)}] Piecewise $C^2$.\\
    There exists a critical set $\mathcal{C}^0 = \{-1=c_{0}<c_{1}<\cdots < c_{d}=1\}$ such that for each $i=0,\dots,d-1$, the map $T^0|_{(c_i,c_{i+1})}$ extends to a $C^2$ function $\hat{T}_i^0$ on a neighbourhood of $[c_i,c_{i+1}]$.  \label{list:I1}
    \end{itemize}
    \begin{itemize}    
    \item[\textbf{(I2)}] Uniform expansion.\\
    There exists $\Lambda>1$ such that 
    $$ \inf_{x\in I\setminus \mathcal{C}^0}|(T^0)^\prime(x)|\geq \Lambda.$$ \label{list:I2}
    \end{itemize}
    \begin{itemize}
    \item[\textbf{(I3)}] Existence of boundary points and invariant subintervals.\\
    There are boundary points $\mathfrak{B}:=\{b_0,\cdots, b_m\}$ where $(b_i)_{i=1}^{m-1}\subset(-1,1),\ b_0=-1$ and $b_m=1$, such that for $i=1,\dots,m$ the sets $I_i := [b_{i-1},b_i]$ are invariant under $T^0$ (for $i\in\{ 1,\cdots,m \}$, $T^0(I_i)\jpt{\subseteq} I_i$).\label{list:I3}
    \end{itemize}
    Denote by $\mathcal{L}^0$ the Perron-Frobenius operator associated with $T^0$ acting on $(\BV(I),\|\cdot\|_{\BV(I)})$ with weak norm $\|\cdot\|_{L^1(\leb)}$.
    \begin{remark}
     Thanks to \cite{EG_LY}, conditions \hyperref[list:I1]{\textbf{(I1)}} and \hyperref[list:I2]{\textbf{(I2)}} ensure that the Perron-Frobenius operator $\mathcal{L}^0$ acting on $(\BV(I),\|\cdot\|_{\BV(I)})$ satisfies a Lasota-Yorke inequality. This fact will be used when we wish to emphasise that $\mathcal{L}^0$ is quasi-compact. \label{rem:LY0}
    \end{remark}\noindent
The existence of an ACIM of bounded variation for $T^0|_{I_i}$ is guaranteed by the classical work of Lasota and Yorke \cite{LY_acim}. We assume in addition the following. 
\begin{itemize}
    \item[\textbf{(I4)}] Unique ACIMs on invariant sets.\\
    For $i \in \{1,\cdots,m\}$, $T^0|_{I_i}$ has only one ergodic ACIM $\mu_i$, whose density is $\phi_i:= d\mu_i/\dleb$.
    \label{list:I4}
\end{itemize}
\hyperref[list:I4]{\textbf{(I4)}} implies that all ACIMs of $T^0$ may be expressed as convex combinations of those ergodic measures supported on $I_1,\dots,I_m$. Conditions guaranteeing that \hyperref[list:I4]{\textbf{(I4)}} is satisfied are outlined in \cite[Theorem 1]{LY_acim}.
\\
\\
Let $\mathfrak{B}$ be as in \hyperref[list:I3]{\textbf{(I3)}}. We call the set of points belonging to $H^0:=(T^0)^{-1}(\mathfrak{B})\setminus \mathfrak{B}$ infinitesimal holes.
\begin{itemize}
    \item[\textbf{(I5)}] No return of the critical set to infinitesimal holes.\\
    For every $k>0$, $T^{0 \,(k)}(\mathcal{C}^0)\cap H^0 = \emptyset$.
    \label{list:I5}
\end{itemize}
As discussed in \cite[Section 2.1]{GTHW_metastable}, condition \hyperref[list:I5]{\textbf{(I5)}} is essential to ensure that for $i\in \{1,\cdots, m\}$, each unique invariant density $\phi_i$, guaranteed by \hyperref[list:I4]{\textbf{(I4)}}, is continuous at all points in $I_i \cap H^0$. Finally, we require that for $i\in \{1,\cdots, m\}$, the invariant densities $\phi_i$ are positive at each point in $I_i \cap H^0$. 
\begin{itemize}
    \item[\textbf{(I6)}] Positive ACIMs at infinitesimal holes.\\
   For $i \in\{1,\cdots, m\}$, $\phi_i$ is positive at each of the points in $I_i\cap H^0$.
   \label{list:I6}
\end{itemize}
Condition \hyperref[list:I6]{\textbf{(I6)}} is satisfied if, for example, the maps $T^0|_{I_i}$ are weakly covering for $i\in\{1,\cdots, m\}$ \cite{Liverani_DOC}.\footnote{\jp{For $j\in\{1,\cdots,m\}$, a piecewise expanding map $T|_{I_j}:{I_j\to I_j}$ with critical set $\mathcal{S} = \{b_{j-1}=s_{0}<s_{1}<\cdots < s_{d}=b_j\}$ is weakly covering if there is some $N\in\mathbb{N}$ such that for every $i$, $\cup_{k=0}^N T^{(k)}|_{I_j}([s_i,s_{i+1}])=I_j$.}}\\

\subsection{The perturbations} 
\label{sec:perts}
In what follows, let $(\Omega,\mathcal{F},\mathbb{P})$ be a probability space. Fix $\varepsilon>0$ and let $\omega\in\Omega$. We consider $C^2$-small perturbations of $T^0:I\to I$, denoted $T^\varepsilon:\Omega\times I \to I$, driven by an ergodic transformation $\sigma:\Omega\to \Omega$.\footnote{For notational convenience $T^0$ will also denote the \textit{random} map $T^0:\Omega\times I \to I$ which satisfies $T^0_\omega :=T^0_{\omega_0}$ for any $\omega_0\in\Omega$.} This means $\sigma:\Omega\to \Omega$ and $T^\varepsilon:\Omega\times I \to I$ satisfy the following.
\begin{itemize}
    \item[\textbf{(P1)}] Ergodic driving and finite range.\\
    $\sigma:\Omega\to \Omega$ is an ergodic, $\mathbb{P}$-preserving homeomorphism of the probability space $(\Omega,\mathcal{F},\mathbb{P})$; for all $\varepsilon\geq0$, the mapping $\omega \mapsto T_{\omega}^\varepsilon$ has finite range; and the skew-product
    $$(\omega,x)\mapsto (\sigma\omega,T_\omega^\varepsilon(x))$$
    is measurable with respect to the product $\sigma$-algebra $\mathcal{F}\otimes \mathcal{B}$ on $\Omega\times I$.
    \label{list:P1}
\end{itemize}
\begin{itemize}
    \item[\textbf{(P2)}] $C^2$-small perturbations.\\
    \jpt{The critical set for $T_\omega^\varepsilon$ is given by} $\mathcal{C}^\varepsilon_\omega = \{-1 = c_{0,\omega}^\varepsilon<\cdots < c_{d,\omega}^\varepsilon = 1\}$ where for each $i=0,\dots,d$, $\omega \mapsto c_{i,\omega}^\varepsilon$ is measurable and $\varepsilon \mapsto c_{i,\omega}^\varepsilon$ is $C^2$. Furthermore, there exists a $\delta>0$ such that: 
    \label{list:P2}
\begin{itemize}
    \item[(a)] for $i=1,\dots,d-2$, $[c_{i}+\delta,c_{i+1}-\delta]\subset [c_{i,\omega}^\varepsilon,c_{i+1,\omega}^\varepsilon] \subset [c_{i}-\delta,c_{i+1}+\delta]$,\footnote{In this way, we have a one-to-one correspondence between the critical sets of $T_\omega^\varepsilon$ and $T^0$ given by $\mathcal{C}_\omega^\varepsilon$ and $\mathcal{C}^0$, respectively.}
    \item[(b)] for $i=0,\cdots,d$, \jpt{$\lim_{\varepsilon\to 0}\esssup_{\omega\in\Omega}|c_{i,\omega}^\varepsilon-c_{i}|=0$, and}
    \item[(c)] for $i=0,\dots,d-1$, there is a $C^2$ extension $\hat{T}_{i,\omega}^\varepsilon:[c_{i}-\delta,c_{i+1}+\delta]\to \mathbb{R}$ of $T_\omega^\varepsilon|_{(c_{i,\omega}^\varepsilon,c_{i+1,\omega}^\varepsilon)}$ that converges in $C^2$, and uniformly over $\omega\in\Omega$ away from a  $\mathbb{P}$-null set to the $C^2$ extension $\hat{T}_{i}^0:[c_{i}-\delta,c_{i+1}+\delta]\to \mathbb{R}$ of $T^0|_{(c_{i},c_{i+1})}$ as $\varepsilon\to 0$.
\end{itemize}
\end{itemize}
Denote by $\mathcal{L}_\omega^\varepsilon$ the Perron-Frobenius operator associated with $T_\omega^\varepsilon$ acting on $(\BV(I),\|\cdot\|_{\BV(I)})$ with weak norm $\|\cdot\|_{L^1(\leb)}$. 
\begin{remark}
  We impose \hyperref[list:P1]{\textbf{(P1)}} to ensure the mapping $\omega\mapsto \mathcal{L}_\omega^\varepsilon$ is $\mathbb{P}$-continuous (recall \dfn{def:P-cont}). This is satisfied through the relaxed condition that $\omega \mapsto T_{\omega}^\varepsilon$ has countable range. In our setting, to ensure uniform convergence of certain quantities, we instead require the stricter condition that $\omega \mapsto T_{\omega}^\varepsilon$ has finite range (as discussed in \cite[Remark 3.10]{BS_rand}).
\end{remark}

\begin{remark}
    Thanks to \cite[Proposition 3.12]{FGTQ_LYmapStab} (which follows from \cite[Lemma 13]{K_StochStab}) and \cite[Example 5.2]{EG_LY}, \hyperref[list:P2]{\textbf{(P2)}} asserts that 
     $$\lim_{\varepsilon\to 0} \esssup_{\omega\in\Omega} \tnorm \mathcal{L}_\omega^\varepsilon - \mathcal{L}^0 \tnorm_{\BV(I)-L^1(\leb)}=0,$$
     where $\tnorm\cdot\tnorm_{\BV(I)-L^1(\leb)}$ denotes the $\BV-L^1$ triple norm (see \dfn{def:triple-norm}).  
     \label{rem:p1-imp-trip}
\end{remark} 
\noindent
\noindent
For $i,j\in \{1,\cdots, m\}$ we define the random holes $H_{i,j,\omega}^\varepsilon$ as the set of all points mapping from $I_i$ to $I_j$ under one iteration of $T_\omega^\varepsilon$. Namely, $H_{i,j,\omega}^\varepsilon:= I_i \cap (T_\omega^\varepsilon)^{-1}(I_j)$. 
\begin{remark}
    Without loss of generality, we assume that under one iteration of $T_\omega^\varepsilon$, for some $i\in\{2,\cdots,m-1\}$, points in $I_i$ can only map to neighbouring sets $I_{i-1},I_{i+1}$ or remain in $I_i$. When $i=1$, points in $I_i$ can only map to $I_{i+1}$ or remain in $I_i$, whereas if $i=m$, points in $I_i$ can only map to $I_{i-1}$ or remain in $I_i$. With this in mind, $H_{i,j,\omega}^\varepsilon\neq \emptyset$ if and only if $|i-j|=1$. 
    \label{rem:neighbour}
\end{remark}
Recall that $H^0=(T^0)^{-1}(\mathfrak{B})\setminus \mathfrak{B}$ consists of infinitesimal holes.
\begin{itemize}
    \item[\textbf{(P3)}] Convergence of holes.\\
    For $i,j\in\{1,\cdots, m\}$ and each $\omega\in\Omega$, $H_{i,j,\omega}^\varepsilon$ is a union of finitely many intervals, and as $\varepsilon\to 0$, $H_{i,j,\omega}^\varepsilon$ converges to $I_{i}\cap H^0 $ (in the Hausdorff metric) uniformly over $\omega\in\Omega$ away from a $\mathbb{P}$-null set.
    \label{list:P3}
\end{itemize}
Since the holes are themselves functions of $\omega\in\Omega$, we place some constraints on how their measures vary across fibres.   
\begin{itemize}
    \item[\textbf{(P4)}] Measure of holes and uniform covering.\\
    There exists a $\beta^*>0$ such that for $i,j\in\{1,\cdots,m\}$, $\mu_i(H_{i,j,\omega}^\varepsilon)=\varepsilon \beta_{i,j,\omega}+o_{\varepsilon\to 0}(\varepsilon)$ where $\beta_{i,j} \in  L^\infty(\mathbb{P})$ satisfies $\beta_{i,j,\omega}\geq \beta^*$ for all $\omega\in\Omega$.\footnote{We emphasise that the error in the measure of the hole is independent of $\omega\in\Omega$.}
    \label{list:P4}
\end{itemize}
\begin{remark}
    In \hyperref[list:P4]{\textbf{(P4)}}, the uniform (over $\omega\in\Omega$) lower bound on $\beta_{i,j}\in L^\infty(\mathbb{P})$ ensures that $(T_\omega^\varepsilon)_{\omega\in \Omega}$ is uniformly (over $\omega\in\Omega$) covering. That is, for every $\varepsilon>0$ and every subinterval $J\subset I$, there exists $k^\varepsilon:=k^\varepsilon(J)$ such that for $\mathbb{P}$-a.e. $\omega\in\Omega$, $T_{\omega}^{\varepsilon\, (k^\varepsilon)}(J)=I$. This assumption is required so that random metastable systems satisfy the quenched CLT of \cite{DFGTV_limthrm}.  
\end{remark}
\noindent
\jp{With the aim to apply \thrm{thrm:crimmins}, we require the perturbed system to satisfy a Lasota-Yorke inequality. In particular, we assume the following.}
\begin{itemize}
    \item[\textbf{(P5)}] Uniform Lasota-Yorke inequality.\\
    The Perron-Frobenius operator associated with $T_\omega^\varepsilon$, denoted $\mathcal{L}_\omega^\varepsilon$, acting on \linebreak $(\BV(I),\|\cdot\|_{\BV(I)})$ with weak norm $\|\cdot\|_{L^1(\leb)}$ satisfies a uniform Lasota-Yorke inequality across both $\omega\in\Omega$ and $\varepsilon>0$ (see \dfn{def:ULY}).
    \label{list:P5}
\end{itemize}
Further, we require $T_\omega^\varepsilon$ to admit a unique ergodic RACIM (see \rem{rem:RACIM}).
\begin{itemize}
    \item[\textbf{(P6)}] Unique RACIM.\\
    For $\varepsilon>0$, $(T_\omega^\varepsilon)_{\omega\in\Omega}$ has only one ergodic RACIM $(\mu_\omega^\varepsilon)_{\omega\in\Omega}$, with density $(\phi_\omega^\varepsilon)_{\omega\in\Omega}:=\left(d\mu_\omega^\varepsilon/\dleb\right)_{\omega\in\Omega}$.
    \label{list:P6}
\end{itemize}
Cases in which \hyperref[list:P6]{\textbf{(P6)}} is satisfied are outlined in \cite{Unique-RACIM}. Finally, as discussed in \cite[Section 2.4]{GTHW_metastable}, we enforce a condition that ensures no holes emerge near the boundary. 
\begin{itemize}
    \item[\textbf{(P7)}] Boundary condition.\\
    For all boundary points $b_i\in \mathfrak{B}$.
    \begin{itemize}
        \item[(a)] If $b_i\notin \mathcal{C}_0$, then $T^0(b_i)=b_i$ and for all $\varepsilon>0$ and $\mathbb{P}$-a.e. $\omega\in\Omega$, $T_\omega^\varepsilon(b_i)=b_i$. 
        \item[(b)] If $b_i\in \mathcal{C}_0$, then $T^0(b^-_i)<b_i<T^0(b^+_i)$ and for all $\varepsilon>0$ and $\mathbb{P}$-a.e. $\omega\in\Omega$, $b_i\in \mathcal{C}_\omega^\varepsilon$.\footnote{We denote by  $T^0(b^\mp_i)$ the left and right limits of $T^0(x)$ as $x\to b_i$, respectively.}
    \end{itemize}
    \label{list:P7}
\end{itemize}

\noindent
Throughout the remainder of the paper, we assume that the conditions of \Sec{sec:map+pert} are satisfied. 
 
\section{The open system}
In this section, we consider a sequence of random open dynamical systems derived from the sequence of random dynamical systems of metastable maps $T_\omega^\varepsilon:I\to I$. We define the Perron-Frobenius operator associated with such systems and study its spectral properties. \\
\\
For $j\in\{1,\cdots,m\}$ and $\varepsilon\geq 0$ consider $H_{j,\omega}^\varepsilon:= H_{j,j-1,\omega}^\varepsilon \cup H_{j,j+1,\omega}^\varepsilon$ as a hole. For $f\in\BV(I_j)$ we let 
\begin{equation}
    \mathcal{L}_{j,\omega}^\varepsilon (f)(x):= \mathcal{L}_\omega^\varepsilon(\mathds{1}_{I_j\setminus H_{j,\omega}^\varepsilon}\cdot f)(x)\label{eqn:openop}
\end{equation}
be the Perron-Frobenius operator acting on $(\BV(I_j),\|\cdot\|_{\BV(I_j)})$ where we recall that $\mathcal{L}_\omega^\varepsilon$ is the Perron-Frobenius operator associated with $T_\omega^\varepsilon$ acting on $(\BV(I),\|\cdot\|_{\BV(I)})$. Let $T_{j,\omega}^\varepsilon:I_j\to I_j$ denote the open map associated with the operator $\mathcal{L}_{j,\omega}^\varepsilon$. When $\varepsilon=0$ in \eqref{eqn:openop}, we denote by $\mathcal{L}_j^0$ the resulting Perron-Frobenius operator associated with the map $T_j^0:I_j\to I_j$.\footnote{The map $T_j^0:I_j\to I_j$ can also be interpreted as the map $T^0|_{I_j}:I_j\to I_j$.}$^{,}$\footnote{For notational convenience $T^0_j$ will also denote the \textit{random} map $T^0_j:\Omega\times I \to I$ which satisfies $T^0_{j,\omega} :=T^0_{j,\omega_0}$ for any $\omega_0\in\Omega$.} In the spirit of \cite[Section 2.C]{thermoformalism}, for all $\varepsilon>0$, and each $\omega\in\Omega$, $n\in\mathbb{N}$, and $j\in\{1,\cdots,m\}$, let $X_{j,\omega,n}^\varepsilon := \{x\in I_j \ | \ T_{j,\omega}^{\varepsilon\, (i)}(x)\notin H_{j,\sigma^i\omega}^\varepsilon \ \mathrm{for\ all} \ 0\leq i \leq n\}$, let $\mathcal{Z}_{j}^{0\,(n)}$ denote the partition of monotonicity of $T_j^{0\, (n)}$, and with $\Lambda>1$ from \hyperref[list:I2]{\textbf{(I2)}}, let $\mathscr{A}_{j}^{0\, (n)}$ be the collection of all finite partitions of $I_j$ such that
\begin{align}
\var_{A_{j,i}}(|(T_j^{0\, (n)})^\prime|^{-1})\leq 2\Lambda^{-n}
\end{align}
for each $\mathcal{A}_j=\{A_{j,i}\}\in\mathscr{A}_{j}^{0\, (n)}$.
Given $\mathcal{A}_j\in\mathscr{A}_{j}^{0\,(n)}$, we set $\mathcal{Z}_{j,\omega,*}^{\varepsilon\, (n)}(\mathcal{A}_j):=\{Z\in \widehat{\mathcal{Z}}_{j,\omega}^{\varepsilon\, (n)}(\mathcal{A}_j): Z\subseteq X_{j,\omega,n-1}^\varepsilon \}$ where $\widehat{\mathcal{Z}}_{j,\omega}^{\varepsilon\, (n)}(\mathcal{A}_j)$ is the coarsest partition among all those finer than $\mathcal{A}_j$ and $\mathcal{Z}_{j}^{0\,(n)}$ so that all elements of $\widehat{\mathcal{Z}}_{j,\omega}^{\varepsilon\, (n)}(\mathcal{A}_j)$ are either disjoint from $X_{j,\omega,n-1}^\varepsilon$ or contained in $X_{j,\omega,n-1}^\varepsilon$. \jpt{In other words, the finite partition $\mathcal{A}_j$ of $I_j$ refines $\mathcal{Z}_j^{0\, (n)}$ and gives control of the expansion of $T_j^{0\, (n)}$. Then $\widehat{\mathcal{Z}}_{j,\omega}^{\varepsilon\, (n)}(\mathcal{A}_j)$ is a further refinement of $\mathcal{A}_j$ distinguishing between intervals which land in the hole by time $n-1$, and those which remain outside the hole for atleast $n-1$ steps.} In what follows, for $f\in \BV(I_j)$, let \Change{$\leb_j(f):= {\leb}( f\cdot\mathds{1}_{I_j})$} denote the Lebesgue measure on $I_j$.
\Change{\begin{remark}
For $j\in \{1,\cdots, m\}$ we note that $\leb_j$ is not necessarily a probability measure. In \Sec{sec:open} and \Sec{sec:jump} we use this as a matter of notational convenience when studying the open maps $(T_{j,\omega}^\varepsilon)_{\omega\in \Omega}$ for $\varepsilon\geq 0$.            
\end{remark}}
We impose the following condition, altered in our setting from \cite{thermoformalism}.
\begin{itemize}
    \item[\textbf{(O1)}] Non-vanishing surviving branches.\\
    \jp{Let $\Lambda>1$ be as in \hyperref[list:I2]{\textbf{(I2)}}. There exists $n'\in \mathbb{N}$ and $\varepsilon_0>0$ such that
    \begin{align}
        {(\Lambda)^{-n'}}<\frac{1}9\label{eqn:enough-expansion}
    \end{align}
    and for each $j \in \{1,\cdots, m\}$}
    \begin{equation}
        \jp{\inf_{0\leq \varepsilon\leq \varepsilon_0 }}\essinf_{\omega\in \Omega}\min_{Z\in \mathcal{Z}_{j,\omega,*}^{\varepsilon\, (n')}(\mathcal{A}_j)}\leb_j(Z)>0.\label{eqn:large-survival}
    \end{equation}
    \label{list:O1}
\end{itemize}
\begin{remark}
    We note that \eqref{eqn:enough-expansion} implies that for $j\in\{1,\cdots,m\}$, 
    $$\sup_{0\leq \varepsilon\leq \varepsilon_0}\esssup_{\omega\in \Omega}\left\| \frac{1}{(T_{j,\omega}^{\varepsilon\, (n')})^\prime}\right\|_{L^\infty(\leb_j)}<\frac{1}{9}.$$
    This follows since \hyperref[list:I2]{\textbf{(I2)}} and \hyperref[list:P2]{\textbf{(P2)}}  ensures that for all $\varepsilon>0$ sufficiently small, and $\mathbb{P}$-a.e. $\omega \in\Omega$,
    $$\left\| \frac{1}{(T_{j,\omega}^{\varepsilon\, (n')})^\prime}\right\|_{L^\infty(\leb_j)}\leq (\Lambda)^{-n'}.$$\label{rem:T-eps'-bound}
\end{remark}
\jpt{The following remark provides an alternative approach to verifying that \eqref{eqn:large-survival} holds in \hyperref[list:O1]{\textbf{(O1)}}.}
\begin{remark}
  \jp{Note that \eqref{eqn:enough-expansion} holds whenever $n'$ is sufficiently large}. For \eqref{eqn:large-survival}, thanks to \hyperref[list:P1]{\textbf{(P1)}}, particularly the fact that $\omega\mapsto T_\omega^\varepsilon$ has finite range for all $\varepsilon\geq 0$, one can obtain a lower bound, uniform over $\omega\in\Omega$, for $\min_{Z\in \mathcal{Z}_{j,\omega,*}^{\varepsilon\, (n')}(\mathcal{A}_j)}\leb_j(Z)$. We now provide alternative checkable conditions to \eqref{eqn:large-survival} \jpt{once $n'$ has been fixed satisfying \eqref{eqn:enough-expansion},} that ensures $\jpt{\inf_{0\leq \varepsilon\leq \varepsilon_0}}\essinf_{\omega\in \Omega}\min_{Z\in \mathcal{Z}_{j,\omega,*}^{\varepsilon\, (n')}(\mathcal{A}_j)}\leb_j(Z)$ does not vanish.   
  \begin{itemize}
      \item[{(\ref*{eqn:large-survival}a)}] Uniform open covering. \\
      Given $n'\in\mathbb{N}$ (\jpt{satisfying \eqref{eqn:enough-expansion}}), for each $j\in \{1,\cdots,m\}$ there exists $k_o(n')\in\mathbb{N}$ such that for $\mathbb{P}$-a.e. $\omega\in\Omega$, all $\varepsilon>0$ sufficiently small, and all $Z\in\mathcal{Z}_{j,\omega,*}^{\varepsilon\, (n')}$ we have $T_j^{0\, (k_o(n'))}(Z) = I_j$. \label{list:7a}
      \item[{(\ref*{eqn:large-survival}b)}] Restriction on periodic critical points.\\
      For each $j\in \{1,\cdots, m\}$ and $c_i\in I_j \cap \mathcal{C}^0$
      $$\essinf_{1\leq k\leq n'}\min_{1\leq i,i'\leq d}|T_j^{0\, (k)}(c_i)-c_{i'}|>0$$
      where $n'$ \jpt{satisfies \eqref{eqn:enough-expansion}}.
 \label{list:7b}
  \end{itemize}
(\hyperref[list:7a]{\ref*{eqn:large-survival}a}) implies \eqref{eqn:large-survival} through a similar computation to that made in the proof of \cite[Lemma 2.5.10]{thermoformalism}. Indeed, for any $Z\in \mathcal{Z}_{j,\omega,*}^{\varepsilon\, (n)}$, (\hyperref[list:7a]{\ref*{eqn:large-survival}a}) implies that $$\leb_j(Z)=\leb_j(\mathcal{L}_j^{0\, (k_o(n'))}\mathds{1}_Z)\geq \frac{\leb_j((T_j^{0\, (k_o(n'))})(Z))}{\|(T_j^{0\, (k_o(n'))})'\|_{L^\infty(\leb_j)}}= \frac{\leb_j(I_j)}{\|(T_j^{0\, (k_o(n'))})'\|_{L^\infty(\leb_j)}}>0.$$
Finally, (\hyperref[list:7b]{\ref*{eqn:large-survival}b}) implies \eqref{eqn:large-survival} through a similar argument to that made in \cite[Section 3.3]{B_positiveop}. (\hyperref[list:7b]{\ref*{eqn:large-survival}b}) ensures that for every $\varepsilon>0$, each element in $\mathcal{Z}_{j,\omega,*}^{\varepsilon\, (n')}$ may be identified with an element in $\mathcal{Z}_{j,\omega,*}^{0\, (n')}$. Thus for $\varepsilon>0$ sufficiently small, the size of the shortest branch for $T_{j,\omega}^{\varepsilon\, (n')}$ can be made arbitrarily close to the size of the smallest branch of $T_j^{0\, (n')}$.    
\end{remark}
In this section, we aim to prove \thrm{thrm:openopprop}. This result is obtained through the following sequence of lemmata. 
\\
\\
Due to \eqref{eqn:openop}, we can express compositions of the open operator in terms of the closed operator.
\begin{lemma}
    In the setting of \thrm{thrm:openopprop}, for any $n\in\mathbb{N}$ and $f\in\BV(I_j)$, 
    $$\mathcal{L}_{j,\omega}^{\varepsilon\, (n)}(f) = \mathcal{L}_\omega^{\varepsilon\, (n)}\left( f\cdot \prod_{i=0}^{n-1}\mathds{1}_{(T_\omega^{\varepsilon\, (i)})^{-1}(I_j\setminus H_{j,\sigma^i\omega}^\varepsilon)} \right).$$
    \label{lem:open-comp}
\end{lemma}
\begin{proof}
    Thanks to \eqref{eqn:openop}, for $n\in\mathbb{N}$ 
    \begin{align*}
        \mathcal{L}_{j,\omega}^{\varepsilon\, (n)}(f) &= \mathcal{L}_{\sigma^n\omega}^\varepsilon\left( \mathds{1}_{I_j \setminus H_{j,\sigma^n \omega}^\varepsilon}\cdot \mathcal{L}_{j,\omega}^{\varepsilon\, (n-1)}( f) \right).
    \end{align*}
    By an inductive procedure, the result follows.
\end{proof}

We are interested in the spectral properties of $\mathcal{L}_{j,\omega}^\varepsilon$ when acting on functions $f\in \BV(I_j)$. The following lemma will be used to show that the Lyapunov exponents and Oseledets projections of $\mathcal{L}_{j,\omega}^\varepsilon$ depend continuously on the perturbations.    
\begin{lemma}
In the setting of \thrm{thrm:openopprop},  
$$\lim_{\varepsilon\to 0}\esssup_{\omega\in \Omega}\tnorm\mathcal{L}_{j,\omega}^\varepsilon - \mathcal{L}_j^0 \tnorm_{\BV(I_j)-L^1(\leb_j)} =0.$$
\label{lem:trip_norm}
\end{lemma}
\begin{proof}
    We compute $\tnorm\mathcal{L}_{j,\omega}^\varepsilon - \mathcal{L}_j^0 \tnorm_{\BV(I_j)-L^1(\leb_j)}$. Indeed, recalling \dfn{def:triple-norm} and \eqref{eqn:openop},
    \begin{align}
     \tnorm\mathcal{L}_{j,\omega}^\varepsilon - \mathcal{L}_j^0 \tnorm_{\BV(I_j)-L^1(\leb_j)}&= \sup_{\|f\|_{\BV(I_j)}=1}\| (\mathcal{L}_{j,\omega}^\varepsilon - \mathcal{L}_j^0)(f)\|_{L^1(\leb_j)}  \nonumber \\
     &\leq \sup_{\|f\|_{\BV(I_j)}=1}\| \mathcal{L}_\omega^\varepsilon(\mathds{1}_{I_j\setminus H_{j,\omega}^\varepsilon}\cdot f) - \mathcal{L}_\omega^\varepsilon(\mathds{1}_{I_j}\cdot f)\|_{L^1(\leb_j)} \nonumber \\&\quad+\sup_{\|f\|_{\BV(I_j)}=1}\| \mathcal{L}_\omega^\varepsilon(\mathds{1}_{I_j}\cdot f) - \mathcal{L}_j^0(\mathds{1}_{I_j}\cdot f)\|_{L^1(\leb_j)} \nonumber \\
     &\leq \sup_{\|f\|_{\BV(I_j)}=1}\| \mathcal{L}_\omega^\varepsilon(\mathds{1}_{ H_{j,\omega}^\varepsilon}\cdot f)\|_{L^1(\leb_j)}+ \tnorm\mathcal{L}_\omega^\varepsilon - \mathcal{L}^0\tnorm_{\BV(I)-L^1(\leb)} \nonumber \\
     &=  \tnorm\mathcal{L}_\omega^\varepsilon - \mathcal{L}^0\tnorm_{\BV(I)-L^1(\leb)}. \label{eqn:zeropush}
    \end{align}
     In \eqref{eqn:zeropush} we have used the fact that $\| \mathcal{L}_\omega^\varepsilon(\mathds{1}_{ H_{j,\omega}^\varepsilon}\cdot f)\|_{L^1(\leb_j)}=0$. Therefore, due to \hyperref[list:P2]{\textbf{(P2)}}, which thanks to \rem{rem:p1-imp-trip}, asserts that $\tnorm\mathcal{L}_\omega^\varepsilon - \mathcal{L}^0\tnorm_{\BV(I)-L^1(\leb)}=o_{\varepsilon\to 0}(1)$, the result follows. 
\end{proof}
\noindent
\begin{lemma} In the setting of \thrm{thrm:openopprop}, for all $f\in \BV(I_j)$, $\omega\in\Omega$, $\varepsilon>0$, and $n\in\mathbb{N}$ we have that for $A\subseteq I_j$, 
$$\|\mathcal{L}_{\omega}^{\varepsilon\, (n)}(f\cdot \mathds{1}_A)\|_{L^1(\leb_j)}\leq \|f\|_{L^1(\leb_j)}.$$
\label{lem:L1bound-closed}
\end{lemma}
\begin{proof}
This is a direct computation. Observe that 
\begin{align*}
\|\mathcal{L}_{\omega}^{\varepsilon\, (n)}(f\cdot \mathds{1}_A)\|_{L^1(\leb_j)}&= \int_I \mathds{1}_{I_j}|\mathcal{L}_{\omega}^{\varepsilon\, (n)}(f\cdot \mathds{1}_A)|\, {\dleb(x)}{}\\
&= \int_I |\mathcal{L}_{\omega}^{\varepsilon\, (n)}(f\cdot \mathds{1}_{A\cap (T_\omega^{\varepsilon\, (n)})^{-1}(I_j)})|\, {\dleb(x)}{}\\
&\leq \int_I |f\cdot \mathds{1}_{A\cap (T_\omega^{\varepsilon\, (n)})^{-1}(I_j)}|\, {\dleb(x)}{}\\
&= \int_{I\cap A\cap (T_\omega^{\varepsilon\, (n)})^{-1}(I_j)} |f|\, \dleb(x)\\
&\leq \|f\|_{L^1(\leb_j)}.
\end{align*}
In the above we have used the fact that $(T_\omega^{\varepsilon\, (n)})^{-1}(I)=I$ and $A\subseteq I_j$.
\end{proof}
\jp{We now show that $\mathcal{L}_{j,\omega}^\varepsilon$ satisfies a uniform Lasota-Yorke inequality across both $\varepsilon>0$ and $\omega\in\Omega$.}
\begin{lemma}
    In the setting of \thrm{thrm:openopprop}, there exist constants $C_1,C_2,r>0$ with $r<1$ such that for all $f\in \BV(I_j)$, $\varepsilon>0$, $\omega\in\Omega$, and $n\in \mathbb{N}$, 
    \begin{equation}
        \|\mathcal{L}_{j,\omega}^{\varepsilon\, (n)}f \|_{\BV(I_j)}\leq C_1r^n\| f\|_{\BV(I_j)}+ C_2\| f\|_{L^1(\leb_j)}. \label{eqn:ULY-open}
    \end{equation}
    \label{lem:ULY_open}
\end{lemma}
\begin{proof}
\Change{Following \cite[Lemma 2.C.1]{thermoformalism}, one can show that given $n'\in\mathbb{N}$ from $\hyperref[list:O1]{\textbf{(O1)}}$, for each $j\in\{1,\cdots, m\}$ and $f\in\BV(I_j)$,
\begin{equation}
    \var_{I_j}(\mathcal{L}_{j,\omega}^{\varepsilon\, (n')}f)\leq \left\| \frac{1}{(T_{j,\omega}^{\varepsilon\, (n')})^\prime}\right\|_{L^\infty(\leb_j)}\left(9 \var_{I_j}(f) + \frac{8}{\min_{Z\in \mathcal{Z}_{j,\omega,*}^{\varepsilon\, (n')}(\mathcal{A}_j)}\leb_j(Z)} \|f\|_{L^1(\leb_j)}\right).
    \label{eqn:varbound}
\end{equation}
Thus, by \hyperref[list:O1]{\textbf{(O1)}} and \rem{rem:T-eps'-bound}, for all $\varepsilon>0$ there exists $K_1,r>0$ such that $9\left\|{1}/{(T_{j,\omega}^{\varepsilon\, (n')})^\prime}\right\|_{L^\infty(\leb_j)}\leq r<1$ and $\min_{Z\in \mathcal{Z}_{j,\omega,*}^{\varepsilon\, (n')}(\mathcal{A}_j)}\leb_j(Z)\geq K_1^{-1}$. Therefore from \eqref{eqn:varbound},
\begin{equation}
     \var_{I_j}(\mathcal{L}_{j,\omega}^{\varepsilon\, (n')}f)\leq r^{n'} \var_{I_j}(f) + K_1r^{n'} \|f\|_{L^1(\leb_j)}. \label{eqn:updated-varbound}
\end{equation}
Next, through \lem{lem:open-comp} and \lem{lem:L1bound-closed} 
\begin{align}
    \left\|\mathcal{L}_{j,\omega}^{\varepsilon\, (n')}f\right\|_{L^1(\leb_j)}&=\left\| \mathcal{L}_{\omega}^{\varepsilon\, (n')}\left(f\cdot \prod_{i=0}^{n'-1}\mathds{1}_{(T_{\omega}^{\varepsilon\, (i)})^{-1}(I_j\setminus H_{j,\sigma^{i}\omega}^\varepsilon)}\right)
 \right\|_{L^1(\leb_j)} \nonumber\\
 &\leq \|f\|_{L^1(\leb_j)}, \label{eqn:L1bound}
\end{align}
since $\bigcap_{i=0}^{n-1} (T_{\omega}^{\varepsilon\, (i)})^{-1}(I_j\setminus H_{j,\sigma^{i}\omega}^\varepsilon) \subset I_j$. Thus, \eqref{eqn:updated-varbound} and \eqref{eqn:L1bound} imply that there exists a $K_2>0$ such that
\begin{align*}
    \|\mathcal{L}_{j,\omega}^{\varepsilon\, (n')}f \|_{\BV(I_j)}&\leq  r^{n'} \var_{I_j}(f) + (K_1r^{n'}+1) \|f\|_{L^1(\leb_j)}\\
    &\leq r^{n'}\|f\|_{\BV(I_j)} + K_2 \|f\|_{L^1(\leb_j)}.
\end{align*}
To obtain \eqref{eqn:ULY-open}, one proceeds in the usual way by using blocks of length $kn'$.} 
\end{proof}
We now prove that the Lyapunov exponents and Oseledets projections of $\mathcal{L}_{j,\omega}^\varepsilon$ depend continuously on $\varepsilon$ through Crimmins' random perturbation theorem, \jp{\thrm{thrm:crimmins}}. As discussed in \Sec{sec:prelim}, the results of \cite{Crimmins} require the Banach space on which the Perron-Frobenius operator is acting to be separable. It is argued in \cite[Appendix 2.B]{thermoformalism} that Crimmins' stability result, in particular \thrm{thrm:crimmins}, holds for the non-separable Banach space $\BV(I_j)$ with \hyperref[list:P1]{\textbf{(P1)}}. Therefore, if 
\begin{itemize}
    \item[1.] $(\Omega,\mathcal{F},\mathbb{P},\sigma, \BV(I_j),\mathcal{L}^0_j)$ has a hyperbolic Oseledets splitting and is $\|\cdot\|_{L^1(\leb_j)}$-bounded (see \dfn{def:hyperbolic-splitting});
    \item[2.] the set $\{(\Omega,\mathcal{F},\mathbb{P},\sigma, \BV(I_j),\mathcal{L}^\varepsilon_j)\}_{\varepsilon\geq 0}$ satisfies a uniform Lasota-Yorke inequality (see \dfn{def:ULY}); and 
    \item[3.] $\lim_{\varepsilon\to 0} \esssup_{\omega\in\Omega} \tnorm \mathcal{L}_{j,\omega}^\varepsilon - \mathcal{L}^0_j \tnorm_{\BV(I_j)-L^1(\leb_j)}=0$ (see \dfn{def:triple-norm}),
\end{itemize}
then $\mathcal{L}_{j,\omega}^\varepsilon$ has an Oseledets splitting for sufficiently small $\varepsilon$, and the Lyapunov exponents and Oseledets projections of $\mathcal{L}_{j,\omega}^\varepsilon$ converge to those of $\mathcal{L}^0_j$ as $\varepsilon\to 0$. Furthermore, \thrm{thrm:crimmins} shows that $\mathcal{L}_{j,\omega}^\varepsilon$ has a uniform spectral gap.  
\begin{lemma}
    In the setting of \thrm{thrm:openopprop}, $\mathcal{L}_{j,\omega}^\varepsilon$ has an Oseledets splitting for sufficiently small $\varepsilon$, and the Lyapunov exponents and Oseledets projections of $\mathcal{L}_{j,\omega}^\varepsilon$ converge to those of $\mathcal{L}^0_j$ as $\varepsilon\to 0$ uniformly over $\omega\in\Omega$ away from a $\mathbb{P}$-null set. Further, for $\mathbb{P}$-a.e. $\omega\in\Omega$ and $\varepsilon\geq 0$ 
    \begin{itemize}
        \item[(a)] There is a functional $\nu_{j,\omega}^\varepsilon\in \BV^*(I_j)$, $\lambda_{j,\omega}^\varepsilon\in \mathbb{R}^+$, and $\phi_{j,\omega}^\varepsilon\in \BV(I_j)$ such that for all $f\in\BV(I_j)$ 
        $$\mathcal{L}_{j,\omega}^\varepsilon (\phi_{j,\omega}^\varepsilon)=\lambda_{j,\omega}^\varepsilon\phi_{j,\sigma\omega}^\varepsilon \quad \mathrm{and} \quad \nu_{j,\sigma\omega}^\varepsilon(\mathcal{L}_{j,\omega}^\varepsilon (f))=\lambda_{j,\omega}^\varepsilon\nu_{j,\omega}^\varepsilon(f).$$
        \item[(b)] There is an operator $Q_{j,\omega}^\varepsilon:\BV(I_j)\to \BV(I_j)$ such that for each $f\in \BV(I_j)$ we have 
        $$(\lambda_{j,\omega}^\varepsilon)^{-1}\mathcal{L}_{j,\omega}^\varepsilon (f)=: \nu_{j,\omega}^\varepsilon(f)\phi_{j,\sigma\omega}^\varepsilon + Q_{j,\omega}^\varepsilon(f).$$
        Furthermore, we have 
        $$Q_{j,\omega}^\varepsilon(\phi_{j,\omega}^\varepsilon)=0 \quad \mathrm{and} \quad \nu_{j,\sigma\omega}^\varepsilon(Q_{j,\omega}^\varepsilon(f))=0.$$
        \item[(c)] The leading Oseledets space of $\mathcal{L}_{j,\omega}^\varepsilon$ is one-dimensional and spanned by $\phi_{j,\omega}^\varepsilon\in \BV(I_j)$. Further, $$\lim_{\varepsilon\to 0} \esssup_{\omega\in\Omega} \|\phi_{j,\omega}^\varepsilon - \phi_j\|_{L^1(\leb_j)} = 0.$$
        \item[(d)] There exists a constant $C>0$ and $\theta \in (0,1)$ such that for all $f\in \BV(I_j)$, $n\geq 0$ and $\varepsilon>0$,
     $$\|Q_{j,\omega}^{\varepsilon\, (n)}(f)\|_{\BV(I_j)}\leq C\theta^n\|f\|_{\BV(I_j)}.$$
    \end{itemize}
     \label{lem:crim_open}
\end{lemma}
\begin{proof}
    This is a direct application of \thrm{thrm:crimmins}. Equip the Banach space $(\BV(I_j),\|\cdot\|_{\BV(I_j)})$ with weak norm $\|\cdot\|_{L^1(\leb_j)}$. Due to \lem{lem:trip_norm} and \lem{lem:ULY_open} it remains to show that $(\Omega,\mathcal{F},\mathbb{P},\sigma, \BV(I_j),\mathcal{L}^0_j)$ has a hyperbolic Oseledets splitting and is $\|\cdot\|_{L^1(\leb_j)}$-bounded. Indeed, observe that $\mathcal{L}_j^0$ is the Perron-Frobenius operator associated with a piecewise $C^2$ uniformly expanding map satisfying \hyperref[list:I1]{\textbf{(I1)}} and \hyperref[list:I2]{\textbf{(I2)}}. Due to \rem{rem:LY0}, this implies that $\mathcal{L}_j^0$ satisfies a Lasota-Yorke inequality and is thus quasi-compact when acting on $\BV(I_j)$. Since $\mathcal{L}_j^0$ is independent of $\omega\in\Omega$, quasi-compactness ensures that $(\Omega,\mathcal{F},\mathbb{P},\sigma, \BV(I_j),\mathcal{L}^0_j)$ has a hyperbolic Oseledets splitting. \jp{We note that in our setting, due to \hyperref[list:I4]{\textbf{(I4)}} and \rem{rmk:free-splitting}, $(\Omega,\mathcal{F},\mathbb{P},\sigma,\BV(I_j),\mathcal{L}^0_j)$ has a hyperbolic Oseledets splitting of index $d=1$}. Further, \lem{lem:L1bound-closed} ensures that for any $f\in\BV(I_j)$, $\|\mathcal{L}_j^0f\|_{L^1(\leb_j)}\leq \|f\|_{L^1(\leb_j)}$. We now address the dimensionality of the leading Oseledets space of $\mathcal{L}_{j,\omega}^\varepsilon$ in the statement of (c). Due to \hyperref[list:I4]{\textbf{(I4)}}, the top Oseledets space of $\mathcal{L}_j^0$ is one-dimensional, spanned by $\phi_j$ with associated Lyapunov exponent $l_{j,1}^0=0$. By continuity of the Oseledets projections, this implies that the leading Oseledets space of $\mathcal{L}_{j,\omega}^\varepsilon$ is also one dimensional for all $\varepsilon>0$ sufficiently small. Further, by \cite[Corollary 2.5]{DFGTV_limthrm}, the leading Oseledets space for the backward adjoint cocycle of $\mathcal{L}_{j,\omega}^\varepsilon$ is one dimensional and spanned by $\nu_{j,\omega}^\varepsilon\in \BV^*(I_j)$.         
\end{proof}

\begin{corollary}
In the setting of \thrm{thrm:openopprop}, the family of functionals $(\nu_{j,\omega}^\varepsilon)_{\omega\in\Omega}\in \BV^*(I_j)$ satisfy 
     $$\lim_{\varepsilon\to 0} \esssup_{\omega\in\Omega} \|\nu_{j,\omega}^\varepsilon - \leb_j\|_{\BV^*(I_j)}=0.$$ 
     \label{cor:meas_conv}
     \begin{proof}
      Let $\Pi_{j,\omega}^\varepsilon$ denote the projection onto the leading Oseledets space of $\mathcal{L}_{j,\omega}^\varepsilon$. That is, for $f\in\BV(I_j)$, let $\Pi_{j,\omega}^\varepsilon(f):=\nu_{j,\omega}^\varepsilon(f)\phi_{j,\omega}^\varepsilon$. Due to \lem{lem:crim_open}, the Oseledets projections of $\mathcal{L}_{j,\omega}^\varepsilon$ converge to those of $\mathcal{L}_j^0$ as $\varepsilon\to 0$ in the sense that  
      \begin{equation}
          \lim_{\varepsilon\to 0}\esssup_{\omega\in\Omega}\tnorm\Pi_{j,\omega}^\varepsilon - \Pi_{j}^0 \tnorm_{\BV(I_j)-L^1(\leb_j)} = 0.\label{eqn:proj_conv}
      \end{equation}
      But, \begingroup \allowdisplaybreaks
      \begin{align*}
         \tnorm\Pi_{j,\omega}^\varepsilon - \Pi_{j}^0 \tnorm_{\BV(I_j)-L^1(\leb_j)}&= \sup_{\|f\|_{\BV(I_j)}=1}\|(\Pi_{j,\omega}^\varepsilon - \Pi_{j}^0)(f)\|_{L^1(\leb_j)}\\
         &=\sup_{\|f\|_{\BV(I_j)}=1}\|\nu_{j,\omega}^\varepsilon(f)\phi_{j,\omega}^\varepsilon - \leb_j(f)\phi_j\|_{L^1(\leb_j)}\\
         &\geq \sup_{\|f\|_{\BV(I_j)}=1} \left| |\nu_{j,\omega}^\varepsilon(f)|\|\phi_{j,\omega}^\varepsilon\|_{L^1(\leb_j)}- |\leb_j(f)| \|\phi_{j}\|_{L^1(\leb_j)} \right|\\
         &= \|\phi_{j}\|_{L^1(\leb_j)}\sup_{\|f\|_{\BV(I_j)}=1} \left\| |\nu_{j,\omega}^\varepsilon(f)|(1+o_{\varepsilon\to 0}(1)) - |\leb_j(f)| \right|\\
         &\geq\sup_{\|f\|_{\BV(I_j)}=1} |\nu_{j,\omega}^\varepsilon(f)- \leb_j(f)|- \sup_{\|f\|_{\BV(I_j)}=1}|\nu_{j,\omega}^{\varepsilon}(f)o_{\varepsilon\to 0}(1)| \\
         &= \|\nu_{j,\omega}^\varepsilon- \leb_j\|_{\BV^*(I_j)} - o_{\varepsilon\to 0}(1)\|\nu_{j,\omega}^{\varepsilon}\|_{\BV^*(I_j)}.
      \end{align*} \endgroup
       \lem{lem:crim_open}(a) asserts that \Change{for all $\varepsilon>0$ and $\mathbb{P}$-a.e. $\omega\in\Omega$, $\nu_{j,\omega}^\varepsilon\in \BV^*(I_j)$}. Therefore, there exists a constant $C>0$ such that for all $\varepsilon>0$ sufficiently small, $\esssup_{\omega\in\Omega}\|\nu_{j,\omega}^{\varepsilon}\|_{\BV^*(I_j)}<C$. Thus,
      \begin{align*}
      \|\nu_{j,\omega}^\varepsilon- \leb_j\|_{\BV^*(I_j)}\leq \tnorm\Pi_{j,\omega}^\varepsilon - \Pi_{j}^0 \tnorm_{\BV(I_j)-L^1(\leb_j)} + o_{\varepsilon\to 0}(1)   
      \end{align*}
      meaning 
      \begin{align*}
          \lim_{\varepsilon\to 0}\esssup_{\omega\in\Omega}\|\nu_{j,\omega}^\varepsilon- \leb_j\|_{\BV^*(I_j)}&\leq \lim_{\varepsilon\to 0}\esssup_{\omega\in\Omega} \tnorm\Pi_{j,\omega}^\varepsilon - \Pi_{j}^0 \tnorm_{\BV(I_j)-L^1(\leb_j)}=0.
      \end{align*}
      In the last line, we have used \eqref{eqn:proj_conv}. 
     \end{proof}
\end{corollary}
\begin{remark}
    By a similar argument to that made in the proof of Lemma 2.5.10 in \cite{thermoformalism}, through \lem{lem:crim_open}(a),(b), $\nu_{j,\omega}^\varepsilon(\phi_{j,\omega}^\varepsilon)=1>0$, and through \cite[Lemma 2.6]{DFGTV_limthrm}, for $g\in F_{j,\omega}^\varepsilon$, the Oseledets space of $\mathcal{L}_{j,\omega}^\varepsilon$ complementary to $\mathrm{span}\{\phi_{j,\omega}^\varepsilon\}$, $\nu_{j,\omega}^\varepsilon(g)=0$. \Change{Due to \lem{lem:crim_open}, $\mathcal{L}_{j,\omega}^\varepsilon$ admits an Oseledets splitting for $\varepsilon>0$ sufficiently small. Therefore, any $f\in \BV(I_j)$ may be expressed as \jp{$\alpha\phi_{j,\omega}^\varepsilon+g$ for $g\in F_{j,\omega}^\varepsilon$ and $\alpha\in\mathbb{R}$.}} \jp{Thanks to \cor{cor:meas_conv}, and the fact that $\leb_{j}(\phi_j^0)=1$, we may deduce that for $\varepsilon>0$ sufficiently small, $\alpha >0$}. Thus, by linearity, $\nu_{j,\omega}^\varepsilon$ is a positive linear functional, meaning that due to the Riesz-Markov Theorem (see for example \cite[Theorem A.3.11]{VO_ET}) $\nu_{j,\omega}^\varepsilon$ can be identified with a real finite Borel measure on $I_j$.  
\end{remark}
As in \cite{gtp_met}, it is important to understand the behaviour of the functions spanning the leading Oseledets space over the holes. 
\begin{lemma}
In the setting of \thrm{thrm:openopprop}, if $\phi_{j,\omega}^\varepsilon\in\BV(I_j)$ denotes the functions spanning the leading Oseledets space of $\mathcal{L}_{j,\omega}^\varepsilon$, then for any $j,k\in\{1,\cdots, m\}$
$$\lim_{\varepsilon\to 0}\esssup_{\omega\in\Omega}\sup_{x\in H_{j,k,\omega}^\varepsilon}|\phi_{j,\omega}^\varepsilon(x) - \phi_j(x)|=0.$$
    \label{lem:phiepsj_open}
\end{lemma}
\begin{proof}
    As in \cite[Lemma 5.5]{gtp_met}, one can follow \jp{the same} argument \jpt{as} that of Lemma 3.14 in \cite{BS_rand} \jp{by keeping track of the randomness}, and replacing the assumption of finite $\Omega$ with \hyperref[list:P1]{\textbf{(P1)}}, ensuring $\omega\mapsto T_\omega^\varepsilon$ is finite. The techniques used in \cite[Lemma 3.14]{BS_rand} involve controlling the so-called regular and saltus parts of the invariant density for the annealed closed operator. Upon close inspection of the proofs, given \lem{lem:ULY_open}, such techniques can be used to provide the same control of the regular and saltus parts of the function spanning the leading Oseledets space of $\mathcal{L}_{j,\omega}^\varepsilon$. In turn, by a similar argument to that made in \cite[Lemma 3.14]{BS_rand}, the result follows. 
\end{proof}
To conclude, we provide a first order approximation for the leading Lyapunov multipliers $\lambda_{j,\omega}^\varepsilon$ from \lem{lem:crim_open}(a). This result relies on the sequential perturbation theorem \cite[
Theorem 2.1.2]{thermoformalism}. 
\begin{remark}
    We note that the conditions of \cite[Theorem 2.1.2]{thermoformalism} are required to hold for each $\omega\in\Omega$. \Change{In our setting, we find that such conditions hold uniformly over $\omega\in\Omega$ away from a $\mathbb{P}$-null set. Upon close inspection of the proof of \cite[Theorem 2.1.2]{thermoformalism}, a uniform over $\omega\in\Omega$ away from a $\mathbb{P}$-null set analogue of \cite[Theorem 2.1.2]{thermoformalism} holds.}
\end{remark}
\begin{lemma}
In the setting of \thrm{thrm:openopprop}, for $\mathbb{P}$-a.e. $\omega\in \Omega$ there exists $a,M>0$ such that \jp{for $\varepsilon>0$ sufficiently small}
$$a\varepsilon \leq \leb_j(H_{j,\omega}^\varepsilon)\leq M\varepsilon.$$
\label{lem:leb_estimate}
\end{lemma}
\begin{proof}
For the upper bound, we use a similar argument to that made in the proof of Lemma 5.8 in \cite{gtp_met}. In particular, if $K\in\mathbb{N}$, consider the set of infinitesimal holes in $I_j$ given by $H^0\cap I_j=\{h_j^1,\cdots, h_j^K\}$.\footnote{Recall that $H^0:=(T^0)^{-1}(\mathfrak{B})\setminus \mathfrak{B}$ is the set of infinitesimal holes.} Take $H_{j,\omega}^{{\varepsilon}}=\cup_{i=1}^K H_{j,\omega}^{i,{\varepsilon}}$ such that For each $i=1,\dots, K$, by \hyperref[list:P3]{\textbf{(P3)}}, $H_{j,\omega}^{i,{\varepsilon}}\to h_j^i$ in the Hausdorff metric uniformly over $\omega\in\Omega$ away from a $\mathbb{P}$-null set. Recall that by \hyperref[list:I5]{\textbf{(I5)}}, $\phi_j$ is continuous at all points in $H^0\cap I_j$. Thus, by Lebesgue's differentiation theorem, 
    \begin{equation}
        {\leb_j(H_{j,\omega}^{{\varepsilon}})} =\mu_j(H_{j,\omega}^{{\varepsilon}}) \left(\sum_{i=1}^K (\phi_j(h_j^i)+o_{{\varepsilon} \to 0}(1)) \frac{\leb_j(H_{j,\omega}^{i,{\varepsilon}})}{\leb_j(H_{j,\omega}^{{\varepsilon}})}\right)^{-1}.
        \label{eqn:meas_ratio}
    \end{equation}
    Note that $\sum_{i=1}^J\frac{\leb_j(H_{j,\omega}^{i,{\varepsilon}})}{\leb_j(H_{j,\omega}^{{\varepsilon}})}=1$. Further, \hyperref[list:I6]{\textbf{(I6)}} asserts that there exists a constant $L>0$ such that for each $i=1,\dots, K$, $\phi_j(h_j^i)\geq L>0$. Therefore, \eqref{eqn:meas_ratio} asserts that
    \begin{align}
        {\leb_j(H_{j,\omega}^{{\varepsilon}})} &\leq \frac{\mu_j(H_{j,\omega}^{{\varepsilon}})}{L+o_{\varepsilon\to 0}(1)} =\frac{\varepsilon(\beta_{j,j+1,\omega}+\beta_{j,j-1,\omega}+o_{\varepsilon\to 0}(1))}{L+o_{\varepsilon\to 0}(1)}.\label{eqn:leb_est}
    \end{align}
   But by \hyperref[list:P4]{\textbf{(P4)}}, $\beta_{i,j}\in L^{\infty}(\mathbb{P})$ for $i,j\in \{1,\cdots,m\}$, and thus we may deduce that for $\mathbb{P}$-a.e. $\omega\in\Omega$ there exists an $M>0$ such that \jp{for $\varepsilon>0$ sufficiently small}
    \begin{align*}
        \leb_j(H_{j,\omega}^{{\varepsilon}})\leq M\varepsilon. 
    \end{align*}
For the lower bound, for $j\in\{1,\cdots, m\}$, since $\phi_j\in \BV(I_j)$, we have via \hyperref[list:P4]{\textbf{(P4)}} that 
\begin{align*}
    \mu_j(H_{j,\omega}^\varepsilon)&= \int_{H_{j,\omega}^\varepsilon}\phi_j(x)\, \dleb(x)\leq \|\phi_j\|_{\BV(I_j)}\leb_j(H_{j,\omega}^\varepsilon)
\end{align*}
where $\mu_j(H_{j,\omega}^\varepsilon) \geq \varepsilon \beta^* + o_{\varepsilon \to 0}(\varepsilon)$ for some $\beta^*>0$. Thus, for $\varepsilon>0$ sufficiently small, $\leb_j(H_{j,\omega}^\varepsilon)\geq a\varepsilon$.
\end{proof}
\begin{lemma}
    In the setting of \thrm{thrm:openopprop}, for $\mathbb{P}$-a.e. $\omega\in \Omega$ and $\varepsilon\geq 0$,
    \begin{align}
        \Delta_{j,\omega}^\varepsilon&:= \leb_j((\mathcal{L}_j^0 - \mathcal{L}_{j,\omega}^\varepsilon)(\phi_j))= \varepsilon(\beta_{j,j-1,\omega}+\beta_{j,j+1,\omega}) +o_{\varepsilon\to 0}(\varepsilon).\label{eqn:delta}
        \end{align}
        Further, there exists a constant $M>0$ such that \Change{for each $j\in\{1,\cdots, m\}$, $\mathbb{P}$-a.e. $\omega\in\Omega$}, \jp{and $\varepsilon>0$ sufficiently small}
        \begin{align}  
        \eta_{j,\omega}^\varepsilon&:=\|\leb_j(\mathcal{L}_j^0 - \mathcal{L}_{j,\omega}^\varepsilon) \|_{\BV^*(I_j)}\leq \varepsilon M. \label{eqn:eta} 
    \end{align}
    \label{lem:delta-eta}
\end{lemma}
\begin{proof}
    We first obtain \eqref{eqn:delta}. Using \eqref{eqn:openop}, \Change{for $\mathbb{P}$-a.e. $\omega\in \Omega$} and $\varepsilon\geq 0$
    \begin{align*}
    \leb_j((\mathcal{L}_j^0 - \mathcal{L}_{j,\omega}^\varepsilon)(\phi_j))&= \int_{I_j} (\mathcal{L}_j^0 - \mathcal{L}_{j,\omega}^\varepsilon)(\phi_j)(x)\, \dleb(x)\\
    &= \int_{I_j} \phi_j(x) \dleb(x) - \int_{I_j} \mathcal{L}_{j,\omega}^\varepsilon(\phi_j)(x)\, \dleb(x)\\
    &=\int_{I_j} \phi_j(x) \dleb(x) - \int_{(T_{\omega}^\varepsilon)^{-1}(I_j)} (\mathds{1}_{I_j\setminus H_{j,\omega}^\varepsilon} \cdot \phi_j)(x)\, \dleb(x)\\
    &=\int_{I_j} \phi_j(x) \dleb(x) - \int_{I_j\setminus H_{j,\omega}^\varepsilon} \phi_j(x)\, \dleb(x)\\
    &= \mu_j(H_{j,\omega}^\varepsilon)\\
    &= \varepsilon(\beta_{j,j-1,\omega}+\beta_{j,j+1,\omega}) +o_{\varepsilon\to 0}(\varepsilon).
    \end{align*}
    In the above we have used \hyperref[list:P4]{\textbf{(P4)}} and the fact that $(T_{\omega}^\varepsilon)^{-1}(I_j)= I_j\setminus H_{j,\omega}^\varepsilon\cup (H_{j+1,j,\omega}^\varepsilon\cup H_{j-1,j,\omega}^\varepsilon)$. We now prove that \eqref{eqn:eta} holds. Indeed, by following a similar calculation to above, since $(T^0)^{-1}(I_j)=I_j$ 
    \begin{align*}
        \|\leb_j(\mathcal{L}_j^0 - \mathcal{L}_{j,\omega}^\varepsilon) \|_{\BV^*(I_j)}&=\sup_{\|f\|_{\BV(I_j)}=1}|\leb_j((\mathcal{L}_j^0 - \mathcal{L}_{j,\omega}^\varepsilon)(f))|\\
        &= \sup_{\|f\|_{\BV(I_j)}=1} \left|\int_{I_j} f(x)\, \dleb(x)-\int_{I_j\setminus H_{j,\omega}^\varepsilon} f(x)\, \dleb(x)  \right|\\
        &\leq \sup_{\|f\|_{\BV(I_j)}=1} \|\mathds{1}_{H_{j,\omega}^\varepsilon}\cdot f\|_{L^1(\leb_j)} \\
        &\leq\|\mathds{1}_{H_{j,\omega}^\varepsilon}\|_{L^1(\leb_j)}\\
        &= \leb_j(H_{j,\omega}^\varepsilon).
    \end{align*}
    By applying \lem{lem:leb_estimate} the result follows.
\end{proof}
We then obtain the following first order estimate on the Lyapunov multipliers for $\mathcal{L}_{j,\omega}^\varepsilon$.
\begin{lemma}
In the setting of \thrm{thrm:openopprop}, fix $k\in\mathbb{N}$, then 
\begin{align*}
    q_{j,\omega}^{0\, (k)} &:= \lim_{\varepsilon\to 0} \frac{\leb_j\left( (\mathcal{L}_j^0 - \mathcal{L}_{j,\omega}^\varepsilon)(\mathcal{L}_{j,\sigma^{-k}\omega}^{\varepsilon \, (k)})(\mathcal{L}_j^0 - \mathcal{L}_{j,\sigma^{-k-1}\omega}^\varepsilon)(\phi_j) \right)}{\leb_j((\mathcal{L}_j^0 - \mathcal{L}_{j,\omega}^\varepsilon)(\phi_j))}=0
\end{align*}
\Change{uniformly over $\omega\in\Omega$ away from a $\mathbb{P}$-null set.}
\label{lem:qk}
\end{lemma}
\begin{proof}
    From \lem{lem:delta-eta}, $\leb_j((\mathcal{L}_j^0 - \mathcal{L}_{j,\omega}^\varepsilon)(\phi_j))=\varepsilon(\beta_{j,j-1,\omega}+\beta_{j,j+1,\omega}) +o_{\varepsilon\to 0}(\varepsilon)$. We aim to show that $$N_{j,k,\omega}^\varepsilon:=\leb_j\left( (\mathcal{L}_j^0 - \mathcal{L}_{j,\omega}^\varepsilon)(\mathcal{L}_{j,\sigma^{-k}\omega}^{\varepsilon \, (k)})(\mathcal{L}_j^0 - \mathcal{L}_{j,\sigma^{-k-1}\omega}^\varepsilon)(\phi_j) \right)= \Change{O_{\varepsilon\to 0}(\varepsilon)(O_{\varepsilon\to 0}(\varepsilon)+o_{\varepsilon\to 0}(1))}.$$ By a similar argument to that made in \lem{lem:delta-eta}, for any $\jp{g}\in\BV(I_j)$ 
    \begin{equation}
    \leb_j((\mathcal{L}_j^0 - \mathcal{L}_{j,\omega}^\varepsilon)(\jp{g}))= \int_{H_{j,\omega}^\varepsilon}\jp{g}(x)\, \dleb(x).    
    \end{equation}
    Take $\jp{g} = (\mathcal{L}_{j,\sigma^{-k}\omega}^{\varepsilon \, (k)})(\mathcal{L}_j^0 - \mathcal{L}_{j,\sigma^{-k-1}\omega}^\varepsilon)(\phi_j)\in\BV(I_j)$, then using \lem{lem:open-comp}
    \begin{align*}
        N_{j,k,\omega}^\varepsilon &= \int_{H_{j,\omega}^\varepsilon}(\mathcal{L}_{j,\sigma^{-k}\omega}^{\varepsilon \, (k)})(\mathcal{L}_j^0 - \mathcal{L}_{j,\sigma^{-k-1}\omega}^\varepsilon)(\phi_j)(x)\, \dleb(x)\\
        &= \int_{H_{j,\omega}^\varepsilon} \mathcal{L}_{\sigma^{-k}\omega}^{\varepsilon\, (k)}\left( (\mathcal{L}_j^0 - \mathcal{L}_{j,\sigma^{-k-1}\omega}^\varepsilon)(\phi_j) \cdot \prod_{i=0}^{k-1}\mathds{1}_{(T_{\sigma^{-k}\omega}^{\varepsilon\, (i)})^{-1}(I_j\setminus H_{j,\sigma^{i-k}\omega}^\varepsilon)} \right)(x)\, \dleb(x)\\
        &= \int_{(T_{\sigma^{-k}\omega}^{\varepsilon\, (k)})^{-1}(H_{j,\omega}^\varepsilon) \cap \left( \bigcap_{i=0}^{k-1} (T_{\sigma^{-k}\omega}^{\varepsilon\, (i)})^{-1}(I_j\setminus H_{j,\sigma^{i-k}\omega}^\varepsilon) \right)} (\mathcal{L}_j^0 - \mathcal{L}_{j,\sigma^{-k-1}\omega}^\varepsilon)(\phi_j)(x)\, \dleb(x).
    \end{align*}
    Set $S_{j,k,\omega}^\varepsilon:=(T_{\sigma^{-k}\omega}^{\varepsilon\, (k)})^{-1}(H_{j,\omega}^\varepsilon) \cap \left( \bigcap_{i=0}^{k-1} (T_{\sigma^{-k}\omega}^{\varepsilon\, (i)})^{-1}(I_j\setminus H_{j,\sigma^{i-k}\omega}^\varepsilon) \right)$. We note that for $\varepsilon\geq 0$, $S_{j,k,\omega}^\varepsilon \subseteq I_j$. Thus, through \eqref{eqn:openop}
    \begin{align*}
        N_{j,k,\omega}^\varepsilon &= \int_{(T^0)^{-1}(S_{j,k,\omega}^\varepsilon)} \phi_j(x)\, \dleb(x) - \int_{(T_{\sigma^{-k-1}\omega}^\varepsilon)^{-1}(S_{j,k,\omega}^\varepsilon)} (\phi_j\cdot \mathds{1}_{I_j\setminus H_{j,\sigma^{-k-1}\omega}^\varepsilon})(x)\, \dleb(x)\\
        &= \mu_j((T^0)^{-1}(S_{j,k,\omega}^\varepsilon)) - \mu_{j}((T_{\sigma^{-k-1}\omega}^\varepsilon)^{-1}(S_{j,k,\omega}^\varepsilon) \cap \left[I_j \cap (T_{\sigma^{-k-1}\omega}^\varepsilon)^{-1}(I_j))\right])\\
        &= \mu_j((T^0)^{-1}(S_{j,k,\omega}^\varepsilon)) - \mu_{j}((T_{\sigma^{-k-1}\omega}^\varepsilon)^{-1}(I_j\cap S_{j,k,\omega}^\varepsilon)\cap I_j)\\
        &=\mu_j((T^0)^{-1}(S_{j,k,\omega}^\varepsilon)) - \mu_{j}((T_{\sigma^{-k-1}\omega}^\varepsilon)^{-1}(S_{j,k,\omega}^\varepsilon)).
    \end{align*}
   Now, observe that due to \hyperref[list:P4]{\textbf{(P4)}} and \eqref{eqn:leb_est}, the set $S_{j,k,\omega}^\varepsilon$ consists of finitely many disjoint intervals $S_{j,k,\omega}^{1,\varepsilon},\dots,S_{j,k,\omega}^{P,\varepsilon}$ with $\leb_{j}(S_{j,k,\omega}^{i,\varepsilon})=O_{\varepsilon\to 0}(\varepsilon)$ for each $i=1,\dots,P$. Thus, from above, for $\varepsilon>0$ sufficiently small \begingroup \allowdisplaybreaks
    \begin{align*}
    N_{j,k,\omega}^\varepsilon&\leq \frac{\mu_j((T^0)^{-1}(S_{j,k,\omega}^\varepsilon))}{\leb_j((T^0)^{-1}(S_{j,k,\omega}^\varepsilon))}\sum_{i=1}^P\leb_j((T^0)^{-1}(S_{j,k,\omega}^{i,\varepsilon})) \\
    & \quad - \frac{\mu_{j}((T_{\sigma^{-k-1}\omega}^\varepsilon)^{-1}(S_{j,k,\omega}^\varepsilon))}{\leb_j((T_{\sigma^{-k-1}\omega}^\varepsilon)^{-1}(S_{j,k,\omega}^\varepsilon))}\sum_{i=1}^P\leb_j((T_{\sigma^{-k-1}\omega}^\varepsilon)^{-1}(S_{j,k,\omega}^{i,\varepsilon})) \\
    &\leq \frac{\mu_j((T^0)^{-1}(S_{j,k,\omega}^\varepsilon))}{\leb_j((T^0)^{-1}(S_{j,k,\omega}^\varepsilon))}\sum_{i=1}^P\frac{\leb_j(S_{j,k,\omega}^{i,\varepsilon})}{\inf_{x\in S_{j,k,\omega}^{i,\varepsilon}}|(T^0)^\prime(x)|} \\
    & \quad - \frac{\mu_{j}((T_{\sigma^{-k-1}\omega}^\varepsilon)^{-1}(S_{j,k,\omega}^\varepsilon))}{\leb_j((T_{\sigma^{-k-1}\omega}^\varepsilon)^{-1}(S_{j,k,\omega}^\varepsilon))}\sum_{i=1}^P\frac{\leb_j(S_{j,k,\omega}^{i,\varepsilon})}{\sup_{x\in S_{j,k,\omega}^{i,\varepsilon}}|(T_\omega^\varepsilon)^\prime(x)|}\\
    &= O_{\varepsilon\to 0}(\varepsilon)\Bigg( \frac{\mu_j((T^0)^{-1}(S_{j,k,\omega}^\varepsilon))}{\leb_j((T^0)^{-1}(S_{j,k,\omega}^\varepsilon))}\sum_{i=1}^P\frac{1}{\inf_{x\in S_{j,k,\omega}^{i,\varepsilon}}|(T^0)^\prime(x)|} \\
    &\quad- \frac{\mu_{j}((T_{\sigma^{-k-1}\omega}^\varepsilon)^{-1}(S_{j,k,\omega}^\varepsilon))}{\leb_j((T_{\sigma^{-k-1}\omega}^\varepsilon)^{-1}(S_{j,k,\omega}^\varepsilon))}\sum_{i=1}^P\frac{1}{\sup_{x\in S_{j,k,\omega}^{i,\varepsilon}}|(T^0)^\prime(x)|+o_{\varepsilon\to 0}(1)} \Bigg).
    \end{align*} \endgroup
   Here we are using \hyperref[list:I5]{\textbf{(I5)}} which implies that $\mathcal{C}^0\cap (T^0)^{-k}(H^0)=\emptyset$. Thus for $\varepsilon$ sufficiently small, thanks to \hyperref[list:P2]{\textbf{(P2)}} and \hyperref[list:P3]{\textbf{(P3)}}, $(T^0)^\prime$ and $(T_\omega^\varepsilon)^\prime$ exist on $S_{j,k,\omega}^{i,\varepsilon}$, and further, $\sup_{x\in S_{j,k,\omega}^{i,\varepsilon}}|(T_\omega^\varepsilon)^\prime(x)| = \sup_{x\in S_{j,k,\omega}^{i,\varepsilon}}|(T^0)^\prime(x)|+o_{\varepsilon\to 0}(1)$. To conclude, we apply a similar argument to that used in the proof of \lem{lem:delta-eta} to obtain \eqref{eqn:meas_ratio}. Note that \hyperref[list:I5]{\textbf{(I5)}} implies that $\phi_j$ is continuous on $I_j \cap (T^0)^{-k}(H^0)$. Due to \hyperref[list:P2]{\textbf{(P2)}} and \hyperref[list:P3]{\textbf{(P3)}}, we may apply Lebesgue's differentiation theorem to say that for $\varepsilon>0$ sufficiently small, there exists $B_1>0$ such that 
    \begin{align*}
        \frac{\mu_j((T^0)^{-1}(S_{j,k,\omega}^\varepsilon))}{\leb_j((T^0)^{-1}(S_{j,k,\omega}^\varepsilon))}=B_1 + o_{\varepsilon\to 0}(1) \quad \mathrm{and} \quad \frac{\mu_{j}((T_{\sigma^{-k-1}\omega}^\varepsilon)^{-1}(S_{j,k,\omega}^\varepsilon))}{\leb_j((T_{\sigma^{-k-1}\omega}^\varepsilon)^{-1}(S_{j,k,\omega}^\varepsilon))}=B_1 + o_{\varepsilon\to 0}(1). 
    \end{align*}
    Therefore, for $\varepsilon>0$ sufficiently small 
    \begin{align*}
        N_{j,k,\omega}^\varepsilon&\leq O_{\varepsilon\to 0}(\varepsilon)\left( \sum_{i=1}^P \frac{\sup_{x\in S_{j,k,\omega}^{i,\varepsilon}}|(T^0)^\prime(x)|-\inf_{x\in S_{j,k,\omega}^{i,\varepsilon}}|(T^0)^\prime(x)|+o_{\varepsilon\to 0}(1)}{\inf_{x\in S_{j,k,\omega}^{i,\varepsilon}}|(T^0)^\prime(x)|(\sup_{x\in S_{j,k,\omega}^{i,\varepsilon}}|(T^0)^\prime(x)|+o_{\varepsilon\to 0}(1))}  \right)\\
        &\leq O_{\varepsilon\to 0}(\varepsilon)\left( \sum_{i=1}^P \frac{\leb_j(S_{j,k,\omega}^{i,\varepsilon})\sup_{x\in S_{j,k,\omega}^{i,\varepsilon}}|(T^0)^{\prime\prime}(x)|+o_{\varepsilon\to 0}(1)}{\inf_{x\in S_{j,k,\omega}^{i,\varepsilon}}|(T^0)^\prime(x)|(\sup_{x\in S_{j,k,\omega}^{i,\varepsilon}}|(T^0)^\prime(x)|+o_{\varepsilon\to 0}(1))}  \right)\\
        &= O_{\varepsilon\to 0}(\varepsilon)(O_{\varepsilon\to 0}(\varepsilon)+o_{\varepsilon\to 0}(1))
    \end{align*}
    since $\leb_{j}(S_{j,k,\omega}^{i,\varepsilon})=O_{\varepsilon\to 0}(\varepsilon)$, and $\inf_{x\in S_{j,k,\omega}^{i,\varepsilon}}|(T^0)^\prime(x)|>1$ (by \hyperref[list:I2]{\textbf{(I2)}}) for each $i=1,\dots,P$. So, we may conclude that 
    \begin{align*}
        q_{j,\omega}^{0\, (k)}&= \lim_{\varepsilon\to 0} \frac{N_{j,k,\omega}^\varepsilon}{\varepsilon(\beta_{j,j-1,\omega}+\beta_{j,j+1,\omega}) +o_{\varepsilon\to 0}(\varepsilon)} = \lim_{\varepsilon\to 0} \frac{O_{\varepsilon\to 0}(\varepsilon)(O_{\varepsilon\to 0}(\varepsilon)+o_{\varepsilon\to 0}(1))}{\varepsilon(\beta_{j,j-1,\omega}+\beta_{j,j+1,\omega}) +o_{\varepsilon\to 0}(\varepsilon)}=0 
    \end{align*}
    \Change{uniformly over $\omega \in \Omega$ away from a $\mathbb{P}$-null set} as \hyperref[list:P4]{\textbf{(P4)}} ensures that \sloppy $\beta_{j,j-1,\omega}+\beta_{j,j+1,\omega}\geq \beta^*>0$ for all $\omega\in\Omega$. 
\end{proof}

\begin{lemma}
In the setting of \thrm{thrm:openopprop}, for $\mathbb{P}$-a.e. $\omega\in\Omega$ the leading Lyapunov multipliers of $\mathcal{L}_{j,\omega}^\varepsilon$ satisfy
    $$\lambda_{j,\omega}^\varepsilon=1-\varepsilon(\beta_{j,j-1,\omega}+\beta_{j,j+1,\omega}) +o_{\varepsilon\to 0}(\varepsilon).$$
    \label{lem:lambda_j}
\end{lemma}
\begin{proof}
    We verify the conditions of \cite[Theorem 2.1.2]{thermoformalism}  \Change{uniformly over $\omega\in\Omega$ away from a $\mathbb{P}$-null set}, calling such conditions \textbf{($\mathcal{Q}$1)}-\textbf{($\mathcal{Q}$9)} instead of \textbf{(P1)}-\textbf{(P9)} for presentation purposes. We only verify conditions \textbf{($\mathcal{Q}$5)}-\textbf{($\mathcal{Q}$9)} as \textbf{($\mathcal{Q}$1)}-\textbf{($\mathcal{Q}$4)} follow immediately from \lem{lem:ULY_open} and \lem{lem:crim_open}. For \textbf{($\mathcal{Q}$5)}, we observe that thanks to \lem{lem:delta-eta}, $\lim_{\varepsilon\to 0}\eta_{j,\omega}^\varepsilon = 0$ \Change{uniformly over $\omega\in\Omega$ away from a $\mathbb{P}$-null set}. Using \lem{lem:delta-eta}, we find that \textbf{($\mathcal{Q}$6)} holds since for $\mathbb{P}$-a.e. $\omega\in\Omega$, there exists $M,\beta^*>0$ \jp{such that for $\varepsilon>0$ sufficiently small}
\begin{align*}
    \limsup_{\varepsilon\to 0} \frac{\eta_{j,\omega}^\varepsilon}{\Delta_{j,\omega}^\varepsilon}&\leq 
\limsup_{\varepsilon\to 0} \frac{M}{\beta_{j,j-1,\omega}+\beta_{j,j+1,\omega}+o_{\varepsilon\to 0}(1)} \leq \Change{\frac{M}{\beta^*}}<\infty.
\end{align*}
\textbf{($\mathcal{Q}$7)} follows from \cor{cor:meas_conv} as $\esssup_{\omega\in\Omega}\|\nu_{j,\omega}^\varepsilon - \leb_j\|_{\BV^*(I_j)}\geq \esssup_{\omega\in\Omega}|\nu_{j,\omega}^\varepsilon(\phi_j) - \leb_j(\phi_j)|$ where $\leb_j(\phi_j)=1$. For \textbf{($\mathcal{Q}$8)}, we observe that for $\mathbb{P}$-a.e. $\omega\in\Omega$
\begin{align*}
    (\Delta_{j,\omega}^\varepsilon)^{-1} \leb_j((\mathcal{L}_j^0 - \mathcal{L}_{j,\omega}^\varepsilon)(Q_{j,\sigma^{-n}\omega}^{\varepsilon\, (n)}\phi_j))&\leq \frac{\|Q_{j,\sigma^{-n}\omega}^{\varepsilon\, (n)}\phi_j\|_{\BV(I_j)}}{\Delta_{j,\omega}^\varepsilon} \int_{I_j}(\mathcal{L}_j^0 - \mathcal{L}_{j,\omega}^\varepsilon)(\mathds{1}_{I_j})(x)\, \dleb(x)\\
    &\leq \frac{C\theta^n\|\phi_j\|_{\BV(I_j)}}{\Delta_{j,\omega}^\varepsilon}\int_{I_j}(\mathcal{L}_j^0 - \mathcal{L}_{j,\omega}^\varepsilon)(\mathds{1}_{I_j})(x)\, \dleb(x)\\
    &=\frac{C\theta^n\|\phi_j\|_{\BV(I_j)}}{\mu_j(H_{j,\omega}^\varepsilon)} {\leb_j(H_{j,\omega}^\varepsilon)}\\
    &\leq \frac{C\theta^n\|\phi_j\|_{\BV(I_j)}}{L+o_{\varepsilon\to 0}(1)}.
\end{align*}
Here $L>0$ is the positive constant appearing in \eqref{eqn:leb_est}. In the above we have used \lem{lem:crim_open}(d) to estimate $\|Q_{j,\sigma^{-n}\omega}^{\varepsilon\, (n)}\phi_j\|_{\BV(I_j)}$, \eqref{eqn:delta} to express $\Delta_{j,\omega}^\varepsilon$, and a similar computation to that made in \lem{lem:delta-eta} to evaluate $\leb_j((\mathcal{L}_j^0 - \mathcal{L}_{j,\omega}^\varepsilon)(\mathds{1}_{I_j}))$. Since $\theta\in(0,1)$ and $\phi_j\in\BV(I_j)$, taking $\varepsilon\to 0$ and then $n\to \infty$, we obtain \textbf{($\mathcal{Q}$8)} \Change{uniformly over $\omega\in\Omega$ away from a $\mathbb{P}$-null set}. We conclude by verifying that \textbf{($\mathcal{Q}$9)} holds \Change{uniformly over $\omega\in\Omega$ away from a $\mathbb{P}$-null set}, however, this was done in \lem{lem:qk}. Therefore, using \lem{lem:delta-eta}, \lem{lem:qk} and \hyperref[list:I4]{\textbf{(I4)}}, \cite[Theorem 2.1.2]{thermoformalism} asserts that (in our notation) for $\mathbb{P}$-a.e. $\omega\in\Omega$, 
\begin{align}
    1-\lambda_{j,\omega}^\varepsilon &= \Delta_{j,\omega}^\varepsilon\left(1-\sum_{k=0}^\infty(\lambda_j^{0\, (k+1)})^{-1}q_{j,\omega}^{0\, (k)}+o_{\varepsilon\to 0}(1)\right) \label{eqn:thrm-app} \\
    &=(\varepsilon(\beta_{j,j-1,\omega}+\beta_{j,j+1,\omega}) +o_{\varepsilon\to 0}(\varepsilon))(1+o_{\varepsilon\to 0}(1))\nonumber.\label{eqn:omerr}
\end{align}
\Change{Since \textbf{($\mathcal{Q}$1)}-\textbf{($\mathcal{Q}$9)} hold uniformly over $\omega\in\Omega$ away from a $\mathbb{P}$-null set, the error appearing in \eqref{eqn:thrm-app} can be made independent of $\omega\in\Omega$}. Rearranging for $\lambda_{j,\omega}^\varepsilon$, the result follows.
\end{proof}
\jp{For $i,j\in\{1,\cdots,m\}$, recall the notation $\bar{\beta}_{i,j}:=\int_\Omega \beta_{i,j,\omega}\, d\mathbb{P}(\omega)$. \jpt{Further, as in \cite{gtp_met}, throughout this paper when upper and lower indices of products and summations take non-integer values, or iterates of a map/operator, we consider the indices/iterates integer part.}}
\begin{corollary}
In the setting of \thrm{thrm:openopprop}, fix $t>0$, then for $\mathbb{P}$-a.e. $\omega\in\Omega$
$$\lambda_{j,\omega}^{\varepsilon\, (\frac{t}{\varepsilon})}= e^{-t(\bar{\beta}_{j,j-1}+\bar{\beta}_{j,j+1}+o_{\omega,\varepsilon\to 0}(1))}.$$   
\label{cor:lambda_iter}
\end{corollary}
\begin{proof}
This follows by Birkhoff's ergodic theorem. Indeed,
\begin{align*}
    \frac{1}{t}\log\left(\lambda_{j,\omega}^{\varepsilon\, (\frac{t}{\varepsilon})} \right)&=\frac{1}{t}\sum_{i=0}^{\frac{t}{\varepsilon}-1}\log(\lambda_{j,\sigma^i\omega}^\varepsilon)\\
    &=\frac{1}{t}\sum_{i=0}^{\frac{t}{\varepsilon}-1}\log(1-\varepsilon(\beta_{j,j-1,\sigma^i\omega}+\beta_{j,j+1,\sigma^i\omega}+o_{\varepsilon\to 0}(1)))\\
    &=\frac{\varepsilon}{t}\sum_{i=0}^{\frac{t}{\varepsilon}-1}-(\beta_{j,j-1,\sigma^i\omega}+\beta_{j,j+1,\sigma^i\omega}+o_{\varepsilon\to 0}(1)).
\end{align*}
\Change{\jp{In the last line we have used Taylor's theorem}, along with \hyperref[list:P4]{\textbf{(P4)}} to establish a uniform error over $\omega\in\Omega$.} Thus, by Birkhoff's ergodic theorem, for $\mathbb{P}$-a.e. $\omega\in\Omega$ 
$$\lim_{\varepsilon\to 0}\frac{1}{t}\log\left(\lambda_{j,\omega}^{\varepsilon\, (\frac{t}{\varepsilon})} \right)= -\left(\bar{\beta}_{j,j-1}+\bar{\beta}_{j,j+1}\right)$$
meaning that for $\mathbb{P}$-a.e. $\omega\in\Omega$  
$$\lambda_{j,\omega}^{\varepsilon\, (\frac{t}{\varepsilon})}= e^{-t(\bar{\beta}_{j,j-1}+\bar{\beta}_{j,j+1}+o_{\omega,\varepsilon\to 0}(1))}.$$ 
\end{proof}
\label{sec:open}
\section{The jump processes}
\label{sec:jump}
In this section, we take advantage of the results of \Sec{sec:open} to study the asymptotics of the system's \textit{jump process} \Change{associated with the partition $I_1,\dots,I_m$} for $\varepsilon$ small. We relate the distribution of jumps of the maps to those of an averaged Markov jump process. Our arguments follow those of \cite{SD}, adapting them to the random setting. 
\subsection{The averaged Markov jump process}
\label{sec:avg_markov}
We begin by introducing the relevant quantities to define the jump process of interest. Consider the $m$-state Markov chains in random environments driven by $\sigma:\Omega\to \Omega$, with transition matrices $(M_\omega^\varepsilon)_{\omega\in\Omega}$ where,
 \small 
\begin{equation}
M_{\omega}^{{\varepsilon}}:=\begin{pmatrix}
    1-{\varepsilon} \beta_{1,2,\omega} & {\varepsilon}\beta_{1,2,\omega} & 0 & 0 &  \cdots & \cdots & 0 & 0 \\
    {\varepsilon} \beta_{2,1,\omega} & 1-{\varepsilon}(\beta_{2,1,\omega}+\beta_{2,3,\omega}) & {\varepsilon} \beta_{2,3,\omega} & 0 & \cdots  & \cdots & \vdots & \vdots \\ 
    \vdots & \vdots & \vdots & \vdots & \vdots & \vdots & \vdots &\vdots \\
    0 & \cdots & \cdots & \cdots & \cdots  & 0 & {\varepsilon}\beta_{m,m-1,\omega} & 1-{\varepsilon}\beta_{m,m-1,\omega} 
\end{pmatrix}. \label{eqn:M-matrix}
\end{equation}\normalsize
Thanks to \rem{rem:neighbour}, for $i,j\in\{1,\cdots ,m\}$, $(M_{\omega}^{{\varepsilon}})_{ij}$ describes the one-step transition probabilities for the $m$-state Markov chain in a random environment, driven by $\sigma:\Omega\to \Omega$, for the map $T_\omega^\varepsilon$ from $I_i$ to $I_j$. \Change{Let $\bar{M}^\varepsilon:=\int_\Omega M_{\omega}^\varepsilon\, d\mathbb{P}(\omega)$ be the averaged $m$-state Markov chain and consider the matrix $\bar{G}\in M_{m\times m}(\mathbb{R})$ where $(\bar{G})_{ij}=\bar{\beta}_{i,j}$ for $i\neq j$, $(\bar{G})_{ii}=-\left( \bar{\beta}_{i,i-1}+\bar{\beta}_{i,i+1}\right)$. We will call $\bar{G}$ the \textit{generator} for the \textit{averaged Markov jump process}.}
\begin{remark}
\jp{Given a matrix $A\in M_{m\times m}(\mathbb{R})$, let $\mathrm{diag}(A)$ denote the diagonal matrix formed from the diagonal entries of $A$.} Note that the matrix defined in \eqref{eqn:M-matrix} is the transpose of that studied in \cite[Section 7]{gtp_met}. However, thanks to \cite[Theorem 7.2]{gtp_met} and \cite[Remark 7.6]{gtp_met}, the limiting invariant measures of the
$m$-state Markov chains in random environments driven by $\sigma :\Omega \to \Omega$ can still be determined as $\varepsilon \to 0$. For $\mathbb{P}$-a.e. $\omega\in\Omega$, this is given by the solution $p=\begin{pmatrix}p_1 & p_2 &\cdots &p_m  
\end{pmatrix}^T$ to
\begin{align*}
    \left(I+\mathrm{diag}(\bar{G})^{-1}(\bar{G}^T-\mathrm{diag}(\bar{G}))\right)p=(\mathrm{diag}(\bar{G})^{-1}\bar{G}^T)p=0
    \iff p^T\bar{G} =0 
\end{align*}
that satisfies $\sum_{i=1}^mp_i=1$ with $p_i\geq 0$ for each $i=1,\dots, m$. Note that $\mathrm{diag}(\bar{G})^{-1}$ exists \jp{since \hyperref[list:P4]{\textbf{(P4)}} implies that all the diagonal entries of $\mathrm{diag}(\bar{G})^{-1}$ are bounded below by $\beta^*>0$ for all $\omega\in\Omega$.}
\end{remark}
\Change{We now \Change{introduce the Markov jump process of concern}}. See \cite[Chapter 2]{jumps_YZ} or \cite[Chapter 2]{MC_N} for further details on Markov jump processes. We consider a similar setup to that of \cite{SD} and record it here for the reader's convenience. \\
\\
\Change{Consider a finite state continuous time stochastic process $(X_t)_{t \geq 0}\subset \{1,\cdots, m\}^{[0,\infty)}$, whose evolution is governed by $(P(t))_{t\geq 0}:=(e^{t \bar{G}})_{t\geq 0}$, where in our setting, $\bar{G}:= \frac{\bar{M}^\varepsilon - I}{\varepsilon}$. Set $t_0^M:=0$, and for $i>0$, let $t_i^M = \inf\{ t> t_{i-1}^M \ | \ X_t\neq X_{t_{i-1}^M}\}$. For $i\geq 1$, we call $\mathcal{T}_i^M = t_i^M - t_{i-1}^M$ the \textit{holding times} for $(X_t)_{t\geq 0}$. Let $z_i^M$ denote the state of the process following the $i^{\mathrm{th}}$ transition, that is, $z_i^M:=X_{t_i^M}$. We call the discrete time process $(z_i^M)_{i\in\mathbb{N}}\subset \{1,\cdots,m\}^{\mathbb{N}}$ the averaged Markov jump process. We highlight that $(X_t)_{t\geq 0}$ is a continuous time process and $(z_i^M)_{i\in\mathbb{N}}$ is a discrete time process taking values from $(X_t)_{t\geq 0}$.} \\ 
\\
For $j\in\{1,\cdots, m\}$, let $\mathbb{P}^j$ denote the probability measure constructed on $\{1,\cdots, m\}^{[0,\infty)}$ with the initial condition $z_0^M=j$ that is \Change{evolved by $P(t)$}. For $i\geq 1$, if $z_{i-1}^M=j$ for some $j\in\{1,\cdots, m\}$, then the holding times $\mathcal{T}_i^M$ are exponentially distributed random variables, more precisely, for $l\in\{1,\cdots, m\}$
$$\mathbb{P}^l(\mathcal{T}_i^M\in B \ |\ z_{i-1}^M = j)=\left( \bar{\beta}_{j,j-1}+\bar{\beta}_{j,j+1}\right) \int_B\,e^{-t\left( \bar{\beta}_{j,j-1}+\bar{\beta}_{j,j+1}\right)} dt$$
for a Borel set $B\subset[0,\infty)$ and $j\in\{1,\cdots,m\}$. Furthermore, given $z_{i-1}^M=j$, the probability of jumping to $z_{i}^M=k$ is independent of $\mathcal{T}_i^M$ where for $l\in\{1,\cdots, m\}$,
$$\mathbb{P}^l(z_{i}^M=k \ |\ z_{i-1}^M = j)= \frac{\bar{\beta}_{j,k}}{\bar{\beta}_{j,j-1}+\bar{\beta}_{j,j+1}}.$$
In this section, we aim to illustrate that the distributions of jumps \Change{between different intervals $I_j,I_k\in \{I_1,\cdots, I_m\}$} for the random maps $(T_\omega^\varepsilon)_{\omega\in\Omega}$ may be approximated by the deterministic distributions of jumps of the averaged Markov jump process.
\subsection{Approximation of jumps for random metastable systems}
\Change{Fix $\varepsilon>0$}. We now introduce the jump process for the collection of random maps $(T_\omega^\varepsilon)_{\omega\in\Omega}$. Let $t_{0,\omega}^\varepsilon(x):=0$ for all $\omega\in\Omega, x\in I$ and $\varepsilon>0$. Define the map $z:I\to \{1,\cdots,m\}$ such that $z(x)=j$ if $x\in I_j$. For $i> 0$ we let $t_{i,\omega}^\varepsilon(x):=\inf\{ n> t_{i-1,\omega}^\varepsilon(x) \ | \ z(T_\omega^{\varepsilon\, (n)}(x))\neq z(T_\omega^{\varepsilon\, (t_{i-1,\omega}^\varepsilon(x))}(x))\}$. One can interpret the function $t_{i,\omega}^\varepsilon:I\to \mathbb{N}$ as the $i^\mathrm{th}$ iteration at which the initial condition $x\in I$ has jumped from $I_j$ to $I_k$ for some $j\neq k$. Finally, for $i\geq 1$ we define the \textit{holding times} for the random maps as $\mathcal{T}_{i,\omega}^\varepsilon(x):=t_{i,\omega}^\varepsilon(x) - t_{i-1,\omega}^\varepsilon(x)$. \\
\\
The main result of this section is \thrm{thrm:conv_jump} whose proof relies on the following lemmata. 
\begin{lemma}
    Fix $t>0$. Let $(\Omega,\mathcal{F},\mathbb{P})$ be a probability space. For $j,k\in\{1,\cdots, m\}$, let  $\sigma:\Omega\to \Omega$ be as in \hyperref[list:P1]{\textbf{(P1)}} and $\beta_{j,k}$ be as in \hyperref[list:P4]{\textbf{(P4)}}. Then for $\mathbb{P}$-a.e. $\omega\in\Omega$
    \begin{align}
        &\lim_{\varepsilon\to 0} \sum_{n=0}^{{\frac{t}{\varepsilon}}-1} (\varepsilon\beta_{j,j+1,\sigma^n\omega}+o_{\varepsilon\to 0}(\varepsilon))\prod_{k=0}^{n-1}(1-\varepsilon(\beta_{j,j-1,\sigma^{k}\omega}+\beta_{j,j+1,\sigma^{k}\omega}+o_{\varepsilon\to 0}(1))) \label{eqn:expdist}\\
        &\quad  = \frac{\bar{\beta}_{j,j+1}}{\bar{\beta}_{j,j-1}+\bar{\beta}_{j,j+1}}\left(1-e^{-t\left(\bar{\beta}_{j,j-1}+\bar{\beta}_{j,j+1}\right)}\right)\nonumber.     
    \end{align}
    \label{lem:exp_dist}
\end{lemma}
\begin{remark}
    The proof of \lem{lem:exp_dist} relies on \cite[Lemma 5.3]{gtp_met}. However, note that the summand in \eqref{eqn:expdist} is different to that from \cite[Lemma 5.3]{gtp_met}. In particular, in \eqref{eqn:expdist}, \Change{an additional error of order $o_{\varepsilon\to 0}(1)$} appears in the products. Fortunately, \Change{the effect of this error} may be controlled since, due to \hyperref[list:P4]{\textbf{(P4)}}, \Change{for $\mathbb{P}$-a.e. $\omega\in \Omega$} one can write 
    \begin{align*}
     \prod_{k=0}^{n-1}(1-\varepsilon(\beta_{j,j-1,\sigma^{k}\omega}+\beta_{j,j+1,\sigma^{k}\omega}+o_{\varepsilon\to 0}(1)))  
     &=\prod_{k=0}^{n-1}(1-\varepsilon(\beta_{j,j-1,\sigma^{k}\omega}+\beta_{j,j+1,\sigma^{k}\omega})) + o_{\varepsilon\to 0}(\varepsilon).
    \end{align*}
    So, using \hyperref[list:P4]{\textbf{(P4)}}, there exists $M>0$ such that for $\mathbb{P}$-a.e. $\omega\in\Omega$ the error term from \eqref{eqn:expdist} arising from the product is
    \begin{align*}
    o_{\varepsilon\to 0}(\varepsilon)\sum_{n=0}^{{\frac{t}{\varepsilon}}-1} (\varepsilon\beta_{j,j+1,\sigma^n\omega}+o_{\varepsilon\to 0}(\varepsilon)) &\leq o_{\varepsilon\to 0}(\varepsilon)(\varepsilon M+o_{\varepsilon\to 0}(\varepsilon))\left(\frac{t}{\varepsilon}\right)=o_{\varepsilon\to 0}(1).
    \end{align*}
    \label{rem:prod-error}
\end{remark}
\jp{Below we discuss the key steps of the proof of  \lem{lem:exp_dist} and refer the reader to \cite{gtp_met} for further details.}
\begin{proof}[Proof of \lem{lem:exp_dist}]
Thanks to \rem{rem:prod-error}, for each $j\in\{1,\cdots,m\}$ it suffices to estimate \begingroup \allowdisplaybreaks
\begin{align*}
    R_{j,t,\omega}^\varepsilon&:=\sum_{n=0}^{{\frac{t}{\varepsilon}}-1} (\varepsilon\beta_{j,j+1,\sigma^n\omega}+o_{\varepsilon\to 0}(\varepsilon))\prod_{k=0}^{n-1}(1-\varepsilon(\beta_{j,j-1,\sigma^{k}\omega}+\beta_{j,j+1,\sigma^{k}\omega}))\\
    &=\sum_{n=0}^{\sqrt{{\frac{t}{\varepsilon}}}-1} (\varepsilon\beta_{j,j+1,\sigma^n\omega}+o_{\varepsilon\to 0}(\varepsilon))\prod_{k=0}^{n-1}(1-\varepsilon(\beta_{j,j-1,\sigma^{k}\omega}+\beta_{j,j+1,\sigma^{k}\omega}))\\
    &\quad+\sum_{n=\sqrt{{\frac{t}{\varepsilon}}}}^{{{\frac{t}{\varepsilon}}}-1} (\varepsilon\beta_{j,j+1,\sigma^n\omega}+o_{\varepsilon\to 0}(\varepsilon))\prod_{k=0}^{n-1}(1-\varepsilon(\beta_{j,j-1,\sigma^{k}\omega}+\beta_{j,j+1,\sigma^{k}\omega}))\\
    &=: N_{j,t,\omega}^\varepsilon+U_{j,t,\omega}^\varepsilon.
\end{align*}\endgroup
\jp{We split the initial sum into two smaller sums with the intention to control $N_{j,t,\omega}^\varepsilon$, and use mixing arguments to sharply bound $U_{j,t,\omega}^\varepsilon$.} Observe that in our notation, by \jp{an identical argument} to that made at Step 4 and Step 7 in the proof of Lemma 5.3 of \cite{gtp_met}, one can show that for $\mathbb{P}$-a.e. $\omega\in\Omega$ $\lim_{\varepsilon\to 0}N_{j,t,\omega}^\varepsilon=0$. Thus, it remains to estimate $U_{j,t,\omega}^\varepsilon$. However, in our notation, aside from the fibres we sum over, $U_{j,t,\omega}^\varepsilon$ is identical to the function appearing at Step 5 in the proof of Lemma 5.3 of \cite{gtp_met}. Thus, by Step 7 of the same proof (which relies on Step 6), we find that for a fixed $t,\delta>0$, for $\mathbb{P}$-a.e. $\omega \in \Omega$,
\begin{align}
     \lim_{\varepsilon\to 0}R_{j,t,\omega}^\varepsilon&\geq\delta e^{-\delta \left(\bar{\beta}_{j,j-1}+\bar{\beta}_{j,j+1}\right)}\bar{\beta}_{j,j+1}\frac{1-e^{-t\left(\bar{\beta}_{j,j-1}+\bar{\beta}_{j,j+1}\right)}}{1-e^{-\delta\left(\bar{\beta}_{j,j-1}+\bar{\beta}_{j,j+1}\right) }} 
     \label{eqn:lower}
     \end{align}
and
     \begin{align}
    \lim_{\varepsilon\to 0}R_{j,t,\omega}^\varepsilon&\leq\delta \bar{\beta}_{j,j+1} \frac{1-e^{-t\left(\bar{\beta}_{j,j-1}+\bar{\beta}_{j,j+1}\right)}}{1-e^{-\delta\left(\bar{\beta}_{j,j-1}+\bar{\beta}_{j,j+1}\right)}}.\label{eqn:upper}
\end{align}
 As in the proof of Step 8 in \cite[Lemma 5.3]{gtp_met}, taking $\delta$ small in \eqref{eqn:lower} and \eqref{eqn:upper}, the result follows.    
\end{proof}
Next, we state a random analogue to the \textit{Growth Lemma} in \cite[Lemma 2]{SD}. Its proof follows by a similar argument to that made in \cite[Lemma 4]{RGL_BE}. 
\\
\\
For $J\subset I$, let $r_{n,\omega}(x)$ be the distance of $T_\omega^{\varepsilon \, (n)} (x)$ to the boundary of $T_\omega^{\varepsilon \, (n)} (J)$ containing it. The following result allows us to study the set of points that map near, and far from the boundary of the hole $H_{j,\sigma^n\omega}^\varepsilon$ at time $n$. 
\begin{lemma}
    There exists $c>0$ such that for all $\varepsilon>0$ sufficiently small, $J\subset I$, $\omega\in\Omega$ and $n\in\mathbb{N}$
    $$\leb(\{ x\, | \ r_{n,\omega}(x)\leq \varepsilon\})\leq \leb(\{x \ | \ \Lambda^n r_{0,\omega}(x)\leq \varepsilon\})+c\varepsilon\leb(J),$$
    \label{lem:growth}
    where $\Lambda>1$ is as in \hyperref[list:I2]{\textbf{(I2)}}. 
\end{lemma}
\begin{proof} 
Following an identical argument to that made in \cite[Lemma 4]{RGL_BE}, \jp{keeping track of the randomness}, the result follows. Note that in our setting, thanks to \hyperref[list:P1]{\textbf{(P1)}}, for all $\varepsilon\geq 0$, the mapping $\omega\mapsto T_\omega^\varepsilon$ has finite range, allowing us to show that $c>0$ is independent of $\omega\in\Omega.$ 
\end{proof}\jp{
\begin{definition}
Fix $r>0$ and $j\in\{1,\cdots, m\}$. We say that a visit of $x$ to the hole $H_{j,\sigma^n\omega}^\varepsilon$ at time $n$ is \textit{$r$-inessential} if $d(T_{\omega}^{\varepsilon\, (n)}(x),\partial H_{j,\sigma^n\omega}^\varepsilon)<r\varepsilon$. The visit is \textit{$r$-essential} otherwise.
    \label{def:r_inessential}
\end{definition}}
Fix $S,r>0$ and $j\in\{1,\cdots,m\}$. The following result shows that for all $\varepsilon>0$ the probability that $x$ will have an $r$-inessential visit to $H_{j,\sigma^n\omega}^\varepsilon$ from time $n=0,\dots, S/\varepsilon$ can be made small.
\begin{lemma}
    Fix $S,r>0$ and $j\in\{1,\cdots, m\}$. For all $\varepsilon>0$ sufficiently small there exists a constant $C>0$ such that 
    $$ \sum_{n=0}^{\frac{S}{\varepsilon}}\leb_j(\{ x\, | \ r_{n,\omega}(x)\leq r \varepsilon\})\leq CrS$$
     \Change{where $r_{n,\omega}(x)$ is the distance of $T_\omega^{\varepsilon\, (n)}(x)$ to $\partial H_{j,\sigma^n\omega}^\varepsilon$.}
    \label{lem:inessential}
\end{lemma}
\begin{proof}
Due to \lem{lem:growth} with $J=H_{j,\sigma^n\omega}^\varepsilon$, there exists $c>0$ such that for all $\varepsilon>0$ small enough,
\begin{align*}
     \sum_{n=0}^{\frac{S}{\varepsilon}}\leb_j(\{ x\, | \ r_{n,\omega}(x)\leq r \varepsilon\})&\leq \sum_{n=0}^{\frac{S}{\varepsilon}} \left(\leb_j(\{x \ | \ \Lambda^n r_{0,\omega}(x)\leq r \varepsilon\})+cr\varepsilon\leb_j(H_{j,\sigma^n \omega}^\varepsilon)\right)\\
        &\leq \tilde{c}r\varepsilon \left(\frac{S}{\varepsilon}+1\right)+\sum_{n=0}^{\frac{S}{\varepsilon}}\leb_j(\{x \ | \ \Lambda^n r_{0,\omega}(x)\leq r \varepsilon\})\\
        &\stackrel{(\star)}= \tilde{c}r\varepsilon \left(\frac{S}{\varepsilon}+1\right)+\leb_j(\{x \ | \ r_{0,\omega}(x)\leq r \varepsilon\}) \sum_{n=0}^{\frac{S}{\varepsilon}}\Lambda^{-n}\\
        &\stackrel{(\star\star)}\leq kr\varepsilon \left(\frac{S}{\varepsilon}+1\right)\\
        &\leq CrS.
\end{align*}
\Change{At $(\star)$ we have used \hyperref[list:I2]{\textbf{(I2)}} as $\Lambda>1$ scales $\leb_j(\{x \ | \ \Lambda^n r_{0,\omega}(x)\leq r \varepsilon\})$. At $(\star \star)$, since \lem{lem:growth} is applied to $J=H_{j,\sigma^n\omega}^\varepsilon$, we apply \lem{lem:leb_estimate} to deduce that for $\varepsilon>0$ sufficiently small, there exists a constant $\tilde{k}>0$ such that $\leb_j(\{x \ | \ r_{0,\omega}(x)\leq r \varepsilon\}) \leq \tilde{k}r\varepsilon$.} 
\end{proof}
We now study trajectories that map far from the boundary of $H_{j,\sigma^n\omega}^\varepsilon$ at time $n$. In light of \dfn{def:r_inessential}, we define $\mathcal{E}_{j,n,\omega}^\varepsilon=\{x \ | \ x \ \mathrm{has \ an} \ r\mathrm{-essential \ visit \ to} \ H_{j,\sigma^n\omega}^\varepsilon \ \mathrm{at \ time} \ n \}$. That is, for each $\omega\in\Omega$ $\mathcal{E}_{j,n,\omega}^\varepsilon:=(T_{\omega}^{\varepsilon\, (n)})^{-1}(\tilde{H}_{j,\sigma^n\omega}^\varepsilon)$ where $\tilde{H}_{j,\sigma^n\omega}^\varepsilon\subset {H}_{j,\sigma^n\omega}^\varepsilon$ is the set of points in ${H}_{j,\sigma^n\omega}^\varepsilon$ that are greater than $r\varepsilon$ away from the boundary of ${H}_{j,\sigma^n\omega}^\varepsilon$.
\begin{lemma}
    Fix $j\in \{1,\cdots, m\}$. Then for all $\varepsilon>0$ sufficiently small and $\mathbb{P}$-a.e. $\omega\in\Omega$ there exists constants $\tilde{m},\tilde{M}>0$ such that for all \Change{$n\in\mathbb{N}$}
    $$ \tilde{m}\varepsilon\leq \leb_j(\mathcal{E}_{j,n,\omega}^\varepsilon)\leq \tilde{M}\varepsilon.$$
    \label{lem:lebj_ess}
\end{lemma}
\begin{proof}
    For the upper bound, recall that $\mathcal{E}_{j,n,\omega}^\varepsilon=(T_{\omega}^{\varepsilon\, (n)})^{-1}(\tilde{H}_{j,\sigma^n\omega}^\varepsilon)$ where $\tilde{H}_{j,\sigma^n\omega}^\varepsilon\subset {H}_{j,\sigma^n\omega}^\varepsilon$. Thus, due to \lem{lem:leb_estimate}, there exists an $\tilde{M}>0$ such that for $\varepsilon>0$ sufficiently small, $\leb_j(\mathcal{E}_{j,n,\omega}^\varepsilon)\leq \leb_j((T_\omega^{\varepsilon\, (n)})^{-1}({H}_{j,\sigma^n\omega}^\varepsilon))\leq \tilde{M}\varepsilon$. Fix $r>0$. For the lower estimate we observe that $$\leb_j(\mathcal{E}_{j,n,\omega}^\varepsilon)= \leb_j(H_{j,\sigma^n\omega}^\varepsilon) -  \leb_j(\{x \ | \ r_{n,\omega}(x) \leq r \varepsilon\}),$$ \Change{where $r_{n,\omega}(x)$ is the distance of $T_\omega^{\varepsilon\, (n)}(x)$ to $\partial H_{j,\sigma^n\omega}^\varepsilon$.} Thus, thanks to \lem{lem:leb_estimate} and \lem{lem:growth} applied to $H_{j,\sigma^n\omega}^\varepsilon$, for $\jp{r},\varepsilon>0$ sufficiently small,
    \begin{align*}
     \leb_j(\mathcal{E}_{j,n,\omega}^\varepsilon)&\geq \leb_j(H_{j,\sigma^n\omega}^\varepsilon) - \leb_j(\{x \ | \ \Lambda^n r_{0,\omega}(x)\leq r\varepsilon\})-c\varepsilon\leb_j(H_{j,\sigma^n\omega}^\varepsilon)\\
     &\geq a\varepsilon - \frac{Cr\varepsilon}{\Lambda^n}-M\varepsilon^2\\
     &\geq \tilde{m}\varepsilon.
    \end{align*}
\end{proof}
\begin{lemma}
    Fix $S,\delta>0$ and $j\in \{1,\cdots,m\}$, then there exists a $\xi>0$ such that for all $\varepsilon>0$ sufficiently small,
    $$\mu_{j}\left(\left\{ x\in I \ | \ \exists k \ \mathrm{with} \ t_{k,\omega}^\varepsilon(x)\leq \frac{S}{\varepsilon} \ \mathrm{and} \ \mathcal{T}_{k+1,\omega}^\varepsilon(x)\leq \frac{\xi}{\varepsilon} \right\} \right)\leq \delta.$$
    \label{lem:no-jump}
\end{lemma}
\begin{proof}
    \jp{We first show that the result holds for $k=0$.} We recall that $t_{0,\omega}^\varepsilon(x)=0$, so due to \thrm{thrm:openopprop}(e) 
    \begin{align*}
        \mu_{j}\left(\left\{ x\in I \ | \  t_{0,\omega}^\varepsilon(x)\leq \frac{S}{\varepsilon} \ \mathrm{and} \ \mathcal{T}_{1,\omega}^\varepsilon(x)\leq \frac{\xi}{\varepsilon} \right\} \right)&\leq 1-\mu_j\left(t_{1,\omega}^\varepsilon(x)>  \frac{\xi}{\varepsilon}\right)\\
        &=1-\int_{I_j}\mathcal{L}_{j,\omega}^{\varepsilon\, ( \frac{\xi}{\varepsilon})}(\phi_j)(x)\, \dleb(x)\\
        &= 1-\nu_{j,\omega}^\varepsilon(\phi_j)\lambda_{j,\omega}^{\varepsilon\, ( \frac{\xi}{\varepsilon})}\int_{I_j}\phi_{j,\jp{\sigma^{\frac{\xi}{\varepsilon}}}\omega}^\varepsilon(x)\, \dleb(x)\\
        &\quad+O_{\varepsilon\to 0}(\theta^{ \frac{\xi}{\varepsilon}}).
    \end{align*}
    But thanks to \cor{cor:lambda_iter}, and \thrm{thrm:openopprop}(b),(c) and (e),
    \begin{align*}
        \lim_{\varepsilon\to 0} \mu_j\left(t_{1,\omega}^\varepsilon(x)>  \frac{\xi}{\varepsilon}\right)&=\leb_j(\phi_j)^2e^{-\xi\left(\bar{\beta}_{j,j-1}+\bar{\beta}_{j,j+1}\right)} \\
        &=1+o_{\xi\to 0}(1).
    \end{align*}
    This implies that $\mu_{j}\left(\left\{ x\in I \ | \ \mathcal{T}_{1,\omega}^\varepsilon(x) \leq \frac{\xi}{\varepsilon}\right\}\right)=o_{\xi\to 0}(1)$. We now prove the result for arbitrary $k$. We first note that since $\mu_j\ll\leb_j$, it suffices to prove the result for $\leb_j$. Set $S_{n,m,\omega}(x):=\sum_{\jp{l}=n+1}^{n+m}\mathds{1}_{ H_{j,\sigma^{\jp{l}}\omega}^\varepsilon}(T_\omega^{\varepsilon\, (\jp{l})}(x))$. Observe that  
    \begin{align*}
        &\leb_{j}\left(\left\{ x\in I \ | \ \exists k \ \mathrm{with} \ t_{k,\omega}^\varepsilon(x)\leq \frac{S}{\varepsilon} \ \mathrm{and} \ \mathcal{T}_{k+1,\omega}^\varepsilon(x)\leq \frac{\xi}{\varepsilon} \right\} \right)\\
        &\qquad= \sum_{n=0}^{\frac{S}{\varepsilon}} \int_{I_j}\mathds{1}_{ H_{j,\sigma^n\omega}^\varepsilon}(T_\omega^{\varepsilon\, (n)}(x))\mathds{1}_{S_{n,\frac{\xi}{\varepsilon},\omega}>0}(x)\, \dleb(x).
    \end{align*}
    Due to \lem{lem:inessential}, the probability that $x$ has an $r$-inessential visit to $H_{j,\sigma^n\omega}^\varepsilon$ at time $n$ can be made small by taking $r$ small enough. Thus, it suffices to show that
    \begin{equation}
        \sum_{n=0}^{\frac{S}{\varepsilon}} \int_{I_j}\mathds{1}_{\mathcal{E}_{j,n,\omega}^\varepsilon}(x)\mathds{1}_{S_{n,\frac{\xi}{\varepsilon},\omega}>0}(x)\, \dleb(x)
        \label{eqn:summand}
    \end{equation}
    tends to zero as $\varepsilon\to 0$ and then as $\xi \to 0$. We note that each term in \eqref{eqn:summand} is \Change{of the form} 
    $$\int_{I_j}\mathds{1}_{\mathcal{E}_{j,n,\omega}^\varepsilon}(x)\mathds{1}_{S_{n,\frac{\xi}{\varepsilon},\omega}>0}(x)\, \dleb(x) = \leb_j(\mathcal{E}_{j,n,\omega}^\varepsilon) \leb_j(S_{n,\frac{\xi}{\varepsilon},\omega}>0 \ | \ \mathcal{E}_{j,n,\omega}^\varepsilon).$$
    Thus by \lem{lem:lebj_ess}, 
    \begin{align*}
        \sum_{n=0}^{\frac{S}{\varepsilon}}\int_{I_j}\mathds{1}_{\mathcal{E}_{j,n,\omega}^\varepsilon}(x)\mathds{1}_{S_{n,\frac{\xi}{\varepsilon},\omega}>0}(x)\, \dleb(x) &\leq \max_{n\leq \frac{S}{\varepsilon}}\leb_j(S_{n,\frac{\xi}{\varepsilon},\omega}>0 \ | \ \mathcal{E}_{j,n,\omega}^\varepsilon) \sum_{\jp{r}=0}^{\frac{S}{\varepsilon}}\leb_j(\mathcal{E}_{j,\jp{r},\omega}^\varepsilon)\\
        &\leq \tilde{M}\varepsilon\left( \frac{S}{\varepsilon}+1 \right) \max_{n\leq \frac{S}{\varepsilon}}\leb_j(S_{n,\frac{\xi}{\varepsilon},\omega}>0 \ | \ \mathcal{E}_{j,n,\omega}^\varepsilon).
    \end{align*}
    We conclude by showing that $\max_{n\leq \frac{S}{\varepsilon}}\leb_j(S_{n,\frac{\xi}{\varepsilon},\omega}>0 \ | \ \mathcal{E}_{j,n,\omega}^\varepsilon)$ is small. Indeed, we note that for a fixed $M_0$, $S_{n,M_0,\omega}(x)=0$ for all $x\in \mathcal{E}_{j,n,\omega}^\varepsilon$ due to \hyperref[list:I5]{\textbf{(I5)}}, provided $\varepsilon$ is small.\footnote{Note that \hyperref[list:I5]{\textbf{(I5)}} implies that for every $\jp{l}>0$, $\mathcal{C}^0\cap (T^{0\, (\jp{l})})^{-1}(H^0)=\emptyset$. } Let $\jp{l}\geq M_0$ and set  $J=H_{j,\sigma^{n+\jp{l}}\omega}^\varepsilon$ in the statement of \lem{lem:growth}. Through a similar argument to that made in the proof of \lem{lem:inessential}, for a fixed $n\in \mathbb{N}$ there exists $M_1,M_2,C_2,C_3>0$ such that for $\varepsilon>0$ sufficiently small, the probability of landing in $H_{j,\sigma^{n+\jp{l}}\omega}^\varepsilon$ after $n+\jp{l}$ steps is bounded by  
    \begin{align}
        \leb_j({r_{n+\jp{l},\omega}(x)\leq M_1\varepsilon})&\leq \leb_{j}(r_{0,\omega}(x)\leq M_1\frac{\varepsilon}{\Lambda^{n+\jp{l}}})+M_2\varepsilon  \leb_j(H_{j,\sigma^{n+\jp{l}}\omega}^\varepsilon)\nonumber \\&\leq C_2\frac{\varepsilon}{\Lambda^{\jp{l}}}+C_3\varepsilon^2. \label{eqn:in_hole}
    \end{align}
    In the above, we have used \lem{lem:leb_estimate} for $\varepsilon>0$ sufficiently small. So, to conclude, we see that with the same $M_1>0$ as above,
    \begin{align*}
        \sum_{n=0}^{\frac{S}{\varepsilon}}\int_{I_j}&\mathds{1}_{\mathcal{E}_{j,n,\omega}^\varepsilon}(x)\mathds{1}_{S_{n,\frac{\xi}{\varepsilon},\omega}>0}(x)\, \dleb(x)\\ &\leq \tilde{M}\varepsilon\left( \frac{S}{\varepsilon}+1 \right)  \max_{n\leq \frac{S}{\varepsilon}}\leb_j(S_{n,\frac{\xi}{\varepsilon},\omega}>0 \ | \ \mathcal{E}_{j,n,\omega}^\varepsilon) \\
        &=\tilde{M}\varepsilon\left( \frac{S}{\varepsilon}+1 \right)  \max_{n\leq \frac{S}{\varepsilon}}\leb_j(S_{n+M_0,\frac{\xi}{\varepsilon}-M_0,\omega}>0 \ | \ \mathcal{E}_{j,n,\omega}^\varepsilon)\\
        &=\tilde{M}\varepsilon\left( \frac{S}{\varepsilon}+1 \right) \max_{n\leq \frac{S}{\varepsilon}}\sum_{\jp{l}=M_0+1}^{\frac{\xi}{\varepsilon}}\bigg(\leb_j(\mathds{1}_{H_{j,\sigma^{n+\jp{l}}\omega}^\varepsilon}( T_{\omega}^{\varepsilon\, (n+\jp{l})}(x))>0 \ | \ \mathcal{E}_{j,n,\omega}^\varepsilon)\bigg)\\
        &\leq \tilde{M}\varepsilon\left( \frac{S}{\varepsilon}+1 \right) \max_{n\leq \frac{S}{\varepsilon}}\sum_{\jp{l}=M_0+1}^{\frac{\xi}{\varepsilon}}\frac{\leb_j({r_{n+\jp{l},\omega}(x)\leq M_1\varepsilon})}{\leb_j(\mathcal{E}_{j,n,\omega}^\varepsilon)}.
    \end{align*}
    But, thanks to \eqref{eqn:in_hole}, there exists a $K_2>0$ such that for $\varepsilon>0$ sufficiently small, $\leb_j({r_{n+\jp{l},\omega}(x)\leq M_1\varepsilon})\leq K_2\varepsilon(\varepsilon+1/\Lambda^{\jp{l}})$.
    So, applying \lem{lem:lebj_ess} we find that there exists a $C>0$ such that
    \begin{align*}
        \sum_{n=0}^{\frac{S}{\varepsilon}}\int_{I_j}\mathds{1}_{\mathcal{E}_{j,n,\omega}^\varepsilon}(x)\mathds{1}_{S_{n,\frac{\xi}{\varepsilon},\omega}>0}(x)\, \dleb(x) &\leq \tilde{M}\varepsilon\left( \frac{S}{\varepsilon}+1 \right) \sum_{\jp{l}=M_0+1}^{\frac{\xi}{\varepsilon}}\frac{K_2}{\tilde{m}}\left(\varepsilon+\frac{1}{\Lambda^{\jp{l}}}\right)\\
        &\leq C(S+\varepsilon)\sum_{\jp{l}=M_0+1}^{\frac{\xi}{\varepsilon}}\left(\varepsilon+\frac{1}{\Lambda^{\jp{l}}}\right)\\
        &= C(S+\varepsilon)\left(\frac{\Lambda \xi+M_0\varepsilon(1-\Lambda)+\Lambda^{-M_0}-\Lambda^{-\frac{\xi}{\varepsilon}}-\xi}{\Lambda -1}\right).
    \end{align*}
    Thus
    $$\lim_{\varepsilon\to 0}\sum_{n=0}^{\frac{S}{\varepsilon}}\int_{I_j}\mathds{1}_{\mathcal{E}_{j,n,\omega}^\varepsilon}(x)\mathds{1}_{S_{n,\frac{\xi}{\varepsilon},\omega}>0}(x)\, \dleb(x) =\frac{CS}{\Lambda - 1}(\Lambda^{-M_0}-\xi +\Lambda\xi)$$
    which tends to zero by taking $\xi\to 0$ and $M_0$ large.
\end{proof}

We may now prove the main result of this section.
\begin{proof}[Proof of \thrm{thrm:conv_jump}]
The result follows by induction on $p$. We divide the proof into several steps. We begin with \Step{stp:B-lim} and \Step{stp:A-lim}, which allow us to prove the base case when $p=1$. 
\begin{step} Fix $j,r_1\in \{1,\cdots,m\}$ and let $[a_1,b_1]$ be an interval. Then for $\mathbb{P}$-a.e. $\omega\in\Omega$, \jp{the class of functions $(B_{j,r_1,\omega}^\varepsilon)_{\varepsilon>0}$ defined by}
\begin{equation}
    B_{j,r_1,\omega}^\varepsilon:=\nu_{j,\omega}^\varepsilon(\phi_j)O_{\varepsilon\to 0}\left(\esssup_{\omega\in\Omega}\sup_{x\in H_{j,r_1,\omega}^\varepsilon}|\phi_{j,\omega}^\varepsilon(x) - \phi_j(x)| \right)\sum_{n=\frac{a_1}{\varepsilon}}^{\frac{b_1}{\varepsilon}}\lambda_{j,\omega}^{\varepsilon\, (n)} \leb_j({H_{j,r_1,\sigma^n\omega}^\varepsilon})\label{eqn:B}
\end{equation}
satisfies
$$\lim_{\varepsilon\to 0}B_{j,r_1,\omega}^\varepsilon=0.$$
    \label{stp:B-lim}
\end{step}
\begin{proof}
    Recall by \lem{lem:leb_estimate}, there exists an $M>0$ such that for $\mathbb{P}$-a.e. $\omega\in\Omega$ \jp{and $\varepsilon>0$ sufficiently small}, $\leb_j(H_{j,r_1,\sigma^n\omega}^\varepsilon)\leq M\varepsilon$. In turn, for $\mathbb{P}$-a.e. $\omega\in\Omega$, due to \lem{lem:leb_estimate} and \thrm{thrm:openopprop}(a)
\begin{align*}
    B_{j,r_1,\omega}^\varepsilon&\leq \nu_{j,\omega}^\varepsilon(\phi_j) O_{\varepsilon\to 0}\left(\esssup_{\omega\in\Omega}\sup_{x\in H_{j,r_1,\omega}^\varepsilon}|\phi_{j,\omega}^\varepsilon(x) - \phi_j(x)| \right)M\varepsilon\sum_{n=\frac{a_1}{\varepsilon}}^{\frac{b_1}{\varepsilon}}\lambda_{j,\omega}^{\varepsilon\, ({n}{})} \nonumber \\
    &\leq\nu_{j,\omega}^\varepsilon(\phi_j) O_{\varepsilon\to 0}\left(\esssup_{\omega\in\Omega}\sup_{x\in H_{j,r_1,\omega}^\varepsilon}|\phi_{j,\omega}^\varepsilon(x) - \phi_j(x)| \right)M\left({b_1-a_1}+{\varepsilon}\right).\label{eqn:S1ineq} 
\end{align*}
Observe that \thrm{thrm:openopprop}(b) asserts that $\lim_{\varepsilon\to 0} \esssup_{\omega\in\Omega}|\nu_{j,\omega}^\varepsilon(\phi_j) - \leb_j(\phi_j)| = 0$ where $\leb_j(\phi_j)=1$ since $\esssup_{\omega\in\Omega}\|\nu_{j,\omega}^\varepsilon - \leb_j\|_{\BV^*(I_j)}\geq \esssup_{\omega\in\Omega}|\nu_{j,\omega}^\varepsilon(\phi_j) - \leb_j(\phi_j)|$. Thus, by applying \thrm{thrm:openopprop}(d), the result follows.
\end{proof}
\begin{step} Fix $j,r_1\in \{1,\cdots,m\}$ and let $[a_1,b_1]$ be an interval. Then for $\mathbb{P}$-a.e. $\omega\in\Omega$, the function
\begin{equation}
    A_{j,r_1,\omega}^\varepsilon:=\nu_{j,\omega}^\varepsilon(\phi_j)\sum_{n=\frac{a_1}{\varepsilon}}^{\frac{b_1}{\varepsilon}}\lambda_{j,\omega}^{\varepsilon\, (n)}\mu_j(H_{j,r_1,\sigma^n\omega}^\varepsilon)\label{eqn:A}
\end{equation}
satisfies
$$\lim_{\varepsilon\to 0}A_{j,r_1,\omega}^\varepsilon=\bar{\beta}_{j,r_1}\int_{a_1}^{b_1} e^{-t\left(\bar{\beta}_{j,j-1}+\bar{\beta}_{j,j+1}\right)}\, dt.$$
    \label{stp:A-lim}
\end{step}
\begin{proof}
 Due to \thrm{thrm:openopprop}(a) and \hyperref[list:P4]{\textbf{(P4)}} we find that
\begin{align*}
    A_{j,r_1,\omega}^\varepsilon &= \nu_{j,\omega}^\varepsilon(\phi_j) \sum_{n=\frac{a_1}{\varepsilon}}^{{\frac{b_1}{\varepsilon}}{}}(\varepsilon\beta_{j,r_1,\sigma^n \omega}+o_{\varepsilon\to 0}(\varepsilon))\prod_{k=0}^{n-1}(1-\varepsilon(\beta_{j,j-1,\sigma^k\omega}+\beta_{j,j+1,\sigma^k\omega}) +o_{\varepsilon\to 0}(\varepsilon))\\
    &=\nu_{j,\omega}^\varepsilon(\phi_j)\lambda_{j,\omega}^{\varepsilon\, (\frac{a_1}{\varepsilon})} \sum_{n=0}^{{\frac{b_1-a_1}{\varepsilon}}{}}(\varepsilon\beta_{j,r_1,\sigma^{n+\frac{a_1}{\varepsilon}} \omega}+o_{\varepsilon\to 0}(\varepsilon))\prod_{k=0}^{n-1}(1-\varepsilon(\beta_{j,j-1,\sigma^{k+\frac{a_1}{\varepsilon}}\omega}\\
    &\qquad+\beta_{j,j+1,\sigma^{k+\frac{a_1}{\varepsilon}}\omega}) +o_{\varepsilon\to 0}(\varepsilon)).
 \end{align*}
As in \Step{stp:B-lim}, we know that due to \thrm{thrm:openopprop}(b), $\lim_{\varepsilon\to 0}\esssup_{\omega\in\Omega}|\nu_{j,\omega}^\varepsilon(\phi_j)-1|=0$. Further, \cor{cor:lambda_iter} asserts that \Change{$\lambda_{j,\omega}^{\varepsilon\, (\frac{a_1}{\varepsilon})}= \exp(-a_1(\bar{\beta}_{j,j-1}+\bar{\beta}_{j,j+1}+o_{\omega,\varepsilon\to 0}(1)))$}. Thus, recalling \rem{rem:prod-error} and applying \lem{lem:exp_dist} with $t=b_1-a_1-\varepsilon$, 
\begin{align*}
    \lim_{\varepsilon\to 0}A_{j,r_1,\omega}^\varepsilon &= \bar{\beta}_{j,r_1}\int_{a_1}^{b_1} e^{-t\left(\bar{\beta}_{j,j-1}+\bar{\beta}_{j,j+1}\right)}\, dt.
\end{align*}    
\end{proof}
We may now prove the base case.
\begin{step}[The base case]
Fix $j\in\{1,\cdots, m\}$. Take an interval $\Delta_1=[a_1,b_1]$ and a number $r_1\in \{1,\cdots, m\}$. Then for $\mathbb{P}$-a.e. $\omega\in\Omega$,
\begin{align*}
    &\lim_{\varepsilon\to 0}\mu_j\left(\left\{x\in I \ \big| \ \varepsilon\mathcal{T}_{1,\omega}^\varepsilon(x)\in \Delta_1 \ \mathrm{and} \ z(T_{\omega}^{\varepsilon \, (t_{1,\omega}^\varepsilon(x))}(x))=r_1 \right\}  \right)\\
    &\qquad=\bar{\beta}_{j,r_1}\int_{a_1}^{b_1} e^{-t\left(\bar{\beta}_{j,j-1}+\bar{\beta}_{j,j+1}\right)}\, dt.
\end{align*}\label{stp:base}
\end{step}
\begin{proof}
    
Fix $p=1$ and $n\in\mathbb{N}$. Take an interval $\Delta_1=[a_1,b_1]$ and a number $r_1\in \{1,\cdots, m\}$, then
\begin{align*}
    &\mu_j\left(\left\{x\in I \ \big| \ \mathcal{T}_{1,\omega}^\varepsilon(x)=n \ \mathrm{and} \ z(T_{\omega}^{\varepsilon \, (t_{1,\omega}^\varepsilon(x))}(x))=r_1\right\}  \right)\\
    &\qquad= \int_{\left\{y\in I \ \big| \ {t}_{1,\omega}^\varepsilon(y)=n \ \mathrm{and} \ z(T_{\omega}^{\varepsilon \, (t_{1,\omega}^\varepsilon(y))}(y))=r_1\right\}} \phi_j(x)\, \dleb(x)\\
    &\qquad=\int_{H_{j,r_1,\sigma^n\omega}^\varepsilon} \mathcal{L}_{j,\omega}^{\varepsilon\, (n)}(\phi_j)(x)\, \dleb(x).
\end{align*}
By \thrm{thrm:openopprop}(e),
 \begin{align*}
    \int_{H_{j,r_1,\sigma^n\omega}^\varepsilon} \mathcal{L}_{j,\omega}^{\varepsilon\, (n)}(\phi_j)(x)\, \dleb(x)&= \lambda_{j,\omega}^{\varepsilon\, (n)}\nu_{j,\omega}^\varepsilon(\phi_j)\int_{H_{j,r_1,\sigma^n\omega}^\varepsilon}\phi_{j,\sigma^n\omega}^\varepsilon(x)\, \dleb(x) +O_{\varepsilon\to 0}(\theta^n).    
 \end{align*}
Summing over $n\in\Delta_1/\varepsilon$, we find that 
\begin{align}
     \sum_{n=\frac{a_1}{\varepsilon}}^{\frac{b_1}{\varepsilon}}\int_{H_{j,r_1,\sigma^n\omega}^\varepsilon} \mathcal{L}_{j,\omega}^{\varepsilon\, (n)}(\phi_j)(x)\, \dleb(x)&= \nu_{j,\omega}^\varepsilon(\phi_j) \sum_{n=\frac{a_1}{\varepsilon}}^{\frac{b_1}{\varepsilon}}\lambda_{j,\omega}^{\varepsilon\, (n)}\int_{H_{j,r_1,\sigma^n\omega}^\varepsilon}\phi_{j,\sigma^n\omega}^\varepsilon(x) \, \dleb(x) \label{eqn:add-sub} \phantom{sdf}\\
     &\quad +O_{\varepsilon\to 0}\left(\frac{\theta^{\frac{a_1}{\varepsilon}}-\theta^{\frac{b_1}{\varepsilon}+1}}{1-\theta} \right) \nonumber\\ &=A_{j,r_1,\omega}^\varepsilon+B_{j,r_1,\omega}^\varepsilon+O_{\varepsilon\to 0}\left(\frac{\theta^{\frac{a_1}{\varepsilon}}-\theta^{\frac{b_1}{\varepsilon}+1}}{1-\theta} \right) \nonumber
 \end{align}
where $A_{j,r_1,\omega}^\varepsilon$ and $B_{j,r_1,\omega}^\varepsilon$ are given by \eqref{eqn:A} and \eqref{eqn:B}, respectively. In the last line, we have added and subtracted $\phi_j \in \BV(I_j)$ under the integral appearing in \eqref{eqn:add-sub}. One can verify that $\lim_{\varepsilon \to 0}O_{\varepsilon\to 0}\left(\frac{\theta^{\frac{a_1}{\varepsilon}}-\theta^{\frac{b_1}{\varepsilon}+1}}{1-\theta} \right) =0$, and thus, by \Step{stp:B-lim} and \Step{stp:A-lim},
\begin{align*}
    &\lim_{\varepsilon\to 0}\mu_j\left(\left\{x\in I \ \big| \ \varepsilon\mathcal{T}_{1,\omega}^\varepsilon(x)\in \Delta_1 \ \mathrm{and} \ z(T_{\omega}^{\varepsilon \, (t_{1,\omega}^\varepsilon(x))}(x))=r_1 \right\}  \right)\\
    &\qquad=\bar{\beta}_{j,r_1}\int_{a_1}^{b_1} e^{-t\left(\bar{\beta}_{j,j-1}+\bar{\beta}_{j,j+1}\right)}\, dt.
\end{align*}
\end{proof}
For the inductive step, let $p\in\mathbb{N}$ and suppose \thrm{thrm:conv_jump} holds for $k=1,\dots, p$ and define the sets 
\begin{equation}
    \Gamma_{p,\omega}^\varepsilon:= \left\{x\in I \ | \ \varepsilon\mathcal{T}_{k,\omega}^\varepsilon(x)\in \Delta_k \ \mathrm{and} \ z(T_{\omega}^{\varepsilon\, (t_{k,\omega}^\varepsilon(x))}(x))=r_k \ \mathrm{for} \ k=1,\cdots,p\right\}.\label{eqn:Gamma}
\end{equation}
\jp{Set $\textbf{n}_p=(n_1,\cdots,n_p)\in \mathbb{N}^p$ and let}
\jp{\begin{equation}
\mathcal{L}_{\Gamma_{p,\textbf{n}_p,\omega}^\varepsilon}(f):= \mathcal{L}_{\sigma^{n_1+\cdots+n_p-1}\omega}^\varepsilon\left(\mathcal{L}_{r_{p-1},\sigma^{n_1+\cdots+n_{p-1}}\omega}^{\varepsilon\,(n_p-1)}(\mathcal{L}_{\Gamma_{p-1,\textbf{n}_p,\omega}^\varepsilon}(f))\cdot \mathds{1}_{H_{r_{p-1},r_{p},\sigma^{n_1+\cdots+n_p-1}\omega}^\varepsilon} \right)     \label{eqn:gamman_op}
\end{equation}
where for $p=1$, 
\begin{equation}
\mathcal{L}_{\Gamma_{1,\textbf{n}_1,\omega}^\varepsilon}(f):= \mathcal{L}_{\sigma^{n_1-1}\omega}^\varepsilon\left(\mathcal{L}_{j,\omega}^{\varepsilon\,(n_1-1)}(f)\cdot \mathds{1}_{H_{j,r_1,\sigma^{n_1-1}\omega}^\varepsilon} \right). \label{eqn:gamma1_op}
\end{equation}
We define the operator $\mathcal{L}_{\Gamma_{p,\omega}^\varepsilon}:\BV(I)\to \BV(I)$ as
\begin{equation}
    \mathcal{L}_{\Gamma_{p,\omega}^\varepsilon}(f):= \sum_{n_p=\frac{a_p}{\varepsilon}}^{\frac{b_p}{\varepsilon}}\cdots\sum_{n_1=\frac{a_1}{\varepsilon}}^{\frac{b_1}{\varepsilon}} \mathcal{L}_{\Gamma_{p,\textbf{n}_p,\omega}^\varepsilon}(f).\label{eqn:gamma_op}
\end{equation}
It remains to show that \thrm{thrm:conv_jump} holds for $k=p+1$, that is, {for a fixed $j\in\{1,\cdots ,m \}$}, an interval $\Delta_{p+1}=[a_{p+1},b_{p+1}]$, and a number $r_{(p+1)}\in\{1,\cdots, m\}$,  
\begin{align}
    &\lim_{\varepsilon\to 0} \mu_j\left(\left\{x\in I \ | \ \varepsilon\mathcal{T}_{p+1,\omega}^\varepsilon(x)\in \Delta_{p+1} \ \mathrm{and} \ z(T_{\omega}^{\varepsilon\, (t_{p+1,\omega}^\varepsilon(x))}(x))=r_{(p+1)}  \right\}\cap\Gamma_{p,\omega}^\varepsilon \right)\nonumber\\
    &\quad =\bar{\beta}_{r_p,r_{(p+1)}}\int_{a_{p+1}}^{b_{p+1}}e^{-t\left(\bar{\beta}_{r_p,r_p-1}+\bar{\beta}_{r_p,r_p+1}\right)}\, dt\lim_{\varepsilon\to 0}\mu_{j}(\Gamma_{p,\omega}^\varepsilon).\label{eqn:WTS} 
\end{align}
Fix $\xi>0$. In what follows, consider the function 
\begin{equation}
\hat{\gamma}_{\xi,j,r_p,\textbf{n}_p,\omega}^\varepsilon:=\mathcal{L}_{r_p,\sigma^{n_1+\cdots+ n_p}\omega}^{\varepsilon\, (\frac{\xi}{\varepsilon})}(\mathds{1}_{I_{r_p}}\cdot \mathcal{L}_{\Gamma_{p,\textbf{n}_p,\omega}^\varepsilon}(\phi_j))\label{eqn:little_gamman}
\end{equation}
and define
\begin{equation}
\hat{\gamma}_{\xi,j,r_p,\omega}^\varepsilon:=\left(\sum_{n_p=\frac{a_p}{\varepsilon}}^{\frac{b_p}{\varepsilon}}\cdots\sum_{n_1=\frac{a_1}{\varepsilon}}^{\frac{b_1}{\varepsilon}} \hat{\gamma}_{\xi,j,r_p,\textbf{n}_p,\omega}^\varepsilon\right)\cdot \mathds{1}_{I_{r_p}}.\label{eqn:little_gamma}
\end{equation}
\jpt{Given the initial density $\phi_j$, the object given by \eqref{eqn:little_gamma} is the density in $I_{{r_p}}$ pushed forward along the path prescribed by \eqref{eqn:Gamma}, and then another $\frac{\xi}{\varepsilon}$ time steps. We emphasise that the operator defined by \eqref{eqn:gamma_op} will not be iterated. We will only be interested in the resulting object given by \eqref{eqn:little_gamma}.} In \Step{stp:unif-Gamma} we show that for $\mathbb{P}$-a.e. $\omega\in\Omega$, $\|\hat{\gamma}_{\xi,j,r_p,\omega}^\varepsilon\|_{\BV(I)}$ is uniformly bounded over $\varepsilon>0$ sufficiently small, crucial to obtain \eqref{eqn:WTS}.
\begin{step}
    Fix $\xi>0,\, p\in\mathbb{N}$ and $j,r_p\in\{1,\cdots ,m\}$. Then for all $\varepsilon>0$ sufficiently small and $\mathbb{P}$-a.e. $\omega\in\Omega$, there exists $C_\omega>0$ such that
    $$\|\hat{\gamma}_{\xi,j,r_p,\omega}^\varepsilon\|_{\BV(I)}\leq C_\omega,$$
    where $\hat{\gamma}_{\xi,j,r_p,\omega}^\varepsilon$ is as in \eqref{eqn:little_gamma}.
    \label{stp:unif-Gamma}
\end{step}
\begin{proof}
    The result follows by induction on $p\in\mathbb{N}$. \jp{When $p=1$, due to \eqref{eqn:gamma1_op}, \eqref{eqn:gamma_op}}, \hyperref[list:P3]{\textbf{(P3)}}, \hyperref[list:P5]{\textbf{(P5)}}, and \lem{lem:ULY_open}, there exists $C>0$ such that for $\mathbb{P}$-a.e. $\omega\in\Omega$,
    \begin{align}
        \|\hat{\gamma}_{\xi,j,r_1,\omega}^\varepsilon\|_{\BV(I)}&\leq C\left( \sum_{n=\frac{a_1}{\varepsilon}}^{\frac{b_1}{\varepsilon}}\left(r^{\frac{\xi}{\varepsilon}}\|\mathcal{L}_{j,\omega}^{\varepsilon\,(n-1)}(\phi_j)\|_{\BV(I)}+ \| \mathcal{L}_{j,\omega}^{\varepsilon\,(n-1)}(\phi_j)\cdot \mathds{1}_{H_{j,r_1,\sigma^{n-1}\omega}^\varepsilon}\|_{L^1(\leb)}\right)\right). \label{eqn:gamma1BV}   \end{align}
    From \Step{stp:base}, observe that by \eqref{eqn:add-sub} 
    \begin{align*}
    &\sum_{n=\frac{a_1}{\varepsilon}}^{\frac{b_1}{\varepsilon}}\| \mathcal{L}_{j,\omega}^{\varepsilon\,(n-1)}(\phi_j)\cdot \mathds{1}_{H_{j,r_1,\sigma^{n-1}\omega}^\varepsilon}\|_{L^1(\leb)}\\ &\qquad=\nu_{j,\omega}^\varepsilon(\phi_j) \sum_{n=\frac{a_1}{\varepsilon}}^{\frac{b_1}{\varepsilon}}\lambda_{j,\omega}^{\varepsilon\, (n-1)}\int_{H_{j,r_1,\sigma^{n-1}\omega}^\varepsilon}\phi_{j,\sigma^{n-1}\omega}^\varepsilon(x) \, \dleb(x) +O_{\varepsilon\to 0}\left(\frac{\theta^{\frac{a_1}{\varepsilon}-1}-\theta^{\frac{b_1}{\varepsilon}}}{1-\theta} \right).    
    \end{align*}
        Therefore, following the same argument from \Step{stp:base}, we find that for $\mathbb{P}$-a.e. $\omega\in\Omega$, 
 \begin{align}
     &\sum_{n=\frac{a_1}{\varepsilon}}^{\frac{b_1}{\varepsilon}}\| \mathcal{L}_{j,\omega}^{\varepsilon\,(n-1)}(\phi_j)\cdot \mathds{1}_{H_{j,r_1,\sigma^{n-1}\omega}^\varepsilon}\|_{L^1(\leb)} \nonumber\\
     &\qquad= \bar{\beta}_{j,r_1}\int_{a_1}^{b_1} e^{-t\left(\bar{\beta}_{j,j-1}+\bar{\beta}_{j,j+1}\right)}\, dt+o_{\omega,\varepsilon\to 0}(1). \label{eqn:L1p1-bound}
 \end{align}
       Due to \thrm{thrm:openopprop}, for all $\varepsilon>0$ sufficiently small and $\mathbb{P}$-a.e. $\omega\in\Omega$, there exists $K>0$ such that $\|\nu_{j,\omega}^\varepsilon\|_{\BV^*(I_j)},\|\phi_{j,\omega}^\varepsilon\|_{\BV(I_j)}\leq K$. \jp{Consider the class of functions $(\tilde{\beta}_{j,\omega}^\varepsilon)_{\varepsilon>0}$ whereby} $\tilde{\beta}_{j,\omega}^\varepsilon:=\bar{\beta}_{j,j-1}+\bar{\beta}_{j,j+1}+o_{\omega,\varepsilon\to 0}(1)$. Using \thrm{thrm:openopprop}(e) and \cor{cor:lambda_iter}, for $\mathbb{P}$-a.e. $\omega\in\Omega$,
        \begin{align}
            &\sum_{n=\frac{a_1}{\varepsilon}}^{\frac{b_1}{\varepsilon}}r^{\frac{\xi}{\varepsilon}}\|\mathcal{L}_{j,\omega}^{\varepsilon\,(n-1)}(\phi_j)\|_{\BV(I)}\\
            &\qquad=r^{\frac{\xi}{\varepsilon}}\nu_{j,\omega}^\varepsilon(\phi_j) \sum_{n=\frac{a_1}{\varepsilon}}^{\frac{b_1}{\varepsilon}}\lambda_{j,\omega}^{\varepsilon\, (n-1)}\|\phi_{j,\sigma^{n-1}\omega}^\varepsilon\|_{\BV(I_j)}+r^{\frac{\xi}{\varepsilon}}O_{\varepsilon\to 0}\left(\frac{\theta^{\frac{a_1}{\varepsilon}-1}-\theta^{\frac{b_1}{\varepsilon}}}{1-\theta} \right)\nonumber\\
            &\qquad\leq Kr^{\frac{\xi}{\varepsilon}}e^{-(a_1+\varepsilon)\tilde{\beta}_{j,\omega}^\varepsilon}\left(\frac{e^{\varepsilon\tilde{\beta}_{j,\omega}^\varepsilon}-e^{-(b_1-a_1)\tilde{\beta}_{j,\omega}^\varepsilon}}{e^{\varepsilon\tilde{\beta}_{j,\omega}^\varepsilon}-1}\right)+r^{\frac{\xi}{\varepsilon}}O_{\varepsilon\to 0}\left(\frac{\theta^{\frac{a_1}{\varepsilon}-1}-\theta^{\frac{b_1}{\varepsilon}}}{1-\theta} \right)\nonumber\\
            &\qquad= o_{\omega,\varepsilon\to 0}(1) \label{eqn:gamma1-bvbound}
        \end{align}
        since $\xi>0$ is fixed and $0<r<1$. Returning to \eqref{eqn:gamma1BV} and applying \eqref{eqn:L1p1-bound} and \eqref{eqn:gamma1-bvbound}, for all $\varepsilon>0$ sufficiently small, the $p=1$ case holds. Next, suppose that there exists $C_\omega>0$ such that $\|\hat{\gamma}_{\xi,j,r_p,\omega}^\varepsilon\|_{\BV(I)}\leq C_\omega$. We aim to prove that there exists $\tilde{C}_\omega>0$ such that $\|\hat{\gamma}_{\xi,j,r_{(p+1)},\omega}^\varepsilon\|_{\BV(I)}\leq \tilde{C}_\omega$. Indeed, observe that for $p\in\mathbb{N}$, with \eqref{eqn:gamma_op}, and by following \eqref{eqn:gamma1BV}, for $\xi>0$,
    \begin{align}
    \|\hat{\gamma}_{\xi,j,r_{(p+1)},\omega}^\varepsilon\|_{\BV(I)}&\leq  C\Bigg( \sum_{n_{p+1}=\frac{a_{p+1}}{\varepsilon}}^{\frac{b_{p+1}}{\varepsilon}}\cdots  \sum_{n_1=\frac{a_{1}}{\varepsilon}}^{\frac{b_{1}}{\varepsilon}}\Bigg(r^{\frac{\xi}{\varepsilon}}\|\mathcal{L}_{r_{p},\sigma^{n_1+\cdots+n_{p}}\omega}^{\varepsilon\,(n_{p+1}-1)}(\mathcal{L}_{\Gamma_{p,\textbf{n}_p,\omega}^\varepsilon}(\phi_j))\|_{\BV(I)}\nonumber\\
    &\qquad+ \| \mathcal{L}_{r_{p},\sigma^{n_1+\cdots+n_{p}}\omega}^{\varepsilon\,(n_{p+1}-1)}(\mathcal{L}_{\Gamma_{p,\textbf{n}_p,\omega}^\varepsilon}(\phi_j))\cdot \mathds{1}_{H_{r_{p},r_{(p+1)},\sigma^{n_1+\cdots+n_{p+1}-1}\omega}^\varepsilon}\|_{L^1(\leb)}\Bigg)\Bigg) \\
    &= C\Bigg( \sum_{n_{p+1}=\frac{a_{p+1}}{\varepsilon}}^{\frac{b_{p+1}}{\varepsilon}}\cdots  \sum_{n_1=\frac{a_{1}}{\varepsilon}}^{\frac{b_{1}}{\varepsilon}}\Bigg(r^{\frac{\xi}{\varepsilon}}\|\mathcal{L}_{r_{p},\sigma^{n_1+\cdots+n_p+\frac{\xi}{\varepsilon}}\omega}^{\varepsilon\,(n_{p+1}-1-\frac{\xi}{\varepsilon})}(\hat{\gamma}_{\xi,j,r_p,\textbf{n}_p,\omega}^\varepsilon)\|_{\BV(I)} \nonumber \\
    &\qquad + \| \mathcal{L}_{r_{p},\sigma^{n_1+\cdots+n_p+\frac{\xi}{\varepsilon}}\omega}^{\varepsilon\,(n_{p+1}-1-\frac{\xi}{\varepsilon})}(\hat{\gamma}_{\xi,j,r_p,\textbf{n}_p,\omega}^\varepsilon)\cdot \mathds{1}_{H_{r_{p},r_{(p+1)},\sigma^{n_1+\cdots+n_{p+1}-1}\omega}^\varepsilon}\|_{L^1(\leb)}\Bigg)\Bigg).\label{eqn:gammapp+1BV}
    \end{align}
    Through a similar computation to that made in the $p=1$ case, by the inductive hypothesis and \thrm{thrm:openopprop}, for $\mathbb{P}$-a.e. $\omega\in\Omega$, 
    \begin{align}
        &\sum_{n_{p+1}=\frac{a_{p+1}}{\varepsilon}}^{\frac{b_{p+1}}{\varepsilon}} \cdots \sum_{n_1=\frac{a_{1}}{\varepsilon}}^{\frac{b_{1}}{\varepsilon}}\| \mathcal{L}_{r_{p},\sigma^{n_1+\cdots+n_p+\frac{\xi}{\varepsilon}}\omega}^{\varepsilon\,(n_{p+1}-1-\frac{\xi}{\varepsilon})}(\hat{\gamma}_{\xi,j,r_p,\textbf{n}_p,\omega}^\varepsilon)\cdot \mathds{1}_{H_{r_{p},r_{(p+1)},\sigma^{n_1+\cdots+n_{p+1}-1}\omega}^\varepsilon}\|_{L^1(\leb)}\nonumber\\&=\sum_{n_{p}=\frac{a_{p}}{\varepsilon}}^{\frac{b_{p}}{\varepsilon}} \cdots \sum_{n_1=\frac{a_{1}}{\varepsilon}}^{\frac{b_{1}}{\varepsilon}} \Bigg(\nu_{{r_p},\omega}^\varepsilon(\hat{\gamma}_{\xi,j,r_p,\textbf{n}_p,\omega}^\varepsilon)\sum_{n_{p+1}=\frac{a_{p+1}}{\varepsilon}}^{\frac{b_{p+1}}{\varepsilon}}\Bigg(\lambda_{{r_p},\sigma^{n_1+\cdots+n_p+\frac{\xi}{\varepsilon}}\omega}^{\varepsilon\, (n_{p+1}-\frac{\xi}{\varepsilon}-1)}\\
        &\quad \times \int_{H_{r_p,r_{(p+1)},\sigma^{n_1+\cdots+n_{p+1}-1}\omega}^\varepsilon}\phi_{{r_p},\sigma^{n_1+\cdots+n_{p+1}-1} \omega}^\varepsilon(x)\, \dleb(x)\Bigg)\Bigg)\nonumber \\
        &\qquad +O_{\varepsilon\to 0}\left(\frac{\theta^{-\frac{\xi}{\varepsilon}}(\theta^{\frac{a_{p+1}}{\varepsilon}-1}-\theta^{\frac{b_{p+1}}{\varepsilon}})}{1-\theta}\prod_{d=1}^p\left(\frac{b_d-a_d+\varepsilon}{\varepsilon} \right) \right) \label{eqn:L1Gammap+1}\\
        &= \nu_{{r_p},\omega}^\varepsilon(\hat{\gamma}_{\xi,j,r_p,\omega}^\varepsilon) \left((1+o_{\xi \to 0}(1))\bar{\beta}_{r_p,r_{(p+1)}}\int_{a_{p+1}}^{b_{p+1}}e^{-t\left(\bar{\beta}_{r_p,r_p-1}+\bar{\beta}_{r_p,r_p+1}\right)}\, dt+o_{\omega,\varepsilon\to 0}(1)\right)\\
        &\quad+O_{\varepsilon\to 0}\left(\frac{\theta^{-\frac{\xi}{\varepsilon}}(\theta^{\frac{a_{p+1}}{\varepsilon}-1}-\theta^{\frac{b_{p+1}}{\varepsilon}})}{1-\theta}\prod_{d=1}^p\left(\frac{b_d-a_d+\varepsilon}{\varepsilon} \right) \right).\\
        &\leq C_\omega^\prime \left((1+o_{\xi \to 0}(1))\bar{\beta}_{r_p,r_{(p+1)}}\int_{a_{p+1}}^{b_{p+1}}e^{-t\left(\bar{\beta}_{r_p,r_p-1}+\bar{\beta}_{r_p,r_p+1}\right)}\, dt+o_{\omega,\varepsilon\to 0}(1)\right).\label{eqn:lastlineL1}
    \end{align}
    Here, by the inductive hypothesis, $C_\omega^\prime := \|\hat{\gamma}_{\xi,j,r_p,\omega}^\varepsilon\|_{\BV(I)} \|\nu_{{r_p},\omega}^\varepsilon\|_{\BV^*(I_{r_p})} \leq C_\omega \|\nu_{{r_p},\omega}^\varepsilon\|_{\BV^*(I_{r_p})}$ where, by \thrm{thrm:openopprop}, $\|\nu_{{r_p},\omega}^\varepsilon\|_{\BV^*(I_{r_p})}$ is uniformly bounded over $\varepsilon>0$ sufficiently small. Finally, by a similar argument to the $p=1$ case, and by following \eqref{eqn:L1Gammap+1}, for $\mathbb{P}$-a.e. $\omega\in\Omega$,
    \begin{align}
        &\sum_{n_{p+1}=\frac{a_{p+1}}{\varepsilon}}^{\frac{b_{p+1}}{\varepsilon}}\cdots  \sum_{n_1=\frac{a_{1}}{\varepsilon}}^{\frac{b_{1}}{\varepsilon}}r^{\frac{\xi}{\varepsilon}}\|\mathcal{L}_{r_{p},\sigma^{n_1+\cdots+n_p+\frac{\xi}{\varepsilon}}\omega}^{\varepsilon\,(n_{p+1}-1-\frac{\xi}{\varepsilon})}(\hat{\gamma}_{\xi,j,r_p,\textbf{n}_p,\omega}^\varepsilon)\|_{\BV(I)} \\&\qquad=r^{\frac{\xi}{\varepsilon}}\Bigg(\sum_{n_{p}=\frac{a_{p}}{\varepsilon}}^{\frac{b_{p}}{\varepsilon}} \cdots \sum_{n_1=\frac{a_{1}}{\varepsilon}}^{\frac{b_{1}}{\varepsilon}} \Bigg(\nu_{{r_p},\omega}^\varepsilon(\hat{\gamma}_{\xi,j,r_p,\textbf{n}_p,\omega}^\varepsilon)\sum_{n_{p+1}=\frac{a_{p+1}}{\varepsilon}}^{\frac{b_{p+1}}{\varepsilon}}\Bigg(\lambda_{{r_p},\sigma^{n_1+\cdots+n_p+\frac{\xi}{\varepsilon}}\omega}^{\varepsilon\, (n_{p+1}-\frac{\xi}{\varepsilon}-1)}\\
        &\qquad\quad \times \|\phi_{{r_p},\sigma^{n_1+\cdots+n_{p+1}-1} \omega}^\varepsilon\|_{\BV(I_{r_p})}\Bigg)\Bigg)+O_{\varepsilon\to 0}\left(\frac{\theta^{-\frac{\xi}{\varepsilon}}(\theta^{\frac{a_{p+1}}{\varepsilon}-1}-\theta^{\frac{b_{p+1}}{\varepsilon}})}{1-\theta}\prod_{d=1}^p\left(\frac{b_d-a_d+\varepsilon}{\varepsilon} \right) \right)  \Bigg)\nonumber\\
        &\qquad=o_{\omega,\varepsilon\to 0}(1) \label{eqn:lastlineBV}.
    \end{align}
    Returning to \eqref{eqn:gammapp+1BV} and applying \eqref{eqn:lastlineL1} and \eqref{eqn:lastlineBV}, the result follows.
\end{proof}}
We use \Step{stp:unif-Gamma}, \Step{stp:D-lim} and \Step{stp:C-lim} to perform the inductive step.
\begin{step}
Fix $\xi>0$, take an interval $[a_{p+1},b_{p+1}]$ and numbers $r_{p},r_{(p+1)}\in\{1,\cdots, m\}$. Then for $\mathbb{P}$-a.e. $\omega\in\Omega$, \jp{the class of functions $(D_{\xi,r_p,r_{(p+1)},\omega}^\varepsilon)_{\varepsilon>0}$} defined by 
\jp{\begin{align}
        D_{\xi,r_p,r_{(p+1)},\omega}^\varepsilon&:= O_{\varepsilon\to 0}\left(\esssup_{\omega\in\Omega}\sup_{x\in H_{r_p,r_{(p+1)},\omega}^\varepsilon}|\phi_{r_p,\omega}^\varepsilon(x) - \phi_{r_p}(x)| \right)\sum_{n_{p}=\frac{a_{p}}{\varepsilon}}^{\frac{b_{p}}{\varepsilon}} \cdots \sum_{n_1=\frac{a_{1}}{\varepsilon}}^{\frac{b_{1}}{\varepsilon}} \Bigg(\nu_{{r_p},\omega}^\varepsilon(\hat{\gamma}_{\xi,j,r_p,\textbf{n}_p,\omega}^\varepsilon)\\
        &\times\sum_{n_{p+1}=\frac{a_{p+1}}{\varepsilon}}^{\frac{b_{p+1}}{\varepsilon}}\Bigg(\lambda_{{r_p},\sigma^{n_1+\cdots+n_p+\frac{\xi}{\varepsilon}}\omega}^{\varepsilon\, (n_{p+1}-\frac{\xi}{\varepsilon})}\leb_{r_p}(H_{r_p,r_{(p+1)},\sigma^{n_1+\cdots+n_{p+1}}\omega}^\varepsilon)\Bigg)\Bigg)
       \label{eqn:D-lim}
\end{align}
satisfies
$$\lim_{\varepsilon\to 0}D_{\xi,r_p,r_{(p+1)},\omega}^\varepsilon = 0.$$}
 \label{stp:D-lim}
\end{step}
\begin{proof}
    This follows in a similar manner to \Step{stp:B-lim}. Indeed, for $\mathbb{P}$-a.e. $\omega\in\Omega$, due to \lem{lem:leb_estimate} and \thrm{thrm:openopprop}(a), there exists $M>0$ such that \jp{for $\varepsilon>0$ sufficiently small},
    \jp{\begin{align}
     D_{\xi,r_p,r_{(p+1)},\omega}^\varepsilon&\leq  \nu_{{r_p},\omega}^\varepsilon(\hat{\gamma}_{\xi,j,r_p,\omega}^\varepsilon) O_{\varepsilon\to 0}\left(\esssup_{\omega\in\Omega}\sup_{x\in H_{r_p,r_{(p+1)},\omega}^\varepsilon}|\phi_{r_p,\omega}^\varepsilon(x) - \phi_{r_p}(x)| \right)M\nonumber\\
    &\qquad \times(b_{p+1}-a_{p+1}+\varepsilon).\label{eqn:d-esti}
    \end{align}
    It remains to estimate $\lim_{\varepsilon\to 0}\nu_{{r_p},\omega}^\varepsilon(\hat{\gamma}_{\xi,j,r_p,\omega}^\varepsilon)$. However, thanks to \Step{stp:unif-Gamma} and \thrm{thrm:openopprop}, \begin{equation}
    \|\hat{\gamma}_{\xi,j,r_p,\omega}^\varepsilon\|_{\BV(I)} \|\nu_{{r_p},\omega}^\varepsilon\|_{\BV^*(I_{r_p})} \leq C_\omega \|\nu_{{r_p},\omega}^\varepsilon\|_{\BV^*(I_{r_p})}\label{eqn:nugam-est}
    \end{equation} 
where, $\|\nu_{{r_p},\omega}^\varepsilon\|_{\BV^*(I_{r_p})}$ is uniformly bounded over $\varepsilon>0$ sufficiently small. Therefore, taking $\varepsilon\to 0$ in \eqref{eqn:d-esti} and applying \thrm{thrm:openopprop}(d) and \eqref{eqn:nugam-est}, the result follows.} 
   \end{proof}  

\begin{step}
Fix $\xi>0$, take an interval $[a_{p+1},b_{p+1}]$ and numbers $r_{p},r_{(p+1)}\in\{1,\cdots, m\}$. Then for $\mathbb{P}$-a.e. $\omega\in\Omega$, the function 
\jp{\begin{equation}
    C_{\xi,r_p,r_{(p+1)},\omega}^\varepsilon:=\sum_{n_{p}=\frac{a_{p}}{\varepsilon}}^{\frac{b_{p}}{\varepsilon}} \cdots \sum_{n_1=\frac{a_{1}}{\varepsilon}}^{\frac{b_{1}}{\varepsilon}} \Bigg(\nu_{{r_p},\omega}^\varepsilon(\hat{\gamma}_{\xi,j,r_p,\textbf{n}_p,\omega}^\varepsilon)\sum_{n_{p+1}=\frac{a_{p+1}}{\varepsilon}}^{\frac{b_{p+1}}{\varepsilon}}\lambda_{{r_p},\sigma^{n_1+\cdots+n_p+\frac{\xi}{\varepsilon}}\omega}^{\varepsilon\, (n_{p+1}-\frac{\xi}{\varepsilon})}\mu_{r_p}(H_{r_p,r_{(p+1)},\sigma^{n_1+\cdots+n_{p+1}}\omega}^\varepsilon)\Bigg) \label{eqn:C-lim}
\end{equation}}
satisfies
\begin{align*}
    \lim_{\varepsilon\to 0}C_{\xi,r_p,r_{(p+1)},\omega}^\varepsilon&=(1+o_{\xi \to 0}(1))\bar{\beta}_{r_p,r_{(p+1)}}\int_{a_{p+1}}^{b_{p+1}}e^{-\left(\bar{\beta}_{r_p,r_p-1}+\bar{\beta}_{r_p,r_p+1}\right)}\, dt \\ &\qquad \times\lim_{\varepsilon\to 0}\leb_{r_p}(\hat{\gamma}_{\xi,j,r_p,\omega}^\varepsilon).
\end{align*}
 \label{stp:C-lim}
\end{step}
\begin{proof}
    This follows by a similar argument to that made in \Step{stp:A-lim}. Through \thrm{thrm:openopprop}(a) and \hyperref[list:P4]{\textbf{(P4)}}, we find that 
\jp{\begin{align*}
{C}_{\xi,r_p,r_{(p+1)},\omega}^\varepsilon&= 
\sum_{n_{p}=\frac{a_{p}}{\varepsilon}}^{\frac{b_{p}}{\varepsilon}} \cdots \sum_{n_1=\frac{a_{1}}{\varepsilon}}^{\frac{b_{1}}{\varepsilon}} \Bigg(\nu_{{r_p},\omega}^\varepsilon(\hat{\gamma}_{\xi,j,r_p,\textbf{n}_p,\omega}^\varepsilon)\lambda_{{r_p},\sigma^{n_1+\cdots+n_p+\frac{\xi}{\varepsilon}}\omega}^{\varepsilon\, (\frac{a_{p+1}-\xi}{\varepsilon})} \\
&\quad\times\sum_{n_{p+1}=0}^{\frac{b_{p+1}-a_{p+1}}{\varepsilon}}(\varepsilon\beta_{r_p,r_{(p+1)},\sigma^{n_1+\cdots+n_{p+1}+\frac{a_{p+1}}{\varepsilon}}\omega}+o_{\varepsilon\to 0}(\varepsilon))\\
&\quad \times\prod_{k=0}^{n_{p+1}-1}(1-\varepsilon(\beta_{r_p,r_p-1,\sigma^{k+n_1+\cdots+n_{p}+\frac{a_{p+1}}{\varepsilon}}}+\beta_{r_p,r_{p}+1,\sigma^{k+n_1+\cdots+n_{p}+\frac{a_{p+1}}{\varepsilon}}})+o_{\varepsilon\to 0}(\varepsilon) )\Bigg).
\end{align*}}
As in \Step{stp:A-lim}, applying \thrm{thrm:openopprop}(b), \cor{cor:lambda_iter}, \lem{lem:exp_dist} and \Step{stp:unif-Gamma}, for $\mathbb{P}$-a.e. $\omega\in\Omega$ 
\begin{align*}
    \lim_{\varepsilon\to 0}{C}_{\xi,r_p,r_{(p+1)},\omega}^\varepsilon 
    &=e^{\xi\left(\bar{\beta}_{r_p,r_p-1}+\bar{\beta}_{r_p,r_p+1}\right)}\bar{\beta}_{r_p,r_{(p+1)}}\\
    &\quad \times \int_{a_{p+1}}^{b_{p+1}}e^{-t\left(\bar{\beta}_{r_p,r_p-1}+\bar{\beta}_{r_p,r_p+1}\right)}\, dt\lim_{\varepsilon\to 0}\leb_{r_p}(\hat{\gamma}_{\xi,j,r_p,\omega}^\varepsilon)\\
    &=(1+o_{\xi \to 0}(1))\bar{\beta}_{r_p,r_{(p+1)}}\int_{a_{p+1}}^{b_{p+1}}e^{-t\left(\bar{\beta}_{r_p,r_p-1}+\bar{\beta}_{r_p,r_p+1}\right)}\, dt\\ &\qquad \times\lim_{\varepsilon\to 0}\leb_{r_p}(\hat{\gamma}_{\xi,j,r_p,\omega}^\varepsilon).
\end{align*}
\end{proof}
We may now continue towards the inductive step.
\begin{step}[Towards the inductive step]
    Fix $j\in \{1,\cdots,m\}$ and $\xi>0$. Take an interval $\Delta_{p+1}=[a_{p+1},b_{p+1}]$ and numbers $r_p,r_{(p+1)}\in \{1,\cdots, m\}$. Then, for $\mathbb{P}$-a.e. $\omega\in\Omega$, 
    \begin{align*}
        &\lim_{\varepsilon\to 0} \mu_j\left(\left\{x\in I \ | \ \varepsilon \mathcal{T}_{p+1,\omega}^\varepsilon(x)\in \Delta_{p+1} \ \mathrm{and} \ z(T_{\omega}^{\varepsilon\, (t_{p+1,\omega}^\varepsilon(x))}(x))=r_{(p+1)}  \right\}\cap\Gamma_{p,\omega}^\varepsilon \right) \\
        &\qquad =(1+o_{\xi \to 0}(1))\bar{\beta}_{r_p,r_p+1}\int_{a_{p+1}}^{b_{p+1}}e^{-t\left(\bar{\beta}_{r_p,r_p-1}+\bar{\beta}_{r_p,r_p+1}\right)}\, dt\\
    &\qquad\qquad \times \lim_{\varepsilon\to 0}(\mu_{j}\left(\Gamma_{p,\omega}^\varepsilon\right)+o_{\xi\to 0}(1))
    \end{align*}
    \label{stp:ind-step}
    where $\Gamma_{p,\omega}^\varepsilon$ is as in \eqref{eqn:Gamma}.
\end{step}
\begin{proof}
For a fixed $\xi>0$, observe that by using \eqref{eqn:gamma_op},
\jp{\begin{align*}
\mu_j\left(\Gamma_{p+1,\omega}^\varepsilon \right)&= \sum_{n_{p+1}=\frac{a_{p+1}}{\varepsilon}}^{\frac{b_{p+1}}{\varepsilon}} \cdots \sum_{n_1=\frac{a_{1}}{\varepsilon}}^{\frac{b_{1}}{\varepsilon}} \int_{H_{r_p,r_{(p+1)},\sigma^{n_1+\cdots+n_{p+1}}\omega}} \mathcal{L}_{r_p,\sigma^{n_1+\cdots+n_p}\omega}^{\varepsilon\, (n_{p+1})}(\mathds{1}_{I_{r_p}}\cdot \mathcal{L}_{\Gamma_{p,\textbf{n}_p,\omega}^\varepsilon}(\phi_j))(x)\, \dleb(x)\\
&=\sum_{n_{p+1}=\frac{a_{p+1}}{\varepsilon}}^{\frac{b_{p+1}}{\varepsilon}} \cdots \sum_{n_1=\frac{a_{1}}{\varepsilon}}^{\frac{b_{1}}{\varepsilon}}\int_{H_{r_p,r_{(p+1)},\sigma^{n_1+\cdots+n_{p+1}}\omega}^\varepsilon} \mathcal{L}_{r_p,\sigma^{n_1+\cdots+n_{p}+\frac{\xi}{\varepsilon}}\omega}^{\varepsilon\, (n_{p+1}-\frac{\xi}{\varepsilon})}(\hat{\gamma}_{\xi,j,r_p,\textbf{n}_p,\omega}^\varepsilon)(x)\, \dleb(x).
\end{align*}
Recall that thanks to \Step{stp:unif-Gamma}, $\hat{\gamma}_{\xi,j,r_p,\omega}^\varepsilon=\left(\sum_{n_p=\frac{a_p}{\varepsilon}}^{\frac{b_p}{\varepsilon}}\cdots\sum_{n_1=\frac{a_1}{\varepsilon}}^{\frac{b_1}{\varepsilon}} \hat{\gamma}_{\xi,j,r_p,\textbf{n}_p,\omega}^\varepsilon\right)\cdot \mathds{1}_{I_{r_p}}$ has uniformly bounded $\BV(I_{r_p})$ norm over $\varepsilon>0$. Thus, by the same argument made in \Step{stp:base},
\begin{align*}
    \mu_j(\Gamma_{p+1,\omega}^\varepsilon)&=\sum_{n_{p}=\frac{a_{p}}{\varepsilon}}^{\frac{b_{p}}{\varepsilon}} \cdots \sum_{n_1=\frac{a_{1}}{\varepsilon}}^{\frac{b_{1}}{\varepsilon}} \Bigg(\nu_{{r_p},\omega}^\varepsilon(\hat{\gamma}_{\xi,j,r_p,\textbf{n}_p,\omega}^\varepsilon)\sum_{n_{p+1}=\frac{a_{p+1}}{\varepsilon}}^{\frac{b_{p+1}}{\varepsilon}}\Bigg(\lambda_{{r_p},\sigma^{n_1+\cdots+n_p+\frac{\xi}{\varepsilon}}\omega}^{\varepsilon\, (n_{p+1}-\frac{\xi}{\varepsilon})}\\
        &\quad \times \int_{H_{r_p,r_{(p+1)},\sigma^{n_1+\cdots+n_{p+1}}\omega}^\varepsilon}\phi_{{r_p},\sigma^{n_1+\cdots+n_{p+1}} \omega}^\varepsilon(x)\, \dleb(x)\Bigg)\Bigg)\nonumber \\
        &\qquad +O_{\varepsilon\to 0}\left(\frac{\theta^{\frac{a_{p+1}-\xi}{\varepsilon}}-\theta^{\frac{b_{p+1}-\xi}{\varepsilon}+1}}{1-\theta}\prod_{d=1}^p\left(\frac{b_d-a_d+\varepsilon}{\varepsilon} \right) \right)\\
        &={C}_{\xi,r_p,r_{(p+1)},\omega}^\varepsilon+ {D}_{\xi,r_p,r_{(p+1)},\omega}^\varepsilon +O_{\varepsilon\to 0}\left(\frac{\theta^{\frac{a_{p+1}-\xi}{\varepsilon}}-\theta^{\frac{b_{p+1}-\xi}{\varepsilon}+1}}{1-\theta}\prod_{d=1}^p\left(\frac{b_d-a_d+\varepsilon}{\varepsilon} \right) \right).
\end{align*}}
Here, $C_{\xi,r_p,r_{(p+1)},\omega}^\varepsilon$ and ${D}_{\xi,r_p,r_{(p+1)},\omega}^\varepsilon$ are given by \eqref{eqn:C-lim} and \eqref{eqn:D-lim}, respectively.
One can check that since $\theta\in(0,1)$, as $\varepsilon\to 0$, the error vanishes for $\xi$ small enough. So, by \Step{stp:D-lim} and \Step{stp:C-lim}, 

\begin{align*}
    &\lim_{\varepsilon\to 0}\mu_j\left(\left\{x\in I \ | \ \varepsilon \mathcal{T}_{p+1,\omega}^\varepsilon(x)\in \Delta_{p+1} \ \mathrm{and} \ z(T_{\omega}^{\varepsilon\, (t_{p+1,\omega}^\varepsilon(x))}(x))=r_{(p+1)}  \right\}\cap\Gamma_{p,\omega}^\varepsilon \right)\\
    &\qquad=(1+o_{\xi \to 0}(1))\bar{\beta}_{r_p,r_p+1}\int_{a_{p+1}}^{b_{p+1}}e^{-t\left(\bar{\beta}_{r_p,r_p-1}+\bar{\beta}_{r_p,r_p+1}\right)}\, dt\\ &\qquad \qquad \times\lim_{\varepsilon\to 0}\leb_{r_p}(\hat{\gamma}_{\xi,j,r_p,\omega}^\varepsilon).
\end{align*}
By an inductive argument, 
\jp{\begin{align*}
    \lim_{\varepsilon\to 0}\leb_{r_p}(\hat{\gamma}_{\xi,j,r_p,\omega}^\varepsilon)&=\lim_{\varepsilon\to 0}\int_{I_{r_p}}\sum_{n_{p}=\frac{a_{p}}{\varepsilon}}^{\frac{b_{p}}{\varepsilon}} \cdots \sum_{n_1=\frac{a_{1}}{\varepsilon}}^{\frac{b_{1}}{\varepsilon}}\mathcal{L}_{r_p,\sigma^{n_1+\cdots +n_p}\omega}^{\varepsilon\, (\frac{\xi}{\varepsilon})}(\mathds{1}_{I_{r_p}}\cdot \mathcal{L}_{\Gamma_{p,\textbf{n}_p,\omega}^\varepsilon}(\phi_j))(x)\, \dleb(x)\\
    &=\lim_{\varepsilon\to 0}\mu_{j}\left(\left\{x\in I_{} \ | \ \mathcal{T}_{p+1,\omega}^\varepsilon(x)>\frac{\xi}{\varepsilon}\right\}\cap\Gamma_{p,\omega}^\varepsilon\right).
\end{align*}}
It remains to show that 
$$\lim_{\varepsilon\to 0}\mu_{j}\left(\left\{x\in I_{} \ | \ \mathcal{T}_{p+1,\omega}^\varepsilon(x)>\frac{\xi}{\varepsilon}\right\}\cap\Gamma_{p,\omega}^\varepsilon\right) = \lim_{\varepsilon\to 0}\mu_{j}\left(\Gamma_{p,\omega}^\varepsilon\right)+o_{\xi \to 0}(1).$$
However, this follows immediately from \lem{lem:no-jump} as $\left\{x\in I_{} \ | \ \mathcal{T}_{p+1,\omega}^\varepsilon(x)>\frac{\xi}{\varepsilon}\right\}$ is contained in the complement of the set appearing in the statement of \lem{lem:no-jump}.
\end{proof}
\Change{To complete the inductive step, we wish to obtain \eqref{eqn:WTS}, however, this follows immediately from \Step{stp:ind-step} by taking $\xi\to 0$.}   
\end{proof}

\section{The diffusion coefficient}
\label{sec:var}
In this section, we show that in the setting of \thrm{thrm:openopprop}, if $\varepsilon>0$ is fixed, then the collection of random dynamical systems $\{(\Omega,\mathcal{F},\mathbb{P},\sigma,\BV(I),\mathcal{L}^\varepsilon)\}_{\varepsilon> 0}$ associated with metastable maps $T_\omega^\varepsilon:I\to I$ satisfies a quenched version of the Central Limit Theorem (CLT) \cite[Theorem B]{DFGTV_limthrm}. Additionally, we provide an approximation for the diffusion coefficient (the variance of the limiting normal distribution in the CLT) when $\varepsilon>0$ is small. This approximation is expressed in terms of the averaged Markov jump process introduced in \Sec{sec:avg_markov}. 
\\
\\
We begin by introducing the class of \textit{observables} for which our results apply. This class is similar to that considered in \cite{DFGTV_limthrm}. In what follows, let $\phi_\omega^0:= \sum_{j=1}^m p_j\phi_j$ be the limiting invariant measure admitted by \cite[Theorem 7.2]{gtp_met}. 
\begin{definition}
    We call the measurable map $\psi:\Omega\times I \to \mathbb{R}$ a \textit{\Change{regular, fibrewise centered} observable} if:
    \begin{itemize}
        \item[(a)] Regularity.
        \begin{equation}
            \esssup_{\omega \in\Omega}\| \psi_\omega\|_{\BV(I)}=C<\infty.
            \label{eqn:obs-reg}
        \end{equation}
        \item[(b)] Fibrewise centering.
        \begin{equation}
            \int_I \psi_\omega(x)\, d\mu_\omega^0(x)=\int_I \psi_\omega(x)\phi_\omega^0(x)\, \dleb(x)=\sum_{j=1}^m p_j \int_{I_j} \psi_\omega(x)\phi_j(x)\, \dleb(x)=0
            \label{eqn:obs-centering2}
        \end{equation}
        for $\mathbb{P}$-a.e. $\omega\in\Omega$.
    \end{itemize}
    \label{def:obs}
\end{definition}
For all $\varepsilon>0$ we define the $\varepsilon$-\textit{centered observable}
\begin{equation}
\tilde{\psi}_\omega^\varepsilon:=\psi_\omega - \mu_{\omega}^\varepsilon(\psi_\omega)\label{eqn:obs-centering1}    
\end{equation}
and consider
\begin{align}
    (\Sigma^\varepsilon(\tilde{\psi}^\varepsilon))^2&:= \int_\Omega \int_I \tilde{\psi}_\omega^\varepsilon(x)^2\phi_\omega^\varepsilon(x)\, \dleb(x)\, d\mathbb{P}(\omega) \nonumber\\
    &\qquad + 2\sum_{n=1}^\infty \int_\Omega \int_I \mathcal{L}_\omega^{\varepsilon\, (n)}(\tilde{\psi}_\omega^\varepsilon(x)\phi_\omega^\varepsilon(x)) \tilde{\psi}_{\sigma^n\omega}^\varepsilon(x)\, \, \dleb(x)\, d\mathbb{P}(\omega).
    \label{eqn:var}
\end{align}
Let $\Psi_\omega(j):= \int_{I_j}\psi_\omega(x) \phi_j(x)\, \dleb(x)$. 
\begin{remark}
    Note that due to \eqref{eqn:obs-reg}, for a fixed $\varepsilon>0$, each $x\in I$ and $\mathbb{P}$-a.e. $\omega\in \Omega$, the $\varepsilon$-centered observable defined in \eqref{eqn:obs-centering1} satisfies 
    \begin{align*}
        | \tilde{\psi}_\omega^\varepsilon(x)|&\leq \|{\psi}_\omega\|_{\BV(I)}+\int_I|\psi_\omega||\phi_\omega^\varepsilon|\, \dleb(x)\leq \|{\psi}_\omega\|_{\BV(I)}+\|{\psi}_\omega\|_{\BV(I)}\mu_\omega^\varepsilon(I)\leq 2C.
    \end{align*}
   
    \label{rmk:bv-center}
\end{remark}
The main result of this section is \thrm{thrm:variance-est}. We proceed by proving a sequence of lemmata. Let $\BV_0(I):=\{f\in \BV(I) \ | \ \int_If(x)\, \dleb(x)=0\}$. The following results allow us to control the tail end of the series appearing in \eqref{eqn:var}. 
\begin{lemma}
\Change{Fix $t>0$ and let $\varepsilon>0$ be sufficiently small. In the setting of \thrm{thrm:openopprop}, if there exists $\alpha>0$} such that for all $f\in\BV_0(I)$ and $\mathbb{P}$-a.e. $\omega\in \Omega$ 
\begin{equation}
    \| \mathcal{L}_\omega^{\varepsilon\, (\frac{t}{2\varepsilon})}f\|_{L^1(\leb)}\leq \alpha{\|f\|_{\BV_0(I)}}, \label{eqn:L1/2}
\end{equation}
then 
$$\| \mathcal{L}_{\omega}^{\varepsilon\, (\frac{t}{\varepsilon})} f\|_{\BV_0(I)}\leq \frac{\|f\|_{\BV_0(I)}}{2}$$
for $\mathbb{P}$-a.e. $\omega\in \Omega$.
\label{lem:L1/2WTS}
\end{lemma}
\begin{proof}
Thanks to \hyperref[list:P5]{\textbf{(P5)}}, there exist constants $K,r>0$ with $r<1$ such that for all $n\in\mathbb{N}$, $f\in \BV_0(I)$ and $\mathbb{P}$-a.e. $\omega\in\Omega$, 
    \begin{align}     \var_I(\mathcal{L}_\omega^{\varepsilon\, (n)}f)&\leq K(r^n \var_I(f) +\|f\|_{L^1(\leb)}) \label{eqn:ly-var}
    \end{align}
    and
\begin{align}
    \|\mathcal{L}_\omega^{\varepsilon\, (n)}f\|_{L^1(\leb)}\leq K\|f\|_{L^1(\leb)}.\label{eqn:ly-L1}
\end{align}
Thus, with \eqref{eqn:ly-var} and \eqref{eqn:ly-L1}, for $\mathbb{P}$-a.e. $\omega\in\Omega$,
\begin{align*}
    \|\mathcal{L}_\omega^{\varepsilon\, (\frac{t}{\varepsilon})}f\|_{\BV_0(I)}&=\|\mathcal{L}^{\varepsilon\, (\frac{t}{2\varepsilon})}_{\sigma^{\frac{t}{2\varepsilon}}\omega}( \mathcal{L}_\omega^{\varepsilon\, (\frac{t}{2\varepsilon})}f)\|_{\BV_0(I)}\\
    &= \|\mathcal{L}^{\varepsilon\, (\frac{t}{2\varepsilon})}_{\sigma^{\frac{t}{2\varepsilon}}\omega}( \mathcal{L}_\omega^{\varepsilon\, (\frac{t}{2\varepsilon})}f)\|_{L^1(\leb)}+\var_I(\mathcal{L}^{\varepsilon\, (\frac{t}{2\varepsilon})}_{\sigma^{\frac{t}{2\varepsilon}}\omega}( \mathcal{L}_\omega^{\varepsilon\, (\frac{t}{2\varepsilon})}f))\\
    &\leq K\|\mathcal{L}_\omega^{\varepsilon\, (\frac{t}{2\varepsilon})}f\|_{L^1(\leb)}+K(r^{\frac{t}{2\varepsilon}} \var_I(\mathcal{L}_\omega^{\varepsilon\, (\frac{t}{2\varepsilon})}f)+\|\mathcal{L}_\omega^{\varepsilon\, (\frac{t}{2\varepsilon})}f\|_{L^1(\leb)})\\
    &\stackrel{(\star)}{\leq} 2\alpha K {\|f\|_{\BV_0(I)}}{}+Kr^{\frac{t}{2\varepsilon}} \var_I(\mathcal{L}_\omega^{\varepsilon\, (\frac{t}{2\varepsilon})}f)\\
    &\leq2\alpha K {\|f\|_{\BV_0(I)}}+K^2r^{\frac{t}{\varepsilon}}\var_I(f) +K^2r^{\frac{t}{2\varepsilon}} \|f\|_{L^1(\leb)}.
\end{align*}
Note that at $(\star)$, we have used \eqref{eqn:L1/2}. \Change{Choose $\alpha \leq \frac{1}{4K}$}. Since $0<r<1$, taking $\varepsilon>0$ sufficiently small, the result follows. 
\end{proof}
We aim to control $\| \mathcal{L}_\omega^{\varepsilon\, (\frac{t}{2\varepsilon})}f\|_{L^1(\leb)}$ so that we can apply \lem{lem:L1/2WTS}. This may be achieved due to the following.
\begin{lemma}
\Change{Fix $\delta>0$. In the setting of \thrm{thrm:openopprop}, if there exists $\varepsilon_0,\delta_1>0$ such that $\varepsilon\leq \varepsilon_0$, and for any $f\in\BV(I)$ with $\|f\|_{\BV(I)}=1$, $0\leq |\int_{I_k} f(x)\, \dleb(x)|\leq \delta_1$ for each $k\in\{1,\cdots, m\}$, then for $\mathbb{P}$-a.e. $\omega\in\Omega$ and $n_0\in\mathbb{N}$ sufficiently large}
$$\|\mathcal{L}_\omega^{\varepsilon\, (n_0)}f\|_{L^1(\leb)}<\delta.$$
\label{lem:L1/2ctrl}
\begin{proof}
    Thanks to \hyperref[list:P1]{\textbf{(P1)}} and \rem{rem:p1-imp-trip}, 
    \begin{align}
        \sup_{\|f\|_{\BV(I)}=1}\|(\mathcal{L}_\omega^\varepsilon-\mathcal{L}^0)f\|_{L^1(\leb)}=o_{\varepsilon\to 0}(1).
        \label{eqn:n=1-trip}
    \end{align}
    However, for $n>1$, due to the identity that
    \begin{equation}
        \mathcal{L}_\omega^{\varepsilon\, (n)}-\mathcal{L}^{0\, (n)}= \sum_{j=0}^{n-1} \mathcal{L}_{\sigma^{n-j}\omega}^{\varepsilon\, (j)}(\mathcal{L}_{\sigma^{n-j-1}\omega}^{\varepsilon}-\mathcal{L}_{}^0)\mathcal{L}_{}^{0\, (n-j-1)},\label{eqn:Ltrip-iter}
    \end{equation}
    through \eqref{eqn:n=1-trip}, it follows that for each $n\in\mathbb{N}$
    \begin{align}
        \sup_{\|f\|_{\BV(I)}=1}\|(\mathcal{L}_\omega^{\varepsilon\, (n)}-\mathcal{L}^{0\, (n)})f\|_{L^1(\leb)}=o_{\varepsilon\to 0}(1).
        \label{eqn:n-trip}
    \end{align}
    From \eqref{eqn:n-trip} we apply the reverse triangle inequality so that
    \Change{\begin{align}
        o_{\varepsilon\to 0}(1)&=\sup_{\|f\|_{\BV(I)}=1}\|(\mathcal{L}_\omega^{\varepsilon\, (n)}-\mathcal{L}^{0\, (n)})f\|_{L^1(\leb)} \nonumber \\
        &\geq \sup_{\|f\|_{\BV(I)}=1}\left\tnorm\mathcal{L}_\omega^{\varepsilon\, (n)}f \|_{L^1(\leb)}-\|\mathcal{L}^{0\, (n)}f\|_{L^1(\leb)} \right|. \label{eqn:Leps-in-L0}
    \end{align}}\noindent
We conclude by arguing that \Change{$\sup_{\|f\|_{\BV(I)}=1}\|\mathcal{L}^{0\, (n)}f\|_{L^1(\leb)}$} is small. Indeed, setting $\varepsilon =0$ in \eqref{eqn:openop}, observe that $\mathcal{L}^{0\, (n)}=\sum_{j=1}^m \mathcal{L}^{0\, (n)}_j$. Further, since $T^0_j:I_j\to I_j$ is mixing,\footnote{Recall that for each $j\in \{1,\cdots, m\}$, $T^0_j:I_j\to I_j$ denotes the map associated with the Perron-Frobenius operator $\mathcal{L}^{0}_j$.} there exist constants $C>0$ and $\kappa \in (0,1)$ such that for each $j\in\{1,\cdots, m\}$, $n\in\mathbb{N}$ and $f\in\BV(I_j)$,
    \begin{equation}
    \left\| \mathcal{L}_j^{0\, (n)}f - \leb(f\cdot \mathds{1}_{I_j})\phi_j\right\|_{L^1(\leb)}\leq C\kappa^n\|f\|_{\BV(I_j)}.
        \label{eqn:SG-Lj}
    \end{equation}
    Therefore, by assumption, if $0\leq |\int_{I_k} f(x)\, \dleb(x)|\leq \delta_1$ for each $k\in\{1,\cdots, m\}$, using \eqref{eqn:SG-Lj}
    \begin{align}
      \|\mathcal{L}^{0\,(n)}f\|_{L^1(\leb)}&\leq\sum_{j=1}^m\| \mathcal{L}_j^{0\, (n)}(f\cdot \mathds{1}_{I_j}) \|_{L^1(\leb)}\nonumber \\
      &\leq \sum_{j=1}^m |\leb(f\cdot\mathds{1}_{I_j})|\|\phi_j\|_{L^1(\leb)}+O_{n\to \infty}(\kappa^n) \nonumber\\
      &\leq m\delta_1+O_{n\to \infty}(\kappa^n). \label{eqn:b4-Leps}
    \end{align}
    Applying \eqref{eqn:b4-Leps} to \eqref{eqn:Leps-in-L0} it follows that for all $\delta>0$, \Change{if there exists $\varepsilon_0,\delta_1>0$ such that $\varepsilon\leq \varepsilon_0$, and for any $f\in\BV(I)$ with $\|f\|_{\BV(I)}=1$, $0\leq |\int_{I_k} f(x)\, \dleb(x)|\leq \delta_1$ for each $k\in\{1,\cdots, m\}$, then for $\mathbb{P}$-a.e. $\omega\in\Omega$ and $n_0$ large}
    $$\|\mathcal{L}_\omega^{\varepsilon\, (n_0)}f \|_{L^1(\leb)}\leq\|\mathcal{L}^{0\, (n_0)}f\|_{L^1(\leb)} +o_{\varepsilon\to 0}(1)\leq m\delta_1+O_{n_0\to \infty}(\kappa^{n_0})+o_{\varepsilon\to 0}(1)<\delta.$$
\end{proof}
\end{lemma}
The following lemma relies on \lem{lem:L1/2WTS} and \lem{lem:L1/2ctrl}. In particular, we use \lem{lem:L1/2ctrl} to obtain \eqref{eqn:L1/2}, from which we apply \lem{lem:L1/2WTS}. 
\begin{lemma}
In the setting of \thrm{thrm:openopprop}, \Change{for all $\varepsilon>0$ sufficiently small} there exist $\rho\in (0,1)$ and $t>0$ such that for all $f\in\BV_0(I)$ \Change{with $\|f\|_{\BV(I)}=1$}, $n\in\mathbb{N}$, and $\mathbb{P}$-a.e. $\omega\in\Omega$
\begin{equation}
    \|\mathcal{L}_{\omega}^{\varepsilon\, (\frac{nt}{\varepsilon})}f\|_{\BV_0(I)}\leq \rho^n\|f\|_{\BV_0(I)}.\label{eqn:tail-est-wts}
\end{equation}
\label{lem:tail-ctrl}
\end{lemma}
\begin{proof}
    \Change{Fix $\delta>0$. Let $\delta_1,\varepsilon_0>0$ and $n_0\in\mathbb{N}$ be as in \lem{lem:L1/2ctrl}. Set $\bar{n}_{t}^{\varepsilon}:=\frac{t}{2\varepsilon}-n_0$. If $\varepsilon\leq \varepsilon_0$, and for any $f\in \BV(I)$ with $\|f\|_{\BV(I)}=1$,
    \begin{equation}
        0\leq\left| \int_{I_k} \mathcal{L}_{\omega}^{\varepsilon\, (\bar{n}_{t}^{\varepsilon})}(f)(x)\, \dleb(x) \right| \leq \delta_1 \label{eqn:lemtail-WTS}
    \end{equation}
    for each $k\in\{1,\cdots, m\}$ and $\mathbb{P}$-a.e. $\omega\in\Omega$, then \lem{lem:L1/2ctrl} asserts that for $\mathbb{P}$-a.e. $\omega\in\Omega$, 
    \begin{equation}
    \|\mathcal{L}_{\sigma^{\bar{n}_{t}^{\varepsilon}}\omega}^{\varepsilon\, (n_0)}(\mathcal{L}_{\omega}^{\varepsilon\, (\bar{n}_{t}^{\varepsilon})})(f)\|_{L^1(\leb)}=\|\mathcal{L}_{\omega}^{\varepsilon\, (\frac{t}{2\varepsilon})}f\|_{L^1(\leb)}<\delta.
        \label{eqn:t/2eps-bound}
    \end{equation}
    Due to the above, it remains to show that \eqref{eqn:lemtail-WTS} holds. By applying \lem{lem:L1/2WTS}, the result follows. Through similar estimates to \eqref{eqn:Leps-in-L0} and \eqref{eqn:SG-Lj}, for any $n\in\mathbb{N}$ and $f\in\BV(I)$ with $\|f\|_{\BV(I)}=1$, 
    \begin{equation}
    \left| \leb (\mathcal{L}_\omega^{\varepsilon\, (n)}f)\right|=\left|\leb\left( \sum_{j=1}^m\leb(f\cdot \mathds{1}_{I_j}) \phi_j\right) \right|+O_{n\to \infty}(\kappa^n)+o_{\varepsilon\to 0}(1). \label{eqn:spec-proj}   
    \end{equation}
    Fix $n_1\in\mathbb{N}$, then by the duality relation of the Perron-Frobenius operator, \eqref{eqn:spec-proj} implies that
    \begin{align}
        \left| \leb(\mathcal{L}_\omega^{\varepsilon\, (\bar{n}_{t}^{\varepsilon})}f)\right|&= \left| \leb(\mathcal{L}_{\sigma^{n_1}\omega}^{\varepsilon\, (\bar{n}_{t}^{\varepsilon}-n_1)}(\mathcal{L}_\omega^{\varepsilon\, ({n}_1)})(f)) 
 \right|\nonumber\\
 &= \left| \leb\left(\mathcal{L}_{\sigma^{n_1}\omega}^{\varepsilon\, (\bar{n}_{t}^{\varepsilon}-n_1)}\left(\sum_{j=1}^m \leb(f\cdot\mathds{1}_{I_j})\phi_j\right)
 \right)\right| +O_{n_1\to \infty}(\theta^{n_1})+o_{\varepsilon\to 0}(1).\phantom{asasfdsdf} \label{eqn:n1large}
    \end{align}}
Given the above, to obtain \eqref{eqn:lemtail-WTS} we need to show that 
\begin{equation}
    \int_{I_k} \mathcal{L}_{\sigma^{n_1}\omega}^{\varepsilon\, (\bar{n}_{t}^{\varepsilon}-n_1)}\left(\sum_{j=1}^m \leb(f\cdot\mathds{1}_{I_j})\phi_j\right)(x)\, \dleb(x) = \sum_{j=1}^m \leb(f\cdot\mathds{1}_{I_j})\int_{I_k}\mathcal{L}_{\sigma^{n_1}\omega}^{\varepsilon\, (\bar{n}_{t}^{\varepsilon}-n_1)}\left(\phi_j\right)(x)\, \dleb(x)
    \label{eqn:wts-small}
\end{equation}
is small for each $k\in \{1,\cdots, m\}$. However, by \thrm{thrm:conv_jump}, for $\mathbb{P}$-a.e. $\omega\in\Omega$
\begin{align*}
    \lim_{\varepsilon\to 0}\int_{I_k}\mathcal{L}_{\sigma^{n_1}\omega}^{\varepsilon\, (\bar{n}_{t}^{\varepsilon}-n_1)}\left(\phi_j\right)(x)\, \dleb(x)&
    = \lim_{\varepsilon\to 0}\mu_j\left((T_{\sigma^{n_1}\omega}^{\varepsilon\, (\frac{t}{2\varepsilon}-{n}_0-n_1)})^{-1}(I_k)\right)
    =p_{jk}\left(\frac{t}{2}\right)
\end{align*}
where $p_{jk}\left(\frac{t}{2}\right) := (e^{\frac{t}{2}\bar{G}})_{jk}$ denotes the transition probability of an initial condition starting in state $I_j$ , and jumping to state $I_k$ at time $\frac{t}{2}$ for the averaged Markov jump process introduced in \Sec{sec:avg_markov}.\footnote{Recall that $\bar{G}$ is the generator of the averaged Markov jump process defined in \Sec{sec:avg_markov}.} But, $\lim_{t\to \infty }p_{jk}\left(\frac{t}{2}\right) = p_k$, where $p_k$ is as in \eqref{eqn:obs-centering2}. So, from \eqref{eqn:wts-small}, for $\mathbb{P}$-a.e. $\omega\in\Omega$,
\begin{align*}
    \lim_{t \to \infty}\lim_{\varepsilon\to 0}\left|\sum_{j=1}^m \leb(f\cdot\mathds{1}_{I_j})\int_{I_k} \mathcal{L}_{\sigma^{n_1}\omega}^{\varepsilon\, (\bar{n}_{t}^{\varepsilon}-n_1)}\left(   \phi_j(x)\right)\, \dleb(x)\right| 
    &=0
\end{align*}
since ${f}\in\BV_0(I)$. Therefore, taking $n_1$ large in \eqref{eqn:n1large}, we obtain \eqref{eqn:lemtail-WTS}. 
\end{proof}
We may now prove the main result of this section.
\begin{proof}[Proof of \thrm{thrm:variance-est}]   
\Change{The proof of \thrm{thrm:variance-est} is divided into several steps. In \Step{step:clt-lemma}, we apply \cite[Theorem B]{DFGTV_limthrm} to show that \eqref{eqn:clt-limit} holds. The remaining steps concern the computation of \eqref{eqn:var-limit}. \Step{stp:var-step1} and \Step{stp:esttail} show that various terms in $\varepsilon(\Sigma^\varepsilon(\tilde{\psi}^\varepsilon))$ tend to zero with $\varepsilon$. In \Step{stp:mixing} and \Step{stp:err-ctrl}, we compute estimates for the remaining terms, which we use in \Step{step:simplify} and \Step{step:b4-sum} to obtain \eqref{eqn:var-limit}.   
\begin{stp}
    In the setting of \thrm{thrm:variance-est}, for every bounded and continuous function $f:\mathbb{R}\to \mathbb{R}$ and $\mathbb{P}$-a.e. $\omega\in\Omega$, we have 
    \begin{equation*}
        \lim_{n\to \infty} \int_{\mathbb{R}}f\left( \frac{1}{\sqrt{n}}\sum_{k=0}^{n-1}(\tilde{\psi}^\varepsilon_{\sigma^{k}\omega}\circ T_{\omega}^{\varepsilon\, (k)})\right)(x)\, d\mu_{\omega}^\varepsilon(x) = \int_\mathbb{R}f(x)\, d\mathcal{N}(0,(\Sigma^\varepsilon(\tilde{\psi}^\varepsilon))^2)(x).
    \end{equation*}    
    \label{step:clt-lemma}
\end{stp}
\begin{proof}
    \Change{It suffices to check that for each $\varepsilon>0$ sufficiently small, \cite[Theorem B]{DFGTV_limthrm} applies to the sequence of random dynamical systems $\{(\Omega,\mathcal{F},\mathbb{P},\sigma,\BV(I),\mathcal{L}^\varepsilon)\}_{\varepsilon> 0}$ of metastable maps $T_\omega^\varepsilon:I\to I$. Due to \cite[Section 2.3.1]{DFGTV_limthrm}, for a fixed $\varepsilon>0$, if \hyperref[list:P1]{\textbf{(P1)}}-\hyperref[list:P7]{\textbf{(P7)}} hold, and $(T_\omega^\varepsilon)_{\omega\in\Omega}$ is uniformly (over $\omega\in\Omega$) covering, then \cite[Theorem B]{DFGTV_limthrm} applies. By assumption, \hyperref[list:P1]{\textbf{(P1)}}-\hyperref[list:P7]{\textbf{(P7)}} hold, and thus it remains to show that $(T_\omega^\varepsilon)_{\omega\in\Omega}$ is uniformly (over $\omega\in\Omega$) covering. However, thanks to \hyperref[list:P4]{\textbf{(P4)}}, particularly the fact that the measure of the holes for $T_\omega^\varepsilon$ are bounded uniformly over $\omega\in\Omega$ away from zero, this is satisfied.} 
\end{proof}
    } 
    \begin{stp} In the setting of \thrm{thrm:variance-est}, 
    $$\lim_{\varepsilon\to 0}\varepsilon \int_\Omega \int_I \tilde{\psi}_\omega^\varepsilon(x)^2\phi_\omega^\varepsilon(x)\, \dleb(x)\, d\mathbb{P}(\omega)=0.$$
    \label{stp:var-step1}
    \end{stp}
    \begin{proof}
    Due to \eqref{eqn:obs-reg}, there exists a constant $C>0$ such that $\esssup_{\omega\in\Omega}\| \psi_\omega\|_{\BV(I)}\leq C$. Therefore, recalling \rem{rmk:bv-center},
    \begin{align*}
        \varepsilon\left| \int_\Omega \int_I \tilde{\psi}_\omega^\varepsilon(x)^2\phi_\omega^\varepsilon(x)\, \dleb(x)\, d\mathbb{P}(\omega) \right|
        &\leq 4\varepsilon C^2. 
    \end{align*}
    Taking $\varepsilon\to 0$, the result follows.
    \end{proof}
    \begin{stp} In the setting of \thrm{thrm:variance-est},  
 \begin{align*}
     \lim_{t\to \infty}\lim_{\varepsilon\to 0}\mathcal{T}^\varepsilon_t:= \lim_{t\to \infty}\lim_{\varepsilon\to 0}2\varepsilon\sum_{n=\frac{t}{\varepsilon}+1}^\infty\int_\Omega \int_I \mathcal{L}_\omega^{\varepsilon\, (n)}(\tilde{\psi}_{\omega}^\varepsilon(x)\phi_\omega^\varepsilon(x)) \tilde{\psi}_{\sigma^n\omega}^\varepsilon(x)\, \, \dleb(x)\, d\mathbb{P}(\omega) =0.
 \end{align*}    
 \label{stp:esttail}
 
\end{stp}
\begin{proof}
Recall that due to \rem{rmk:bv-center}, for $\mathbb{P}$-a.e. $\omega \in \Omega$, 
\begin{align}
    \left|\mathcal{T}^\varepsilon_t\right|&\leq 4C\varepsilon\ \sum_{n=\frac{t}{\varepsilon}+1}^\infty\int_\Omega \int_I |\mathcal{L}_\omega^{\varepsilon\, (n)}(\tilde{\psi}_{\omega}^\varepsilon(x)\phi_\omega^\varepsilon(x))| \, \, \dleb(x)\, d\mathbb{P}(\omega) \nonumber\\
    &\leq 4C\varepsilon \sum_{n=\frac{t}{\varepsilon}+1}^\infty\esssup_{\omega\in\Omega} \|\mathcal{L}_\omega^{\varepsilon\, (n)}(\tilde{\psi}_{\omega}^\varepsilon\phi_\omega^\varepsilon)\|_{\BV(I)}.\label{eqn:tail-line}
\end{align}
Observe that 
\begin{align*}
    \sum_{n=\frac{t}{\varepsilon}+1}^\infty\esssup_{\omega\in\Omega} \|\mathcal{L}_\omega^{\varepsilon\, (n)}(\tilde{\psi}_{\omega}^\varepsilon\phi_\omega^\varepsilon)\|_{\BV(I)}
    &= \sum_{j=1}^\infty \sum_{k=0}^{\frac{t}{\varepsilon}-1} \esssup_{\omega\in\Omega}\|\mathcal{L}_\omega^{\varepsilon\, (k+\frac{jt}{\varepsilon}+1)}(\tilde{\psi}_{\omega}^\varepsilon\phi_\omega^\varepsilon)\|_{\BV(I)}\\
    &=\sum_{j=1}^\infty \sum_{k=0}^{\frac{t}{\varepsilon}-1} \esssup_{\omega\in\Omega}\|\mathcal{L}_{\sigma^{\frac{jt}{\varepsilon}}\omega}^{\varepsilon\, (k+1)}(\mathcal{L}_\omega^{\varepsilon\, (\frac{jt}{\varepsilon})}(\tilde{\psi}_{\omega}^\varepsilon\phi_\omega^\varepsilon))\|_{\BV(I)}.
\end{align*}
By \hyperref[list:P5]{\textbf{(P5)}}, there exists a constant $M>0$ such that for $\mathbb{P}$-a.e. $\omega\in\Omega$   
\Change{\begin{align*}
    \sum_{j=1}^\infty \sum_{k=0}^{\frac{t}{\varepsilon}-1} &\esssup_{\omega\in\Omega}\|\mathcal{L}_{\sigma^{\frac{jt}{\varepsilon}}\omega}^{\varepsilon\, (k+1)}(\mathcal{L}_\omega^{\varepsilon\, (\frac{jt}{\varepsilon})}(\tilde{\psi}_{\omega}^\varepsilon\phi_\omega^\varepsilon))\|_{\BV(I)} \\&\leq M \sum_{j=1}^\infty \sum_{k=0}^{\frac{t}{\varepsilon}-1} \esssup_{\omega\in\Omega}\|\mathcal{L}_\omega^{\varepsilon\, (\frac{jt}{\varepsilon})}(\tilde{\psi}_{\omega}^\varepsilon\phi_\omega^\varepsilon)\|_{\BV(I)} \\
    &=\frac{Mt\|\tilde{\psi}_{\omega}^\varepsilon\phi_\omega^\varepsilon\|_{\BV(I)}}{\varepsilon}\sum_{j=1}^\infty \esssup_{\omega\in\Omega}\Bigg\|\mathcal{L}_\omega^{\varepsilon\, (\frac{jt}{\varepsilon})}\Bigg(\frac{\tilde{\psi}_{\omega}^\varepsilon\phi_\omega^\varepsilon}{\|\tilde{\psi}_{\omega}^\varepsilon\phi_\omega^\varepsilon\|_{\BV(I)}}\Bigg)\Bigg\|_{\BV(I)}. 
\end{align*}
Due to \rem{rmk:bv-center} and \cite[Lemma 4.2]{gtp_met}, $\tilde{\psi}_{\omega}^\varepsilon\phi_\omega^\varepsilon$ has uniformly bounded $\BV(I)$ norm over $\varepsilon>0$ sufficiently small, and over $\omega\in\Omega$ away from a $\mathbb{P}$-null set. Thus, applying \lem{lem:tail-ctrl} with $f=\frac{\tilde{\psi}_{\omega}^\varepsilon\phi_\omega^\varepsilon}{\|\tilde{\psi}_{\omega}^\varepsilon\phi_\omega^\varepsilon\|_{\BV(I)}}$, there exist constants $\rho\in (0,1)$ and $\tilde{M}>0$ such that for $\varepsilon>0$ sufficiently small and $\mathbb{P}$-a.e. $\omega\in\Omega$
\begin{align}
    \frac{Mt}{\varepsilon}\sum_{j=1}^\infty \esssup_{\omega\in\Omega}\left\|\mathcal{L}_\omega^{\varepsilon\, (\frac{jt}{\varepsilon})}\left(\frac{\tilde{\psi}_{\omega}^\varepsilon\phi_\omega^\varepsilon}{\|\tilde{\psi}_{\omega}^\varepsilon\phi_\omega^\varepsilon\|_{\BV(I)}}\right)\right\|_{\BV(I)}\|\tilde{\psi}_{\omega}^\varepsilon\phi_\omega^\varepsilon\|_{\BV(I)}&\leq \frac{\tilde{M}t}{\varepsilon}\sum_{j=1}^\infty (\rho^t)^j\nonumber\\
    &= \frac{\tilde{M}t\rho^t}{\varepsilon(1-\rho^t)}\label{eqn:tail-estlem}.
\end{align}
In turn, combining \eqref{eqn:tail-line} and \eqref{eqn:tail-estlem}, we find that there exists a constant $\tilde{C}>0$ such that
$$\lim_{\varepsilon\to 0}|\mathcal{T}_{t}^\varepsilon|\leq \tilde{C}\frac{t\rho^t}{1-\rho^t},$$
which is small for $t$ large. }
\end{proof}
From \Step{stp:var-step1} and \Step{stp:esttail},
\begin{equation}
    \lim_{\varepsilon\to 0}\varepsilon(\Sigma^\varepsilon(\tilde{\psi}^\varepsilon))^2=\lim_{t\to\infty}\lim_{\varepsilon\to 0}2\varepsilon\sum_{n=1}^{\frac{t}{\varepsilon}} \int_\Omega \int_I \mathcal{L}_\omega^{\varepsilon\,(n)}(\tilde{\psi}_{\omega}^\varepsilon(x)\phi_\omega^\varepsilon(x)) \tilde{\psi}_{\sigma^n\omega}^\varepsilon(x)\, \, \dleb(x)\, d\mathbb{P}(\omega).\label{eqn:lim-after-steps}
\end{equation}
Further, from \eqref{eqn:lim-after-steps} observe that for $K\in\mathbb{N}$ fixed,
\begin{equation}
    \lim_{\varepsilon\to 0}\varepsilon(\Sigma^\varepsilon(\tilde{\psi}^\varepsilon))^2=\lim_{t\to\infty}\lim_{\varepsilon\to 0}2\varepsilon\sum_{n=K}^{\frac{t}{\varepsilon}} \int_\Omega \int_I \mathcal{L}_\omega^{\varepsilon\,(n)}(\tilde{\psi}_{\omega}^\varepsilon(x)\phi_\omega^\varepsilon(x)) \tilde{\psi}_{\sigma^n\omega}^\varepsilon(x)\, \, \dleb(x)\, d\mathbb{P}(\omega).\label{eqn:lim-after-steps_K}
\end{equation}
The following steps allow us to compute \eqref{eqn:lim-after-steps}.
\begin{stp}
Fix $n_0\in\mathbb{N}$. In the setting of \thrm{thrm:variance-est} and for $n\geq 2n_0$, let 
\begin{equation}
    \chi_{n,\omega}^\varepsilon:= \sum_{j=1}^m\left(\int_{I_j}\tilde{\psi}_{\omega}^\varepsilon(x)\phi_\omega^\varepsilon(x)\, \dleb(x)\right) \mathcal{L}_{\sigma^{n_0}\omega}^{\varepsilon\, (n-2n_0)}(\phi_j).\label{eqn:chi}
\end{equation}
Then for $\mathbb{P}$-a.e. $\omega\in\Omega$,
\begin{align*}
    \int_{I} \mathcal{L}_{\sigma^{n-n_0}\omega}^{\varepsilon\, (n_0)}(\chi_{n,\omega}^\varepsilon)(x)\tilde{\psi}_{\sigma^n\omega}^\varepsilon(x)\, \, \dleb(x)&= \sum_{k=1}^m \int_{I_k} \phi_k(x) \tilde{\psi}_{\sigma^n\omega}^\varepsilon(x) \, \dleb(x)\left(\int_{I_k} \chi_{n,\omega}^\varepsilon(x)\, \dleb(x) \right) \\
   &\quad +O_{n_0\to\infty}(\kappa^{n_0})+  o_{\varepsilon\to 0}(1). 
\end{align*}
\label{stp:mixing}
\end{stp}\begin{proof}
    Observe that for $n\geq 2n_0$ 
    \begin{align}
        \int_{I} \mathcal{L}_{\sigma^{n-n_0}\omega}^{\varepsilon\, (n_0)}(\chi_{n,\omega}^\varepsilon)(x)\tilde{\psi}_{\sigma^n\omega}^\varepsilon(x)\, \, \dleb(x) = \sum_{k=1}^m\int_{I_k} \mathcal{L}_{\sigma^{n-n_0}\omega}^{\varepsilon\, (n_0)}(\chi_{n,\omega}^\varepsilon)(x)\tilde{\psi}_{\sigma^n\omega}^\varepsilon(x)\, \, \dleb(x).\label{eqn:mix-decomp}
    \end{align}
    For each $k\in\{1,\cdots,m\}$, due to \eqref{eqn:SG-Lj},
    \begin{align}
        \int_{I_k} \mathcal{L}_{}^{0\, (n_0)}(\chi_{n,\omega}^\varepsilon)(x)\tilde{\psi}_{\sigma^n\omega}^\varepsilon(x)\, \, \dleb(x)&= \int_{I_k} \mathcal{L}_{k}^{0\, (n_0)}(\chi_{n,\omega}^\varepsilon)(x)\tilde{\psi}_{\sigma^n\omega}^\varepsilon(x)\, \, \dleb(x)\nonumber\\
        &=\leb(\chi_{n,\omega}^\varepsilon\cdot\mathds{1}_{I_k})\int_{I_k}\phi_k(x)\tilde{\psi}_{\sigma^n\omega}^\varepsilon(x)\, \dleb(x)\nonumber \\
    &\qquad+O_{n_0\to \infty}(\kappa^{n_0})\|\chi_{n,\omega}^\varepsilon\|_{\BV(I)}.\label{eqn:mixing-est1}
    \end{align}
   Due to \hyperref[list:P5]{\textbf{(P5)}} and \eqref{eqn:obs-reg}, $\|\chi_{n,\omega}^\varepsilon\|_{\BV(I)}$ is bounded uniformly over $\omega\in\Omega$, $\varepsilon>0$ and $n\in\mathbb{N}$. Thus, for each $k\in\{1,\cdots, m\}$, due to \rem{rmk:bv-center}, since $n_0\in\mathbb{N}$ is fixed,
    \begin{align*}
        &\int_{I_k} (\mathcal{L}_{\sigma^{n-n_0}\omega}^{\varepsilon\, (n_0)}- \mathcal{L}_{}^{0\, (n_0)})(\chi_{n,\omega}^\varepsilon)(x)\tilde{\psi}_{\sigma^n\omega}^\varepsilon(x)\, \, \dleb(x) \\
        &\quad\leq 2C\|\chi_{n,\omega}^\varepsilon\|_{\BV(I)}\|(\mathcal{L}_{\sigma^{n-n_0}\omega}^{\varepsilon\, (n_0)}- \mathcal{L}_{}^{0\, (n_0)})(\mathds{1}_{I})\|_{L^1(\leb)}= o_{\varepsilon\to 0}(1).
    \end{align*}
      \Change{In the last line we have used \eqref{eqn:obs-reg}, \eqref{eqn:Ltrip-iter} and \rem{rem:p1-imp-trip}. Due to \eqref{eqn:mixing-est1}, and recalling that $\|\chi_{n,\omega}^\varepsilon\|_{\BV(I)}$ is bounded uniformly over $\omega\in\Omega$, $\varepsilon>0$ and $n\in\mathbb{N}$,} we find that for each $k\in\{1,\cdots, m\}$,
      \begin{align*}
       &\int_{I_k} \mathcal{L}_{\sigma^{n-n_0}\omega}^{\varepsilon\, (n_0)}(\chi_{n,\omega}^\varepsilon)(x)\tilde{\psi}_{\sigma^n\omega}^\varepsilon(x)\, \dleb(x)\\
       &\quad= \int_{I_k} \mathcal{L}_{}^{0\, (n_0)}(\chi_{n,\omega}^\varepsilon)(x)\tilde{\psi}_{\sigma^n\omega}^\varepsilon(x)\, \dleb(x) +o_{\varepsilon\to 0}(1)\\
       &\quad=\leb(\chi_{n,\omega}^\varepsilon\cdot\mathds{1}_{I_k})\int_{I_k}\phi_k(x)\tilde{\psi}_{\sigma^n\omega}^\varepsilon(x)\, \dleb(x) +O_{n_0\to \infty}(\kappa^{n_0}) +o_{\varepsilon\to 0}(1).
      \end{align*}
      Summing over $k\in\{1,\cdots,m\}$, the result follows.
\end{proof}
In the remainder of the proof, we will make use of \thrm{thrm:conv_jump}. It is crucial that we study the average (over $\omega\in\Omega$) speed of convergence of the distribution of jumps for $T_\omega^\varepsilon$ as $\varepsilon\to0$. 
\begin{stp}
 In the setting of \thrm{thrm:conv_jump}, set
\begin{align*}
    E_{j,\omega}^\varepsilon&:=\Bigg|\mu_j\left(\left\{x\in I \ \big| \ \varepsilon\mathcal{T}_{k,\omega}^\varepsilon(x)\in \Delta_k \ \mathrm{and} \ z(T_{\omega}^{\varepsilon \, (t_{k,\omega}^\varepsilon(x))}(x))=r_k \ \mathrm{for} \ k=1,\dots, p\right\}  \right) \\
    &\quad -  \mathbb{P}^j(\mathcal{T}_k^M\in \Delta_k\ \mathrm{and} \ z_k^M=r_k \ \mathrm{for}\, k=1,\dots,p) \Bigg|.
\end{align*}
 Then, 
\begin{align*}
&\lim_{\varepsilon\to 0}\int_\Omega E_{j,\omega}^\varepsilon \, d\mathbb{P}(\omega) = 0.
\end{align*}   
\label{stp:err-ctrl}
\end{stp}
\begin{proof}
\Change{Observe that for all $\varepsilon>0$ sufficiently small and each $\omega\in \Omega$, $|E_{j,\omega}^\varepsilon|\leq 2$. Thus, by \thrm{thrm:conv_jump} and the dominated convergence theorem, the result follows.}
\end{proof}\Change{
\begin{stp}In the setting of \thrm{thrm:variance-est}, for a \Change{fixed $n_0\in\mathbb{N}$} and $\mathbb{P}$-a.e. $\omega\in \Omega$, if $n\geq 2n_0$
\begin{align*}
    \Sigma_{n,\omega}^\varepsilon&:= \int_I \mathcal{L}_\omega^{\varepsilon\, (n)}(\tilde{\psi}_{\omega}^\varepsilon(x)\phi_\omega^\varepsilon(x)) \tilde{\psi}_{\sigma^n\omega}^\varepsilon(x)\, \, \dleb(x)\\
    &= \sum_{j,k=1}^m p_j (\Psi_{\sigma^n\omega}(k)+o_{\varepsilon\to 0}(1))\left(\Psi_\omega(j)+o_{\varepsilon\to 0}(1)\right)\mu_j(T_{\sigma^{n_0}\omega}^{\varepsilon\, (n-2n_0)}(x)\in I_k) \\
    &\quad+O_{n_0\to\infty}(\kappa^{n_0})+  o_{\varepsilon\to 0}(1).
\end{align*}
\label{step:simplify}
\end{stp}}
\begin{proof}   
Due to \eqref{eqn:spec-proj}, observe that for $\mathbb{P}$-a.e. $\omega\in\Omega$, if $n\geq 2n_0$
\begin{align*}
    \Sigma_{n,\omega}^\varepsilon&:= \int_I \mathcal{L}_\omega^{\varepsilon\, (n)}(\tilde{\psi}_{\omega}^\varepsilon(x)\phi_\omega^\varepsilon(x)) \tilde{\psi}_{\sigma^n\omega}^\varepsilon(x)\, \, \dleb(x)\\
    &= \int_I \mathcal{L}_{\sigma^{n_0}\omega}^{\varepsilon\, (n-2n_0)}(\mathcal{L}_\omega^{\varepsilon\, (n_0)}(\tilde{\psi}_{\omega}^\varepsilon(x)\phi_\omega^\varepsilon(x))) \tilde{\psi}_{\sigma^n\omega}^\varepsilon(T_{\sigma^{n-n_0}\omega}^{\varepsilon\, (n_0)}(x))\, \, \dleb(x)\\
    &=\int_I \chi_{n,\omega}^\varepsilon(x)\tilde{\psi}_{\sigma^n\omega}^\varepsilon(T_{\sigma^{n-n_0}\omega}^{\varepsilon\, (n_0)}(x))\, \dleb(x) +O_{n_0\to \infty}(\kappa^{n_0})+o_{\varepsilon\to 0}(1) 
\end{align*}
where $\chi_{n,\omega}^\varepsilon$ is as in \eqref{eqn:chi}. Now, thanks to \Step{stp:mixing}, 
\begin{align*}
     \Sigma_{n,\omega}^\varepsilon&= \int_I \mathcal{L}_{\sigma^{n-n_0}\omega}^{\varepsilon\, (n_0)}(\chi_{n,\omega}^\varepsilon)(x)\tilde{\psi}_{\sigma^n\omega}^\varepsilon(x)\, \dleb(x) +O_{n_0\to \infty}(\kappa^{n_0})+o_{\varepsilon\to 0}(1)\\
     &= \sum_{k=1}^m \int_{I_k} \tilde{\psi}_{\sigma^n\omega}^\varepsilon(x)\phi_k(x)  \, \dleb(x)\left(\int_{I_k} \chi_{n,\omega}^\varepsilon(x)\, \dleb(x) \right)  +O_{n_0\to\infty}(\kappa^{n_0})+  o_{\varepsilon\to 0}(1)\\
   &= \sum_{j,k=1}^m \left( \int_{I_k} \tilde{\psi}_{\sigma^n\omega}^\varepsilon(x)\phi_k(x)  \, \dleb(x) \right)\left(\int_{I_j}\tilde{\psi}_{\omega}^\varepsilon(x)\phi_\omega^\varepsilon(x)\, \dleb(x)\right)\\
   &\qquad \times \int_{I_k}  \mathcal{L}_{\sigma^{n_0}\omega}^{\varepsilon\, (n-2n_0)}(\phi_j)(x)\, \dleb(x) +O_{n_0\to\infty}(\kappa^{n_0})+  o_{\varepsilon\to 0}(1)
\end{align*}
Recall from the statement of \thrm{thrm:variance-est} that $\Psi_{\omega}(k):= \int_{I_k} \phi_k(x) {\psi}_{\omega}(x) \, \dleb(x)$. Thus, using \eqref{eqn:obs-centering1}
\begin{align*}
    \Sigma_{n,\omega}^\varepsilon&= \sum_{j,k=1}^m \left( \Psi_{\sigma^n\omega}(k) - \mu_{\sigma^n\omega}^\varepsilon(\psi_{\sigma^n\omega}) \right)\left(\int_{I_j}({\psi}_{\omega}(x)-\mu_\omega^\varepsilon(\psi_\omega))\phi_\omega^\varepsilon(x)\, \dleb(x)\right)\\
   &\qquad \times \int_{I_k}  \mathcal{L}_{\sigma^{n_0}\omega}^{\varepsilon\, (n-2n_0)}(\phi_j)(x)\, \dleb(x) +O_{n_0\to\infty}(\kappa^{n_0})+  o_{\varepsilon\to 0}(1).
\end{align*}
Now, by \cite[Theorem 7.2]{gtp_met}, $$\int_{I_j} \psi_\omega(x)\phi_\omega^\varepsilon(x)\, \dleb(x)=p_j\int_{I_j} \psi_\omega(x)\phi_j(x)\, \dleb(x)+o_{\varepsilon\to 0}(1)=p_j\Psi_\omega(j)+o_{\varepsilon\to 0}(1).$$ \Change{So, recalling that for $\mathbb{P}$-a.e. $\omega\in\Omega$, $\mu_\omega^\varepsilon(\psi_\omega)=\sum_{i=1}^m p_i\Psi_\omega(i)+o_{\varepsilon\to 0}(1)$}, the result follows since
\begin{align*}
    \Sigma_{n,\omega}^\varepsilon
   &=\sum_{j,k=1}^m \left( \Psi_{\sigma^n\omega}(k) - \sum_{i=1}^m p_i \Psi_{\sigma^n\omega}(i) + o_{\varepsilon\to 0}(1)\right)\left( p_j\left(\Psi_\omega(j) -\sum_{i=1}^m p_i \Psi_\omega(i)\right)+o_{\varepsilon\to 0}(1)\right)\\
   &\quad \times\int_{I_k}  \mathcal{L}_{\sigma^{n_0}\omega}^{\varepsilon\, (n-2n_0)}(\phi_j)(x)\, \dleb(x) +O_{n_0\to\infty}(\kappa^{n_0})+  o_{\varepsilon\to 0}(1) \\
   &=\sum_{j,k=1}^m p_j (\Psi_{\sigma^n\omega}(k)+o_{\varepsilon\to 0}(1))\left(\Psi_\omega(j)+o_{\varepsilon\to 0}(1)\right)\mu_j(T_{\sigma^{n_0}\omega}^{\varepsilon\, (n-2n_0)}(x)\in I_k) \\
    &\qquad+O_{n_0\to\infty}(\kappa^{n_0})+  o_{\varepsilon\to 0}(1). 
\end{align*}
In the last line we have used \eqref{eqn:obs-centering2} which implies that $\sum_{i=1}^m p_i \Psi_\omega(i)=0$ for $\mathbb{P}$-a.e. $\omega\in\Omega$.\end{proof}
We may now compute \eqref{eqn:lim-after-steps}.
\Change{\begin{stp}
In the setting of \thrm{thrm:variance-est},  
\begin{equation}
     \lim_{\varepsilon\to 0}\varepsilon(\Sigma^\varepsilon(\tilde{\psi}^\varepsilon))^2=2\sum_{j,k=1}^mp_j\int_\Omega \Psi_\omega(j)\, d\mathbb{P}(\omega)\int_\Omega \Psi_\omega(k)\, d\mathbb{P}(\omega)\int_0^\infty p_{jk}(s)\, \dleb(s). \label{eqn:S7lim}
\end{equation}
\label{step:b4-sum}
\end{stp}}
\begin{proof}\jpt{
Fix $n_0\in\mathbb{N}$. Due to \Step{step:simplify} and \eqref{eqn:lim-after-steps_K} (with $K=2n_0$), \eqref{eqn:lim-after-steps} becomes
\begin{align}
    \lim_{\varepsilon\to 0}\varepsilon(\Sigma^\varepsilon(\tilde{\psi}^\varepsilon))^2&= \sum_{j,k=1}^m\lim_{t\to \infty}\lim_{\varepsilon\to 0}\Bigg(2p_j\varepsilon  \int_\Omega\left(\Psi_\omega(j)+o_{\varepsilon\to 0}(1)\right) \sum_{n=2n_0}^{\frac{t}{\varepsilon}}\Bigg((\Psi_{\sigma^n\omega}(k)+o_{\varepsilon\to 0}(1))\nonumber\\
    &\quad\times\mu_j(T_{\sigma^{n_0}\omega}^{\varepsilon\, (n-2n_0)}(x)\in I_k)\Bigg) \,d\mathbb{P}(\omega)+ 2{t}{}\left( O_{n_0\to\infty}(\kappa^{n_0})+  o_{\varepsilon\to 0}(1)  \right)\Bigg).\label{eqn:b4-sub-split}
\end{align}
Fix $\delta>0$. From \eqref{eqn:b4-sub-split}, we write
\begin{align}
    &\sum_{n=2n_0}^{\frac{t}{\varepsilon}}(\Psi_{\sigma^n\omega}(k)+o_{\varepsilon\to 0}(1))\mu_j(T_{\sigma^{n_0}\omega}^{\varepsilon\, (n-2n_0)}(x)\in I_k)\nonumber\\&\qquad= \sum_{u=0}^{\frac{t}{\delta}-1}\sum_{i=0}^{\frac{\delta}{\varepsilon}-2n_0}(\Psi_{\sigma^{i+\frac{u\delta}{\varepsilon}+2n_0}\omega}(k)+o_{\varepsilon\to 0}(1))\mu_j\nonumber( T_{\sigma^{n_0}\omega}^{\varepsilon\, (i+\frac{u\delta}{\varepsilon})}(x)\in I_k)\nonumber\\
    &\qquad =\sum_{u=0}^{\frac{t}{\delta}-1}\sum_{i=0}^{\frac{\delta}{\varepsilon}-2n_0}(\Psi_{\sigma^{i+\frac{u\delta}{\varepsilon}+2n_0}\omega}(k)+o_{\varepsilon\to 0}(1))(p_{jk}(u\delta)+E_{j,\omega}^\varepsilon) \label{eqn:var-jump-pro}
\end{align}
where in the last line we have applied \thrm{thrm:conv_jump}, and $E_{j,\omega}^\varepsilon$ is the error term appearing in \Step{stp:err-ctrl} satisfying $\lim_{\varepsilon\to 0}\int_\Omega E_{j,\omega}^\varepsilon\, d\mathbb{P}(\omega)=0$. For $j,k\in\{1,\cdots, m\}$, let 
\begin{align}
    L_{j,k,\delta,t}^\varepsilon&:= 2p_j\varepsilon  \int_\Omega\Psi_\omega(j)\sum_{u=0}^{\frac{t}{\delta}-1}\sum_{i=0}^{\frac{\delta}{\varepsilon}-2n_0}(\Psi_{\sigma^{i+\frac{u\delta}{\varepsilon}+2n_0}\omega}(k)+o_{\varepsilon\to 0}(1))p_{jk}(u\delta)\,d\mathbb{P}(\omega)\\ 
    V_{j,k,\delta,t}^\varepsilon&:=2p_j\varepsilon  \int_\Omega\Psi_\omega(j)\sum_{u=0}^{\frac{t}{\delta}-1}\sum_{i=0}^{\frac{\delta}{\varepsilon}-2n_0}(\Psi_{\sigma^{i+\frac{u\delta}{\varepsilon}+2n_0}\omega}(k)+o_{\varepsilon\to 0}(1))E_{j,\omega}^\varepsilon\nonumber\\ 
 &\qquad+o_{\varepsilon\to 0}(1)\sum_{u=0}^{\frac{t}{\delta}-1}\sum_{i=0}^{\frac{\delta}{\varepsilon}-2n_0}(\Psi_{\sigma^{i+\frac{u\delta}{\varepsilon}+2n_0}\omega}(k)+o_{\varepsilon\to 0}(1))(p_{jk}(u\delta)+E_{j,\omega}^\varepsilon) \,d\mathbb{P}(\omega).
\end{align}
In this way, by \eqref{eqn:var-jump-pro}, \eqref{eqn:b4-sub-split} becomes 
\begin{align}
    \lim_{\varepsilon\to 0}\varepsilon(\Sigma^\varepsilon(\tilde{\psi}^\varepsilon))^2
    &=\sum_{j,k=1}^m\lim_{t\to \infty}\lim_{\varepsilon\to 0}\Bigg( L_{j,k,\delta,t}^\varepsilon+V_{j,k,\delta,t}^\varepsilon+ 2{t}{}\left( O_{n_0\to\infty}(\kappa^{n_0})+  o_{\varepsilon\to 0}(1)  \right)\Bigg).\label{eqn:var-error-est}
\end{align}
By \eqref{eqn:obs-reg}, $\Psi_\omega(k)= \int_{I_k} \psi_\omega(x)\phi_k(x)\, \dleb(x) \leq \esssup_{\omega \in\Omega}\|\psi_\omega\|_{\BV(I)}\leq C$. Thus, 
\begin{align*}
 V_{j,k,\delta,t}^\varepsilon
 &\leq  2p_j\left(t - \frac{\varepsilon t (2n_0 -1) }{\delta}\right) \left({ (C^2+o_{\varepsilon\to 0}(1))}{}\int_\Omega { E_{j,\omega}^\varepsilon}{}\, d\mathbb{P}(\omega) + o_{\varepsilon\to 0}(1){ }{}\right)\\
 &=\left(t - \frac{\varepsilon t (2n_0 -1) }{\delta}\right)o_{\varepsilon\to 0}(1),
\end{align*}
where in the last equality we have used \Step{stp:err-ctrl}. Therefore, we can rewrite \eqref{eqn:var-error-est} as 
\begin{align}
    \lim_{\varepsilon\to 0}\varepsilon(\Sigma^\varepsilon(\tilde{\psi}^\varepsilon))^2&=\sum_{j,k=1}^m\lim_{t\to\infty}\lim_{\varepsilon\to 0}\Bigg(L_{j,k,\delta,t}^\varepsilon+2{t}{}\left( O_{n_0\to\infty}(\kappa^{n_0})+  o_{\varepsilon\to 0}(1)  \right)\Bigg)\label{eqn:finalvarlim}.
\end{align}
Now, observe that for all $\delta>0$ 
\begin{align}
    \lim_{\varepsilon\to 0} L_{j,k,\delta,t}^\varepsilon&=\lim_{\varepsilon\to 0}2p_j\varepsilon  \int_\Omega\Psi_\omega(j)\sum_{u=0}^{\frac{t}{\delta}-1}\sum_{i=0}^{\frac{\delta}{\varepsilon}-2n_0}(\Psi_{\sigma^{i+\frac{u\delta}{\varepsilon}+2n_0}\omega}(k)+o_{\varepsilon\to 0}(1))p_{jk}(u\delta)\,d\mathbb{P}(\omega)\nonumber\\ & =\lim_{\varepsilon\to 0}2p_j(\delta-2\varepsilon n_0+\varepsilon)  \int_\Omega\Psi_\omega(j)\sum_{u=0}^{\frac{t}{\delta}-1}p_{jk}(u\delta)\frac{\varepsilon}{\delta-2\varepsilon n_0+\varepsilon }\\
    &\qquad\times\sum_{i=0}^{\frac{\delta}{\varepsilon}-2n_0}(\Psi_{\sigma^{i+\frac{u\delta}{\varepsilon}+2n_0}\omega}(k)+o_{\varepsilon\to 0}(1))\,d\mathbb{P}(\omega)\nonumber\\
     & \stackrel{(\star)}{=}2p_j  \int_\Omega\Psi_\omega(j)\delta\sum_{u=0}^{\frac{t}{\delta}-1}p_{jk}(u\delta)\,d\mathbb{P}(\omega)\int_\Omega \Psi_\omega(k)\, d\mathbb{P}(\omega)\nonumber\\
    & \stackrel{(\star\star)}{=} 2p_j\int_\Omega \Psi_\omega(j)\, d\mathbb{P}(\omega)\int_\Omega \Psi_\omega(k)\, d\mathbb{P}(\omega)\Bigg(\int_0^t p_{jk}(s)\, \dleb(s) +o_{\delta\to 0}(1)\Bigg). \label{eqn:main-lim}
\end{align}}
At $(\star)$ we have used the dominated convergence theorem (thanks to \eqref{eqn:obs-reg}) along with Birkhoff's ergodic theorem, whilst at $(\star\star)$ we used the fact that $\delta\sum_{u=0}^{\frac{t}{\delta}-1}p_{jk}(u\delta)$ is a Riemann sum. Therefore, using \eqref{eqn:main-lim}, \Change{since $\delta>0$ and $n_0\in\mathbb{N}$ are fixed}, taking limits in \eqref{eqn:finalvarlim} in the order of $\varepsilon \to 0$, $\delta\to 0$, $n_0\to \infty$, and then $t\to \infty$, the result follows.
 \end{proof}
To conclude the proof of \thrm{thrm:variance-est}, since $p_{jk}(s):= (e^{s\bar{G}})_{jk}$,\footnote{Recall that $\bar{G}$ is the generator for the averaged Markov jump process defined in \Sec{sec:avg_markov}} \Change{by computing the sum over $j,k\in\{1,\cdots, m\}$ in \eqref{eqn:S7lim}, we obtain \eqref{eqn:var-limit}.}
\end{proof}
\section{Example: Random paired tent maps}
\label{sec:ptm}
In this section, we show that our results may be used to approximate the diffusion coefficient for random paired tent maps depending on a parameter $\varepsilon$. This class of random dynamical systems was introduced by Horan in \cite{horan_lin,Horan} and has been investigated as a non-autonomous system that admits metastability \cite{GTQ_quarantine,gtp_met}. For this reason, we benchmark our results against such a model. \\
\\
We consider a similar setup to \cite[Section 8]{gtp_met}. Let $I=[-1,1]$ and for $0< a,b\leq 1$ consider the paired tent map $T_{a,b}:I\to I$ given by
\begin{equation}
    T_{a,b}(x) := \begin{cases} 2(1+b)(x+1)-1, \ \ \ \ \ &x\in [-1,-1/2]\\
    -2(1+b)x-1, &x\in[-1/2,0)\\
    0, & x=0\\
    -2(1+a)x+1, &x\in(0,1/2]\\
    2(1+a)(x-1)+1, &x\in[1/2,1]
    \end{cases}.
    \label{eqn:ptm}
\end{equation}
This map satisfies $T_{a,b}(-1)=-1$, $T_{a,b}(-1/2)=b$ and $\lim_{x\to 0^-}T(x)=-1$; $\lim_{x\to 0^+}T_{a,b}(x)=1$, $T_{a,b}(1/2)=-a$ and $T_{a,b}(1)=1$. The map $T_{0,0}$ comprises of tent maps on disjoint subintervals $I_L=[-1,0]$ and $I_R=[0,1]$. For small positive $a$ and $b$, there is a small amount of leakage between these subintervals: points near $-1/2$ are mapped to $I_R$ whilst points near $1/2$ are mapped to $I_L$. This behaviour may be seen in \fig{fig:ptm}.
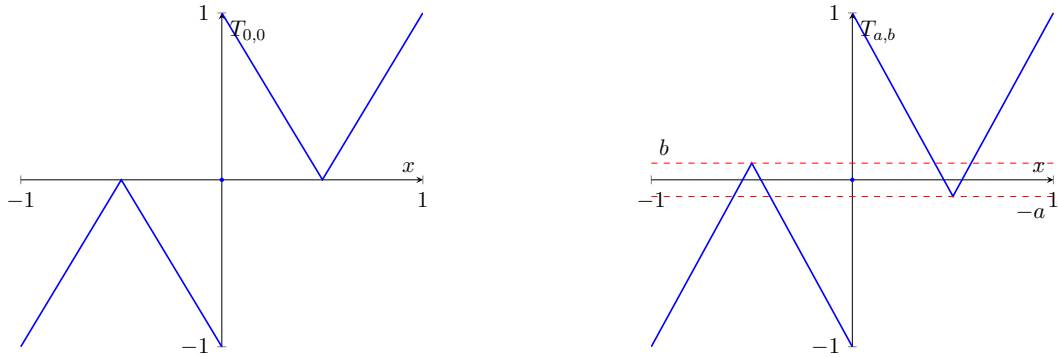
\begin{figure}[H]

\begin{subfigure}{0.45\linewidth}
\begin{tikzpicture}[
  declare function={
    func(\x)= (\x <-1/2) * (2*\x+1)   +
              and(\x >= -1/2, \x < 0) * (-2*\x-1)     +
              and(\x > 0, \x <= 1/2) * (-2*\x+1)   +
              (\x > 1/2) * (2*\x-1);
  }
,scale=0.78]
\begin{axis}[
  samples = 7, axis x line=middle, axis y line=middle,
  ymin=-1, ymax=1, ytick={-1,...,1}, ylabel=$T_{0,0}$,
  xmin=-1, xmax=1, xtick={-1,...,1}, xlabel=$x$
]
\addplot [blue, domain=-1:0, thick] {func(x)}; 
\addplot [blue, domain=0.001:1, thick] {func(x)}; 
\addplot[blue, only marks, mark=*, mark size=0.8] coordinates {(0,0)};
\end{axis}
\end{tikzpicture}
    \label{fig:ptma}
    \end{subfigure}\hfill
\begin{subfigure}{0.45\linewidth}
\begin{tikzpicture}[
  declare function={
    func(\x)= (\x <-1/2) * (2.2*\x+1.2)   +
              and(\x >= -1/2, \x < 0) * (-2.2*\x-1)     +
              and(\x > 0, \x <= 1/2) * (-2.2*\x+1)   +
              (\x > 1/2) * (2.2*\x-1.2);
  }
,scale=0.78]
\begin{axis}[
  samples = 7, axis x line=middle, axis y line=middle,
  ymin=-1, ymax=1, ytick={-1,...,1}, ylabel=$T_{a,b}$,
  xmin=-1, xmax=1, xtick={-1,...,1}, xlabel=$x$
]
\addplot [blue, domain=-1:0, thick] {func(x)}; 
\addplot [blue, domain=0.001:1, thick] {func(x)}; 
\addplot[blue, only marks, mark=*, mark size=0.8] coordinates {(0,0)};
\addplot [mark=none, red, samples=2,dashed] {0.1};
\node[above right] at (axis cs: -1, 0.1) {$b$};

\addplot [mark=none, red, samples=2,dashed] {-0.1};
\node[below left] at (axis cs: 1, -0.1) {$-a$};
\end{axis}
\end{tikzpicture}
    \label{fig:ptmb}
    \end{subfigure}
    \caption{Paired tent map $T_{a,b}$ on $I=[-1,1]$ with $a=b=0$ (left) and $a=b=\frac{1}{10}$ (right)}
    \label{fig:ptm}
\end{figure}
\noindent
Now consider a cocycle of paired tent maps driven by an ergodic, measure-preserving and invertible transformation $\sigma:\Omega \to \Omega$. The $\omega$-dependence is introduced in this system by making the leakage between intervals $I_L$ and $I_R$ random. This is guaranteed by considering the evolution of a point $x\in I$ under the dynamics $T_\omega := T_{a_\omega,b_\omega}$ where \Change{for $\beta^*>0$, $a,b:\Omega \to [\beta^*,1]$ are measurable functions}, and thus $a,b\in L^\infty(\mathbb{P})$. Since we are interested in the behaviour of such a system when small amounts of leakage can occur, we consider for sufficiently small $\varepsilon>0$ the map $T_\omega^\varepsilon := T_{\varepsilon a_\omega,\varepsilon b_\omega}$.\footnote{One can interpret the parameter $\varepsilon$ as a leakage controller. The larger $\varepsilon$ is, the more likely it is for points to switch between $I_L$ and $I_R$.} Here, there is a low probability that a point will be mapped from $I_L$ to $I_R$ in one step and vice-versa. When $\varepsilon = 0$, $T^0 := T_{0,0}$ and the map consists of tent maps on disjoint subintervals. Furthermore, the dynamics is autonomous with the $\varepsilon=0$ perturbation, removing the randomness ($\omega$-dependence) of the system. 
\\
\\
The random maps $(T_\omega^\varepsilon)_{\omega\in\Omega} := (T_{\varepsilon a_\omega, \varepsilon b_\omega})_{\omega\in\Omega}$, driven by $\sigma$, are the primary focus of this section. We verify the conditions of  \thrm{thrm:variance-est} to provide an approximation of the diffusion coefficient for random paired tent maps. These include \hyperref[list:I1]{\textbf{(I1)}}-\hyperref[list:I6]{\textbf{(I6)}}, \hyperref[list:P1]{\textbf{(P1)}}-\hyperref[list:P7]{\textbf{(P7)}} and \hyperref[list:O1]{\textbf{(O1)}}.   \\
\\
The main result of this section is the following. 
\begin{theorem}
 \Change{Let $\beta^*>0$ and suppose that $a,b :\Omega \to [\beta^*,1]$ are measurable functions.} Let $\{(\Omega,\mathcal{F},\mathbb{P},\sigma,\BV(I),\mathcal{L}^\varepsilon)\}_{\varepsilon\geq 0}$ be a sequence of random dynamical systems of paired tent maps $(T^\varepsilon_\omega)_{\omega\in\Omega}:= (T_{\varepsilon a_\omega, \varepsilon b_\omega})_{\omega\in\Omega}$ satisfying \hyperref[list:P1]{\textbf{(P1)}}. Fix $\varepsilon>0$ and take an observable $\psi:\Omega\times I \to \mathbb{R}$ satisfying \eqref{eqn:obs-reg} and \eqref{eqn:obs-centering2}. Assume that the variance defined in \eqref{eqn:var} satisfies $(\Sigma^\varepsilon(\tilde{\psi}^\varepsilon))^2>0$ where $\tilde{\psi}^\varepsilon:\Omega\times I \to \mathbb{R}$ is as in \eqref{eqn:obs-centering1}. Then, for every bounded and continuous function $f:\mathbb{R}\to \mathbb{R}$ and $\mathbb{P}$-a.e. $\omega\in\Omega$, we have 
    \begin{equation*}
        \lim_{n\to \infty} \int_{\mathbb{R}}f\left( \frac{1}{\sqrt{n}}\sum_{k=0}^{n-1}(\tilde{\psi}^\varepsilon_{\sigma^{k}\omega}\circ T_{\omega}^{\varepsilon\, (k)})\right)(x)\, d\mu_{\omega}^\varepsilon(x) = \int_\mathbb{R}f(x)\, d\mathcal{N}(0,(\Sigma^\varepsilon(\tilde{\psi}^\varepsilon))^2)(x).
    \end{equation*}
    Furthermore,
    \begin{equation}
        \lim_{\varepsilon\to 0} \varepsilon(\Sigma^\varepsilon(\tilde{\psi}^\varepsilon))^2 =  \frac{2 \int_\Omega b_\omega \, d\mathbb{P}(\omega)}{\int_\Omega a_\omega \, d\mathbb{P}(\omega)\int_\Omega a_\omega+b_\omega\, d\mathbb{P}(\omega)}\Bigg(\int_\Omega\int_{I_R}\psi_\omega(x)\, \dleb(x)d\mathbb{P}(\omega) \Bigg)^2.\label{eqn:ptm-limit}
    \end{equation}
    \label{thrm:ptm-var}
\end{theorem}
\begin{remark}
    We will find that due to \eqref{eqn:obs-centering2}, one can deduce from \eqref{eqn:ptm-limit} that
    \begin{align*}
            \lim_{\varepsilon\to 0} \varepsilon(\Sigma^\varepsilon(\tilde{\psi}^\varepsilon))^2&=\frac{2\int_\Omega a_\omega\, d\mathbb{P}(\omega)}{\int_\Omega b_\omega\, d\mathbb{P}(\omega)\int_\Omega a_\omega+b_\omega\, d\mathbb{P}(\omega)}\left(\int_\Omega \int_{I_L}\psi_\omega(x)\, \dleb(x)d\mathbb{P}(\omega)\right)^2\\
    &=-\frac{2}{\int_\Omega a_\omega+b_\omega\, d\mathbb{P}(\omega)}\int_\Omega \int_{I_L}\psi_\omega(x)\, \dleb(x)d\mathbb{P}(\omega)\int_\Omega \int_{I_R}\psi_\omega(x)\, \dleb(x)d\mathbb{P}(\omega).
    \end{align*}
\end{remark}
Due to \cite[Section 8]{gtp_met}, by enforcing \hyperref[list:P1]{\textbf{(P1)}}, the sequence of random dynamical systems of random paired tent maps $(T^\varepsilon_\omega)_{\omega\in\Omega}:= (T_{\varepsilon a_\omega, \varepsilon b_\omega})_{\omega\in\Omega}$ satisfies \hyperref[list:I1]{\textbf{(I1)}}-\hyperref[list:I6]{\textbf{(I6)}}, and \hyperref[list:P1]{\textbf{(P1)}}-\hyperref[list:P7]{\textbf{(P7)}}.
\begin{remark}
    Note that our condition \hyperref[list:P4]{\textbf{(P4)}} is stronger than that asked in \cite{gtp_met}. In our setting, this is required to ensure the sequence of random paired tent maps, $(T_\omega^\varepsilon)_{\omega\in\Omega}$, is uniformly (over $\omega\in\Omega$) covering.\footnote{See \Step{step:clt-lemma} in the proof of \thrm{thrm:variance-est}.} For this reason, in \thrm{thrm:ptm-var}, \Change{we let $\beta^*>0$ and assume that $a,b:\Omega\to [\beta^*,1]$ are measurable functions, ensuring \hyperref[list:P4]{\textbf{(P4)}} holds.}
\end{remark}
\Change{It remains to verify \hyperref[list:O1]{\textbf{(O1)}.} Observe that \eqref{eqn:enough-expansion} is satisfied with $n'=4$ since one can take $\Lambda= 2$. Further, \eqref{eqn:large-survival} follows from the fact that for every $\varepsilon>0$ and $j\in\{L,R\}$, the maps $T^{\varepsilon}_{j,\omega}:I_j\to I_j$ have all full branches, and $\omega\mapsto T_{j,\omega}^\varepsilon$ has finite range (by \hyperref[list:P1]{\textbf{(P1)}}). We may now prove the main result for this section.} 
\begin{proof}[Proof of \thrm{thrm:ptm-var}]
It remains to compute the relevant quantities appearing in \eqref{eqn:var-limit}. Let $\bar{a}:= \int_\Omega a_\omega\, d\mathbb{P}(\omega),\ \bar{b}:= \int_\Omega b_\omega\, d\mathbb{P}(\omega)$, and for $j\in \{L,R\}$ let $\bar{\psi}_j:=\int_\Omega \int_{I_j}\psi_\omega(x)\, \dleb(x)d\mathbb{P}(\omega)$. Thanks to \cite[Theorem 8.1]{gtp_met}, 
\begin{equation}
    p=\begin{pmatrix}
    \frac{\bar{a}}{\bar{a} + \bar{b}}\\ \frac{\bar{b}}{\bar{a}+\bar{b}} 
\end{pmatrix} \label{eqn:ptm-inv-meas}
\end{equation}
and
\begin{equation}
    \int_\Omega\Psi_\omega\,d\mathbb{P}(\omega)=\begin{pmatrix}
     \bar{\psi}_L \\  \bar{\psi}_R  
\end{pmatrix}.\label{eqn:ptm-psi}
\end{equation}
Further, due to \cite[Lemma 8.6]{gtp_met}, the averaged Markov jump process from \Sec{sec:avg_markov} is generated by 
\begin{equation}
    \bar{G}=\begin{pmatrix}
    -{\bar{b}} & {\bar{b}}\\ {\bar{a}} & -{\bar{a}}
\end{pmatrix}.\label{eqn:ptm-gen}
\end{equation}
Note that due to \eqref{eqn:obs-centering2}, we require that for $\mathbb{P}$-a.e. $\omega\in\Omega$, 
\begin{align}
    \frac{\bar{a}}{\bar{a} + \bar{b}}\int_{I_L}\psi_\omega(x)\, \dleb(x) + \frac{\bar{b}}{\bar{a} + \bar{b}}\int_{I_R}\psi_\omega(x)\, \dleb(x)=0.\label{eqn:ptm-centering}
\end{align}
Therefore, substituting \eqref{eqn:ptm-inv-meas}, \eqref{eqn:ptm-psi} and \eqref{eqn:ptm-gen} into \eqref{eqn:var-limit}, we find that for random paired tent maps
\begin{align*}
    p\odot\int_\Omega \Psi_\omega\, d\mathbb{P}(\omega)= \frac{1}{\bar{a} + \bar{b}}\begin{pmatrix}
     {\bar{a}}{}\bar{\psi}_L \\  {\bar{b}}\bar{\psi}_R
\end{pmatrix}
\end{align*}
and
\begin{align*}
    \int_0^\infty e^{t\bar{G}}\,dt\int_\Omega\Psi_\omega\, d\mathbb{P}(\omega)&=\frac{1}{\bar{a} + \bar{b}}\int_0^\infty\left( \begin{pmatrix} {{\bar{a}}} & {{\bar{b}}}\\
{{\bar{a}}} & {{\bar{b}}} \end{pmatrix} - \bar{G} e^{-t(\bar{a} + \bar{b})}\right)\begin{pmatrix}
     \bar{\psi}_L  \\ \bar{\psi}_R   
\end{pmatrix}\, dt \\
&\stackrel{(\star)}{=}-\frac{1}{\bar{a} + \bar{b}}\int_0^\infty \bar{G} e^{-t(\bar{a} + \bar{b})}\begin{pmatrix}
    \bar{\psi}_L  \\  \bar{\psi}_R  
\end{pmatrix} \, dt \\
&=\frac{1}{(\bar{a} + \bar{b})^2}\left(\begin{matrix}
     \bar{b}\Big(\bar{\psi}_L -\bar{\psi}_R\Big) 
     \\
     \bar{a}\Big(\bar{\psi}_R-\bar{\psi}_L\Big) 
     \end{matrix}\right).
\end{align*}
Note that at $(\star)$ we have applied \eqref{eqn:ptm-centering}. Therefore, by \thrm{thrm:variance-est} and \eqref{eqn:ptm-centering},
\begin{align*}
    \lim_{\varepsilon\to 0} \varepsilon(\Sigma^\varepsilon(\tilde{\psi}^\varepsilon))^2&=\frac{2\bar{a}\bar{b}}{(\bar{a} + \bar{b})^3}\left( \bar{\psi}_L-\bar{\psi}_R\right)^2=\frac{2\bar{b}}{\bar{a}(\bar{a} + \bar{b})}\bar{\psi}_R^2.\\
\end{align*}
\end{proof}

\section*{Conflict of interest}
The authors declare that there is no conflict of interest regarding the publication of this paper.

\section*{Acknowledgments}
The authors acknowledge support from the Australian Research Council. JP acknowledges the Australian Government Research Training Program for financial support. The authors thank Jason Atnip and Wael Bahsoun for helpful discussions and Dmitry Dolgopyat for suggesting this topic.

\footnotesize
\bibliographystyle{plain}
\bibliography{diffusion}

\end{document}